\documentclass[12pt]{amsart}
\usepackage{color}
\usepackage{epsfig}
\usepackage{graphicx}
\usepackage{hyperref}
\usepackage{amssymb}
\usepackage{amsfonts}
\usepackage{euscript}
\usepackage[all]{xy}
\usepackage{setspace}
\usepackage{amsbsy}

\numberwithin{equation}{section}
\renewcommand{\subsection}[1]{\hspace{-\parindent}\refstepcounter{subsection}{\bf
(\arabic{section}\alph{subsection}) #1.}\addcontentsline{toc}{subsection}
{\bf #1.}}

\newtheorem{thm}{Theorem}[section]
\newtheorem{theorem}[thm]{Theorem}

\newtheorem{corollary}[thm]{Corollary}

\newtheorem{definition}[thm]{Definition}

\newtheorem{remark}[thm]{Remark}

\newtheorem{proposition}[thm]{Proposition}
\newtheorem{example}[thm]{Example}

\newtheorem{lemma}[thm]{Lemma}
\newtheorem{setup}[thm]{Setup}

\newtheorem{conjecture}[thm]{Conjecture}
\newtheorem*{claim*}{Claim}
\newtheorem*{lemma*}{Lemma}

\newcommand{\bC}{{\mathbb C}}

\newcommand{\bF}{{\mathbb F}}

\newcommand{\bK}{{\mathbb K}}

\newcommand{\bQ}{{\mathbb Q}}
\newcommand{\bR}{{\mathbb R}}

\newcommand{\bZ}{{\mathbb Z}}

\newcommand{\scrC}{\EuScript C}

\newcommand{\scrE}{\EuScript E}

\newcommand{\scrH}{\EuScript H}

\newcommand{\scrJ}{\EuScript J}

\newcommand{\scrL}{\EuScript L}

\newcommand{\scrO}{\EuScript O}
\newcommand{\scrP}{\EuScript P}

\newcommand{\scrR}{\EuScript R}

\newcommand{\iso}{\cong}
\newcommand{\htp}{\simeq}
\newcommand{\smooth}{C^\infty}

\newcommand{\lbr}{[\![}
\newcommand{\rbr}{]\!]}

\title{Disjoinable Lagrangian spheres and dilations}
\author{Paul Seidel}
\date{September 1, 2013}

\begin{document}
\maketitle

\begin{abstract}
We consider open symplectic manifolds which admit dilations (in the sense previously introduced by Solomon and the author). We obtain restrictions on collections of Lagrangian submanifolds which are pairwise disjoint (or pairwise disjoinable by Hamiltonian isotopies) inside such manifolds. This includes the Milnor fibres of isolated hypersurface singularities which have been stabilized (by adding quadratic terms) sufficiently often.
\end{abstract}

\section{Introduction}

\subsection{Motivation}
What restrictions are there on collections of pairwise disjoint Lagrangian submanifolds inside a given symplectic manifold? This is a reasonably natural question, or rather class of questions, in symplectic topology. One can view it as an attempt to bound the ``size'' or ``complexity'' of the symplectic manifold under consideration, in analogy with the fact that a closed oriented surface of genus $g>1$ can contain at most $3g-3$ pairwise disjoint simple closed curves in a nontrivial sense (nontrivial means that no curve is contractible, and no two are isotopic). There is additional motivation for specializing to Lagrangian spheres, because of their relation with nodal degenerations in algebraic geometry; and from there possibly further to the three-dimensional case, where collections of disjoint Lagrangian three-spheres provide the starting point for the surgery construction from \cite{smith-thomas-yau02}, itself modelled on small resolutions of nodal singularities. 

By way of example, let's consider hypersurfaces in projective space, of a fixed degree and dimension. Starting from a hypersurface with nodal singularities, one obtains a collection of pairwise disjoint Lagrangian spheres in the corresponding smooth hypersurface, as vanishing cycles (as an aside, note that since one can choose a hyperplane which misses all the original nodes, the Lagrangian spheres obtained in this way actually all lie inside an affine smooth hypersurface). There are numerous constructions of hypersurfaces with many nodes in the literature, see for instance \cite[vol.\ 2, p.\ 419]{arnold-gusein-zade-varchenko} or \cite[Section 8.1]{labs06}. On the other hand, the size of collections obtained in this way is limited by known upper bounds for the possible number of nodes, such as those derived in \cite{varchenko83} using Hodge-theoretic methods. It is an interesting open question how these bounds compare to the situation for general collections of Lagrangian spheres. There are reasons to be cautious about proposing a direct relationship. One is the possible existence of nodal degenerations where the special fibre is not itself a hypersurface, even though a smooth fibre is (this is certainly an issue in the toy case of algebraic curves: the maximal number of nodes on a plane curve of degree $d$, including reducible curves, is $d^2/2 - d/2$ by the Pl{\"u}cker formula, and that is less than $3g-3 = 3d^2/2 - 9d/2$ for $d \geq 5$). From a wider perspective, there are other constructions of Lagrangian spheres (as components of real loci, or more speculatively in terms of SYZ fibrations) which are not a priori related to vanishing cycles.


\subsection{First observations}
Let's return to the general question as formulated at the outset. For any oriented closed totally real submanifold $L^n$ in an almost complex manifold $M^{2n}$, the selfintersection equals the Euler characteristic up to a dimension-dependent sign:
\begin{equation} \label{eq:euler}
[L] \cdot [L] = (-1)^{n(n+1)/2} \chi(L).
\end{equation}
This holds mod $2$ without the orientation assumption (there is also a mod $4$ version in the unoriented case, involving the Pontryagin square \cite{massey69b}; but we won't use it).

\begin{theorem}[folklore] \label{th:folk}
Fix some field $\bK$. Consider closed totally real submanifolds $L \subset M$ such that $\chi(L)$ is not a multiple of the characteristic $\mathrm{char}(\bK)$. If $\mathrm{char}(\bK) \neq 2$, we also assume that our $L$ are oriented. Let $(L_1,\dots,L_r)$ be a collection of such submanifolds, which are pairwise disjoint (or disjoinable by homotopies). Then the classes $[L_1],\dots,[L_r] \in H_n(M;\bK)$ are linearly independent. 
\end{theorem}

That follows directly from \eqref{eq:euler}, since any relation $a_1[L_1] + \cdots + a_r[L_r] = 0$ implies that
\begin{equation}
0 = [L_i] \cdot (a_1[L_1] + \cdots + a_r[L_r]) = (-1)^{n(n+1)/2} a_i \,\chi(L_i) \in \bK,
\end{equation}
and $\chi(L_i)$ is nonzero as an element of $\bK$ by assumption. The limitations of such topological methods are clear: $\bC^3$ contains an embedded totally real three-sphere, hence also (by translating it) an infinite number of pairwise disjoint such spheres.
%
%
%

Within symplectic topology properly speaking, one can group the existing literature into two approaches. The first one is quite general in principle, but hard to carry out in practice: it considers Lagrangian submanifolds as objects of the Fukaya category with suitable properties, and proceeds via a classification of such objects up to isomorphism. The second approach is a priori limited to manifolds with semisimple quantum cohomology rings. One could view the two as related, since in many examples the Fukaya categories of manifolds with semisimple quantum cohomology are themselves semisimple, hence lend themselves easily to a classification of objects. However, the point of the second approach is precisely that one can bypass an explicit study of the Fukaya category, and instead just use general properties of open-closed string maps. We will now survey both approaches briefly.

\subsection{Classification results for the Fukaya category} Possibly the earliest example of this strategy is \cite[Theorem 1(4)]{seidel04c} about cotangent bundles of odd-dimensional spheres (the resulting statement is in fact a special case of an earlier theorem \cite[Th\'eor\`eme 2]{lalonde-sikorav91}; but the earlier proof lies outside the scope of our discussion, because it uses a geometric trick specific to homogeneous spaces). Since then, the state of knowledge about cotangent bundles has developed considerably \cite{nadler06, fukaya-seidel-smith07, abouzaid11c}. Maybe more importantly for our purpose, there is at least
%
%
one case other than cotangent bundles which has been successfully analyzed in this way, namely the Milnor fibres of type $(A_m)$ singularities \cite{abouzaid-smith11, seidel12}: 

\begin{theorem}[\protect{\cite[Theorems 1.1 and 1.2]{seidel12}}] \label{th:equivariant}
Let $M^{2n}$ be the Milnor fibre of the $(A_m)$ singularity, with $n \geq 3$ odd and arbitrary $m$. Consider Lagrangian submanifolds $L \subset M$ which are rational homology spheres and {\em Spin}. Then:

(i) $[L] \in H_n(M;\bZ)$ is nonzero and primitive. 

(ii) Two such submanifolds which are disjoint must have different mod $2$ homology classes.
\end{theorem}

The necessary algebraic classification of objects in the Fukaya category of $M$ was carried out by hand in \cite{abouzaid-smith11} for $m = 2$, whereas \cite{seidel12} relied on algebro-geometric results from \cite{ishii-uehara05,ishii-ueda-uehara10}. As an illustration of the difficulty of obtaining such a classification, note that so far we only have partial analogues of the results from \cite{ishii-uehara05} for the remaining simple singularities, of types $(D_m)$ and $(E_m)$ \cite{brav-thomas11}.

\subsection{Manifolds with semisimple quantum cohomology} 
Results that fall into this class can be found in \cite{entov-polterovich09, biran-cornea09c} (a significant precursor is \cite{albers05,albers10}). Here is a sample:

\begin{theorem}[\protect{A version of \cite[Theorem 1.25]{entov-polterovich09}}] \label{th:semisimple}
Let $M$ be a closed monotone symplectic manifold, whose ($\bZ/2$-graded) quantum cohomology $\mathit{QH}^*(M) = H^*(M;\bK)$, defined over some algebraically closed field $\bK$, is semisimple. Consider monotone Lagrangian submanifolds $L \subset M$ which are oriented and {\em Spin} (if $\mathrm{char}(\bK) = 2$, one can drop the {\em Spin} assumption). Let $(L_1,\dots,L_r)$ be a collection of such submanifolds, which are pairwise disjoint or disjoinable by Hamiltonian isotopies. 

(i) Suppose that $\mathit{HF}^*(L_i,L_i) \neq 0$ for all $i$. Then $r \leq \mathrm{dim}\, \mathit{QH}^0(M)$, where the right hand side is the sum of the even Betti numbers (with $\bK$-coefficients).

(ii) In the same situation, suppose that all the $[L_i] \in H_n(M;\bK)$ are nonzero. Then they must be linearly independent.
\end{theorem}

A short outline of the argument may be appropriate. First, semisimplicity of $\mathit{QH}^*(M)$ means that its even part splits as
\begin{equation} \label{eq:semisimplicity}
\mathit{QH}^0(M) = \bigoplus_{i \in I} \bK u_i,
\end{equation}
where $(u_i)_{i \in I}$ is a collection of pairwise orthogonal idempotents. Because the intersection pairing is nondegenerate, it must be nontrivial on each summand \eqref{eq:semisimplicity}. Fix some $L$, and consider the open-closed string map
\begin{equation} \label{eq:oc-map}
\mathit{HF}^*(L,L) \longrightarrow \mathit{QH}^{*+n}(M).
\end{equation}
Because of its compatibility with the structure of $\mathit{HF}^*(L,L)$ as a module over $\mathit{QH}^*(M)$, the image of \eqref{eq:oc-map} in even degrees consists of a subset of summands in \eqref{eq:semisimplicity}. Moreover, if $\mathit{HF}^*(L,L)$ is nonzero, the composition of \eqref{eq:oc-map} with $\int_M: \mathit{QH}^0(M) \rightarrow \bK$ is nontrivial, hence the previously mentioned subset is nonempty. Finally, if $\mathit{HF}^*(L_0,L_1)$ is well-defined and vanishes, then the images of the open-closed string maps for $L_0$ and $L_1$ must be mutually orthogonal with respect to the intersection pairing (by a form of the Cardy relation, see for instance \cite[Proposition 5.3 and Figure 2]{seidel12}). This implies (i) provided that all Lagrangian submanifolds involved have minimal Maslov number $> 2$, so that $\mathit{HF}^*(L_0,L_1)$ is always well-defined. To remove that additional assumption, one decomposes $\mathit{QH}^0(M)$ into eigenspaces of quantum multiplication with $c_1(M)$, and considers the Lagrangian submanifolds with any given Maslov index $2$ disc count separately, using \cite[Lemma 6.7]{auroux07}. 

It is instructive to compare the argument so far with Theorem \ref{th:folk}. What we have done is to replace the intersection pairing on $H_n(M)$ with that on each summand in \eqref{eq:semisimplicity} considered separately. The shift to even degrees ensures that the argument can be effective even if $n$ is odd, but at the same time prevents us from proving that the $[L_i]$ are nonzero (which is indeed false, even for $M = S^2$). Instead, for part (ii) of Theorem \ref{th:semisimple} one argues as follows. The class $[L]$ is the image of the unit element in $\mathit{HF}^0(L,L)$ under \eqref{eq:oc-map}. In particular, if $[L]$ is nonzero, then so is $\mathit{HF}^*(L,L)$. Now, for a collection $(L_1,\dots,L_r)$ as in the statement of the theorem, we get a decomposition of $\mathit{QH}^*(M)$ into direct summands, and each $[L_i]$ must lie in a different summand, which precludes having any nontrivial relations (in fact, not just relations over $\bK$, but ones with coefficients in $\mathit{QH}^0(M)$ as well).

\subsection{New results} 
We now turn to the actual substance of this paper. 

\begin{theorem} \label{th:1}
Let $M^{2n}$, $n > 1$ odd, be a (finite type complete) Liouville manifold. This should satisfy $c_1(M) = 0$, and we choose a trivialization of the anticanonical bundle $K_M^{-1}$. Assume that its ($\bZ$-graded) symplectic cohomology $\mathit{SH}^*(M)$, with coefficients in some field $\bK$, contains a dilation. Then there is a constant $N$ such that the following holds. Consider closed Lagrangian submanifolds in $M$ which are $\bK$-homology spheres and {\em Spin}. Suppose that $(L_1,\dots,L_r)$ is a collection of such submanifolds, which are pairwise disjoint (or disjoinable by Lagrangian isotopies). Then $r \leq N$.
\end{theorem}

The notion of dilation comes from \cite{seidel-solomon10}. The {\em Spin} assumption on Lagrangian submanifolds arises as usual from sign considerations in Floer theory, and one can drop it in $\mathrm{char}(\bK) = 2$. That would make no difference in the context of Theorem \ref{th:1}, since a homology sphere over a field of characteristic $2$ must be {\em Spin} anway. However, one can do a little better by exploiting fortuitous cancellations. Recall that the ($\bF_2$-coefficient) Kervaire semi-characteristic of a closed manifold of odd dimension $n$ is \cite{kervaire56,lusztig-milnor-peterson69} 
\begin{equation} \label{eq:kervaire}
\chi_{1/2}(L) = \sum_{i=0}^{(n-1)/2} \mathrm{dim}\, H^i(L;\bF_2) \in \bF_2.
\end{equation}
Semi-characteristics have appeared before in the context of totally real embeddings \cite{audin89}, but that has apparently nothing to do with our result, which is the following:

\begin{theorem} \label{th:2}
Let $M$ be as in Theorem \ref{th:1}, and suppose that $\mathrm{char}(\bK) = 2$. Replace the topological assumptions on closed Lagrangian submanifolds $L \subset M$ with the following weaker ones:
\begin{equation} \label{eq:semi}
H^1(L;\bK) = 0\;\; \text{ and }\;\; 
\chi_{1/2}(L) = 1.
\end{equation}
Then the same conclusion will hold.
\end{theorem}

%

Making the bound $N$ explicit depends on an understanding of the geometry of the dilation. In particularly simple cases, one may be able to make a connection with the ordinary topology of $M$. Here is an instance where that is possible:
%

\begin{theorem} \label{th:3}
Take an affine algebraic hypersurface $\{p(z_1,\dots,z_{n+1}) = 0\} \subset \bC^{n+1}$, with $n$ odd, which has an isolated singular point at the origin. Suppose that the Hessian of the defining polynomial at the singular point satisfies
\begin{equation} \label{eq:quadric-add}
\mathrm{rank}(D^2p_{z = 0}) \geq 3.
\end{equation}
Let $M$ be the Milnor fibre of that singularity. Consider Lagrangian submanifolds $L \subset M$ which are $\bQ$-homology spheres and {\em Spin}. If $(L_1,\dots,L_r)$ is a collection of such submanifolds which are pairwise disjoint (or disjoinable by Lagrangian isotopies), then the classes $[L_i] \in H_n(M;\bQ)$ are linearly independent.
\end{theorem}

\begin{theorem} \label{th:4}
Take $M$ as in Theorem \ref{th:3}. Fix a field $\bK$ of odd positive characteristic. Consider Lagrangian submanifolds $L \subset M$ which are $\bK$-homology spheres and {\em Spin}. If $(L_1,\dots,L_r)$ is a collection of such submanifolds which are pairwise disjoint (or disjoinable by Lagrangian isotopies), then the classes $[L_i] \in H_n(M;\bK)$ span a subspace of dimension $\geq r/2$ (in particular, by setting $r = 1$ one sees that each $[L_i]$ must be nonero).
\end{theorem}

\begin{theorem} \label{th:5}
Take $M$ as in Theorem \ref{th:3}, but sharpening \eqref{eq:quadric-add} to $\mathrm{rank}(D^2p_{z = 0}) \geq 4$. Take a field $\bK$ of characteristic $2$. If we consider Lagrangian submanifolds $L \subset M$ which satisfy \eqref{eq:semi}, then the same conclusion as in Theorem \ref{th:4} holds.
\end{theorem}
%

Obviously, \eqref{eq:quadric-add} holds if $p$ is triply stabilized, which means that
\begin{equation} \label{eq:3-stabilize}
p(z) = z_1^2 + z_2^2 + z_3^2 + \tilde{p}(z_4,\dots,z_{n+1}); 
\end{equation}
and conversely, any polynomial satisfying \eqref{eq:quadric-add} can be brought into the form \eqref{eq:3-stabilize} by a local holomorphic coordinate change \cite[Vol.~I, Section 11.1]{arnold-gusein-zade-varchenko}. 

\begin{example}
As a concrete example, take the $(A_m)$ singularity of odd dimension $n > 1$. In that case, we nearly recover Theorem \ref{th:equivariant}(i) (the missing piece would be an extension of Theorem \ref{th:5} to the lowest dimension $n = 3$, to show that the homology classes are nonzero mod $2$). 
If $(L_1,\dots,L_r)$ is a collection of Lagrangian $\bQ$-homology spheres which are {\em Spin} and pairwise disjoinable, then Theorem \ref{th:3}, together with the fact that the subspace of $H_n(M;\bQ)$ spanned by $([L_1],\dots,[L_r])$ must be isotropic for the intersection pairing, yields the bound
\begin{equation}
r \leq \left[\frac{m+1}{2}\right].
\end{equation}
This is sharp, and much better than what one would get from applying Theorem \ref{th:equivariant}(ii) (which on the other hand has no counterpart among our results). 
\end{example}

In spite of these discrepancies between our results and those of \cite{seidel12} (which could probably be narrowed by investing more work on both sides), there is a fundamental similarity between our notion of dilation (giving rise to an infinitesimal symmetry of the Fukaya category) and the $\bC^*$-actions used in that paper. We refer interested readers to \cite[Lectures 13--19]{seidel13b}, where both viewpoints are considered. 

Concerning the comparison with Theorem \ref{th:semisimple} and other results of that nature, we were unable to find a substantial relation between them and our approach. There is some common philosophical ground, in that both approaches are based on replacing the intersection pairing in middle-dimensional homology with another one, which has different symmetry properties; and that for this replacement to work, strong restrictions on the class of symplectic manifolds under consideration have to be imposed. However, in Theorem \ref{th:semisimple} the replacement is essentially the even-dimensional cohomology, whereas in our case we remain in the middle dimension but replace cohomology with a different space, built from symplectic cohomology. 

Finally, note that in all our theorems, we have required the (complex) dimension $n$ to be odd. In fact, the proof of Theorem \ref{th:1} works for all $n>1$; and there is an analogue of Theorem \ref{th:2} for even $n$, which involves half of the ordinary Euler characteristic instead of the semi-characteristic (in the proof, part (ii) of Corollary \ref{th:new-main-properties} would be used instead of (iii), with the rest of the argument remaining the same). Similar remarks apply to Theorems \ref{th:3}--\ref{th:5}. However, the resulting statements are not stronger than what one can get in an elementary way, which means from Theorem \ref{th:folk}.

\subsection{Contents}
The structure of this paper is as follows. Section \ref{sec:dilations} is aimed at readers interested in the strength and applicability of our results. For that purpose, the main aim is to understand what dilations are. As one important example, this includes the construction of dilations on the Milnor fibres appearing in Theorems \ref{th:3}--\ref{th:5}. Very little of this material is new, but we have tweaked the presentation from \cite{seidel-solomon10} slightly to make it more convenient for our purpose.

Section \ref{sec:strategy} is aimed at readers primarily interested in seeing the overall structure of the argument. We introduce a number of additional algebraic structures associated to Hamiltonian Floer cohomology groups, and state their properties without proof. Those properties, when combined with the ones introduced in \cite{seidel-solomon10}, lead directly to the theorems stated above. 

In Section \ref{sec:operations} we flesh out the argument, which largely means specifying the families of Riemann surfaces which give rise to the various (cochain level) operations underlying our constructions. Section \ref{sec:technical} provides selected technical details. 

The last part, Section \ref{sec:motivation}, placed there so as not to interrupt the main expository thread, mentions some parallel constructions in other parts of mathematics. This is not strictly necessary for our argument, but can provide additional motivation.

\subsection{Acknowledgments} This work benefited from conversations that the author had with Mohammed Abouzaid, Ailsa Keating, and Ivan Smith. Major expository changes were made following referees' reports on the initial version of the manuscript. Partial support was provided by NSF grant DMS-1005288, and by a Simons Investigator Award from the Simons Foundation.

\section{Background\label{sec:dilations}}

\subsection{Hamiltonian Floer cohomology}
Our geometric setup for Floer cohomology is almost identical to that in \cite[Section 3]{seidel-solomon10}, but we reproduce it to make the discussion more self-contained. 

\begin{setup} \label{th:setup}
Let $(M,\omega_M)$ be a non-compact symplectic manifold, together with an exhausting (proper and bounded below) function $H_M \in \smooth(M,\bR)$. Let $X_M$ be the Hamiltonian vector field of $H_M$. We write $\scrP_M \subset \bR$ for the set of those $\lambda$ such that the $1$-periodic orbits of $\lambda X_M$ are not contained in a compact subset of $M$. One always has $0 \in \scrP_M$; and $\lambda \in \scrP_M$ iff $-\lambda \in \scrP_M$. We will assume throughout that $H_M$ is such that $\scrP_M$ has measure zero (so that in particular, $\bR \setminus \scrP_M$ is unbounded); note that this implies that all critical points of $H_M$ must be contained in a compact subset.

The other assumptions are of a more technical kind, and needed in order to get the Floer-theoretic machinery off the ground in the desired form. First of all we assume that $M$ is exact, meaning that $\omega_M = d\theta_M$ for some fixed $\theta_M$. We also want it to satisfy $c_1(M) = 0$, and fix a trivialization (up to homotopy) of the anticanonical bundle $K_M^{-1} = \Lambda_{\bC}^{\mathit{top}}(TM)$. 

Finally, we need some property that prevents solutions of Floer-type equations from escaping to infinity. This holds with respect to some fixed compatible almost complex structure $I_M$. Namely, for every compact subset $K \subset M$ there should be another such subset $\tilde{K} \subset M$, such that the following holds. Suppose that $S$ is a connected compact Riemann surface with nonempty boundary, $\nu_S$ a real one-form on it such that $d\nu_S \leq 0$, and $u: S \rightarrow M$ a solution of
\begin{equation} \label{eq:simple-dbar}
(du - X_M \otimes \nu_S)^{0,1} = 0,
\end{equation}
where the $(0,1)$-part is taken with respect to $I_M$. Then, what we want is that
\begin{equation} \label{eq:convexity}
u(\partial S) \subset K \;\;\Longrightarrow\;\;
u(S) \subset \tilde{K}.
\end{equation}
\end{setup}

\begin{example} \label{th:liouville}
The most important class of examples are (finite type complete) Liouville manifolds. These are $(M,\omega_M,\theta_M,H_M)$ where the Liouville vector field $Z_M$ dual to $\theta_M$ satisfies
\begin{equation}
Z_M.H_M = H_M \;\; \text{outside a compact subset.}
\end{equation}
This implies that a sufficiently large level set $N = H_M^{-1}(c)$, $c \gg 0$, is a closed contact type hypersurface. Moreover, the part of $M$ lying outside that hypersurface, which is $H_M^{-1}([c,\infty))$, can be identified with the positive half of the symplectization of $N$. With respect to this identification, $X_M$ is $c$ times the Reeb vector field for $\theta_M|N$. Hence, $c\scrP_M \cap \bR^{>0}$ is the set of periods of Reeb orbits (including multiples). 

One uses an almost complex structure $I_M$ such that $dH_M \circ I_M = -\theta_M$ outside a compact subset, and derives the required property \eqref{eq:convexity} from a maximum principle argument. Of course, the condition $c_1(M) = 0$ still has to be imposed separately.
\end{example}

The first step is to define Floer cohomology groups $\mathit{HF}^*(\lambda)$ for each $\lambda \in \bR \setminus \scrP_M$. To do that, choose a time-dependent function $H_\lambda \in \smooth(S^1 \times M,\bR)$, $S^1 = \bR/\bZ$, with associated time-dependent vector field $X_\lambda$. This should satisfy $H_{\lambda,t} = \lambda H_M$ outside a compact subset, and we also require all solutions of
\begin{equation} \label{eq:periodic-orbit}
\left\{\begin{aligned} & x: S^1 \longrightarrow M, \\ &
dx/dt = X_{\lambda,t}(x) \end{aligned} \right.
\end{equation}
to be nondegenerate. 
Choose a time-dependent compatible almost complex structure $J_{\lambda}$ such that $J_{\lambda,t} = I_M$ outside a compact subset of $M$, and for which all solutions of Floer's equation are nondegenerate. Fix an arbitrary coefficient field $\bK$. We then define $\mathit{CF}^*(\lambda) = \mathit{CF}^*(H_\lambda,J_\lambda)$ to be the associated Floer complex, which is a finite-dimensional $\bZ$-graded complex of $\bK$-vector spaces, with generators (at least up to sign) corresponding bijectively to solutions of \eqref{eq:periodic-orbit}. The differential will be denoted by $d$. Its cohomology $\mathit{HF}^*(\lambda)$ is independent of $(H_\lambda,J_\lambda)$ up to canonical isomorphism. Note that its Euler characteristic is not interesting:
\begin{equation}
\chi(\mathit{HF}^*(\lambda)) = \chi(M) \;\; \text{for all $\lambda \in \bR \setminus \scrP_M$.}
\end{equation}

Even though Floer cohomology is independent of all choices, we find it convenient to co-ordinate those choices in a particular way, namely to require that 
\begin{equation} \label{eq:swap-sign}
(H_{-\lambda,t},J_{-\lambda,t}) = (-H_{\lambda,-t},J_{\lambda,-t}). 
\end{equation}
for all $\lambda$. Then, there is a canonical nondegenerate pairing
\begin{equation} \label{eq:chain-duality}
\langle \cdot,\cdot \rangle\; : \; \mathit{CF}^*(\lambda) \otimes \mathit{CF}^{2n-*}(-\lambda) \longrightarrow \bK.
\end{equation}
We use the same notation for the induced cohomology level pairing, which gives rise to a Poincar{\'e} duality type isomorphism
\begin{equation} \label{eq:poincare-duality-hf}
\mathit{HF}^*(-\lambda) \iso \mathit{HF}^{2n-*}(\lambda)^\vee.
\end{equation}

The next observation is that Floer cohomology groups come with continuation maps \cite{salamon-zehnder92}
\begin{equation} \label{eq:continuation-map}
\mathit{HF}^*(\lambda_1) \longrightarrow \mathit{HF}^*(\lambda_0), \;\; \lambda_0 \geq \lambda_1.
\end{equation}
To define these, one takes the cylinder $S = \bR \times S^1$ with coordinates $(s,t)$, and equips it with a one-form $\nu_S$ such that $\nu_S = \lambda_0 \, \mathit{dt}$ for $s \ll 0$, $\nu_S = \lambda_1\,\mathit{dt}$ for $s \gg 0$, and $d\nu_S \leq 0$ everywhere. One then chooses a perturbation datum, namely a section $K_S$ of the pullback bundle $T^*S \rightarrow S \times M$ (or equivalently a one-form on $S$ with values in the space of functions on $M$), such that $K_S = H_{\lambda_0,t} \, \mathit{dt}$ for $s \ll 0$, $K_S = H_{\lambda_1,t}\,\mathit{dt}$ for $s \gg 0$, and 
\begin{equation} \label{eq:k-infty}
K_S = H_M \nu_S \quad \text{ on } S \times \{\text{the complement of some compact subset of $M$}\}. 
\end{equation}
Similarly, let $J_S$ be a family of compatible almost complex structures on $M$ parametrized by $(s,t) \in S$, such that $J_{S,s,t} = J_{\lambda_0,t}$ for $s \ll 0$, $J_{S,s,t} = J_{\lambda_1,t}$ for $s \gg 0$, and 
\begin{equation} \label{eq:j-infty}
J_{S,s,t} = I_M \quad \text{ on } S \times \{\text{the complement of some compact subset of $M$}\}. 
\end{equation}
Assuming that the choices have been made generically, the chain map underlying \eqref{eq:continuation-map} is defined by counting solutions of the associated equation
\begin{equation} \label{eq:cont-equation}
\left\{
\begin{aligned}
& u: S \longrightarrow M, \\
& (du - Y_S)^{0,1} = 0, \\
& \textstyle\lim_{s \rightarrow -\infty} u(s,\cdot) = x_0, \\
& \textstyle\lim_{s \rightarrow +\infty} u(s,\cdot) = x_1, \\
\end{aligned}
\right.
\end{equation}
where $Y_S$ is the one-form on $S$ with values in Hamiltonian vector fields on $M$, derived from $K_S$, and the $(0,1)$-part is formed with respect to $J_S$. Outside a compact subset of the target space $M$, the Cauchy-Riemann equation in \eqref{eq:cont-equation} reduces to \eqref{eq:simple-dbar}. Moreover, all possible limits $x_0,x_1$ are contained in a compact subset, by the definition of $\scrP_M$. These two facts together with \eqref{eq:convexity} imply that all solutions $u$ are contained in a compact subset. 

The maps \eqref{eq:continuation-map} are independent of all choices, and well-behaved with respect to composition. This makes $\mathit{HF}^*(\lambda)$ into a directed system, and one can define
\begin{equation} \label{eq:infinite-slope}
\mathit{HF}^*(\infty) \stackrel{\mathrm{def}}{=} \underrightarrow{\lim}_\lambda \;
\mathit{HF}^*(\lambda).
\end{equation}

For the next step, take a partial compactification of the cylinder, $S = (\bR \times S^1) \cup \{s = +\infty\}$ (which is a Riemann surface isomorphic to the complex plane $\bC$). All previous conditions for $s \gg 0$ should now be replaced by ones asking that the relevant data extends smoothly over $s = +\infty$. Note that the number $\lambda = \lambda_0$ associated to the remaining end $s \ll 0$ must then necessarily satisfy $\lambda>0$ ($\lambda \geq 0$ because of $d\nu_S \leq 0$, and $\lambda \neq 0$ since we always ask that $\lambda \notin \scrP_M$). The analogue of \eqref{eq:cont-equation} yields a cochain in $\mathit{CF}^0(\lambda)$, whose cohomology class
\begin{equation} \label{eq:1-element}
1 \in \mathit{HF}^0(\lambda), \;\; \lambda>0,
\end{equation}
is independent of all choices, and preserved under continuation maps. One can further generalize this construction by using the evaluation map at the point $s = +\infty$. By asking that $u(+\infty)$ should go through the stable manifold of an exhausting Morse function on $M$, one can construct a chain map from the associated Morse complex to the Floer complex, hence a map
\begin{equation} \label{eq:morse-to-floer}
H^*(M;\bK) \longrightarrow \mathit{HF}^*(\lambda), \;\; \lambda > 0.
\end{equation}
This is again canonical and compatible with continuation maps. Moreover, the previously defined \eqref{eq:1-element} is simply the image of $1 \in H^0(M;\bK)$ under \eqref{eq:morse-to-floer}. Note that dually in terms of \eqref{eq:poincare-duality-hf}, we have maps into compactly supported cohomology
\begin{equation} \label{eq:dual-morse-to-floer}
\mathit{HF}^*(-\lambda) \longrightarrow H^*_{\mathit{cpt}}(M;\bK) \iso H_{2n-*}(M;\bK), \;\; \lambda > 0.
\end{equation}
The composition 
\begin{equation}
\mathit{HF}^*(-\lambda) \longrightarrow H^*_{\mathit{cpt}}(M;\bK) \longrightarrow H^*(M;\bK) \longrightarrow \mathit{HF}^*(\lambda)
\end{equation}
can again be identified with the relevant continuation map, by a gluing argument which (like the construction of the maps \eqref{eq:morse-to-floer}, \eqref{eq:dual-morse-to-floer} in itself) is modelled on those in \cite{piunikhin-salamon-schwarz94}.

\begin{example} \label{th:symplectic-cohomology}
In the situation from Example \ref{th:liouville}, the limit \eqref{eq:infinite-slope} is the symplectic cohomology $\mathit{SH}^*(M)$ in the sense of \cite{viterbo97a}, which is a symplectic invariant (see also \cite{cieliebak-floer-hofer95} for a closely related construction). In this situation, if $\mu$ is the length of the shortest periodic Reeb orbit, so that $(0,\mu) \cap \scrP_M = \emptyset$, one finds that \eqref{eq:morse-to-floer} is an isomorphism for $\lambda \in (0,\mu)$. On the other hand, passing to the limit yields a map
\begin{equation}
H^*(M;\bK) \longrightarrow \mathit{SH}^*(M),
\end{equation}
which played a crucial role in early applications to the Weinstein conjecture \cite{viterbo97a,viterbo97b}.
\end{example}

The next piece of structure we need is the BV (Batalin-Vilkovisky) operator
\begin{equation} \label{eq:bv-operator-1}
\Delta: \mathit{HF}^*(\lambda) \longrightarrow \mathit{HF}^{*-1}(\lambda)
\end{equation}
(see \cite[Section 8]{seidel07}, \cite[Section 3]{seidel-solomon10}, or \cite[Definition 2.11]{bourgeois-oancea12}). This squares to zero; commutes with continuation maps (in particular, induces an operation on \eqref{eq:infinite-slope}, for which we use the same notation); and vanishes on the image of \eqref{eq:morse-to-floer}. The definition uses a family of equations of type \eqref{eq:cont-equation} depending on an additional parameter $r \in S^1$. The Riemann surface $S_r = \bR \times S^1$ is the same for all $r$, but it comes with a family of one-forms $\nu_{S_r}$, inhomogeneous terms $K_{S_r}$, and almost complex structures $J_{S_r}$. The precise requirement is that
\begin{equation}
\left.
\begin{aligned}
& \nu_{S_r} = \lambda\, \mathit{dt} \\
& K_{S_r} = H_{\lambda,t} \mathit{dt} \\
& J_{S_r,s,t} = J_{\lambda,t}
\end{aligned}
\right\} \text{ for $s \ll 0$,}
\qquad
\left.
\begin{aligned}
& \nu_{S_r} = \lambda\, \mathit{dt} \\
& K_{S_r} = H_{\lambda,t-r} \mathit{dt} \\
& J_{S_r,s,t} = J_{\lambda,t-r}
\end{aligned}
\right\} \text{ for $s \gg 0$.}
\end{equation}
As one sees from this, it is possible (but not really necessary) to choose $\nu_{S_r}$ to be the same for all $r$. On the other hand, it is usually impossible to choose the same $K_{S_r}$ and $J_{S_r}$ for all $r$, because that would require $H_{\lambda}$ and $J_{\lambda}$ to be constant in $t$, which is incompatible with the transversality requirement for \eqref{eq:periodic-orbit}. Hence, the parametrized moduli space of pairs $(r,u)$, where $r \in S^1$ and $u: S \rightarrow M$ is a solution of the appropriate equation \eqref{eq:cont-equation}, can have a nontrivial zero-dimensional part, which one uses to construct the cochain level map underlying \eqref{eq:bv-operator-1}.

\subsection{Dilations}
We have now collected all the ingredients (Floer cohomology, continuation maps, the element \eqref{eq:1-element}, the BV operator) that enter into the following:

\begin{definition} \label{th:dilation}
A dilation is a class $B \in \mathit{HF}^1(\lambda)$, for some $\lambda > 0$, whose image $\tilde{B} \in \mathit{HF}^*(\tilde\lambda)$ under the continuation map, for some $\tilde\lambda \geq \lambda$, satisfies $\Delta \tilde{B} = 1$.
\end{definition}

We will also allow $\tilde{\lambda} = \infty$ or $\lambda = \infty$; in particular in the context of Example \ref{th:liouville}, where for $\lambda = \infty$ we would be talking about an element of $\mathit{SH}^1(M)$ (this is the formulation used in Theorem \ref{th:1}). Note that by definition \eqref{eq:infinite-slope}, any element of $\mathit{HF}^*(\infty)$ comes from $\mathit{HF}^*(\lambda)$ for some $\lambda>0$, and the same applies to relations between elements, such as the equation $\Delta B = 1$. Hence, saying that a solution of $\Delta B = 1$ exists in $\mathit{HF}^*(\infty)$ is equivalent to saying that Definition \ref{th:dilation} is satisfied for some finite values of $\lambda$, $\tilde{\lambda}$. However, for the purpose of the present paper the quantitative aspect, which means trying to get $\lambda$ (and, less importantly, $\tilde\lambda$) to be as small as possible, is also important.

\begin{example} \label{th:cotangent-sphere}
Take $M = T^*S^n$, with the standard forms $\omega_M = d\theta_M$, and a function $H_M$ such that $H_M(x) = \|x\|$ is the length (in the standard round metric) outside a compact subset. This means that, at infinity, $X_M$ is the normalized geodesic flow. The isomorphism \cite{viterbo97b,salamon-weber03,abbondandolo-schwarz06}
\begin{equation} \label{eq:string-topology}
\mathit{HF}^*(\infty) = \mathit{SH}^*(M) \iso H_{n-*}(\scrL S^n;\bK)
\end{equation}
with the homology of the free loop space, can be used to show the following \cite[Example 6.1 and Example 6.4]{seidel-solomon10}:
\begin{itemize}
\itemsep1em
\item $T^*S^1$ does not admit a dilation for any choice of coefficient field $\bK$;
\item $T^*S^2$ admits a dilation iff $\mathrm{char}(\bK) \neq 2$;
\item $T^*S^n$, $n>2$, admits a dilation for all choices of coefficient field $\bK$.
\end{itemize}

Let's focus on the situation where dilations exist. For the general reasons mentioned in Example \ref{th:symplectic-cohomology}, the map \eqref{eq:morse-to-floer} is an isomorphism if $0<\lambda<2\pi$. On the other hand, a Conley-Zehnder index computation shows that in nonnegative degrees, the maps
\begin{equation}
\mathit{HF}^*(\lambda) \longrightarrow \mathit{HF}^*(\infty), \;\; \ast \geq 0,
\end{equation}
are isomorphisms as soon as $\lambda>2\pi$. Hence, one can choose any $\tilde{\lambda} = \lambda > 2\pi$ in Definition \ref{th:dilation}, which means that the dilation appears at the earliest theoretically possible stage.
\end{example}

We refer to \cite[Example 6.4]{seidel-solomon10} as well as Example \ref{th:lens-space} below for further discussion of cotangent bundles. A very basic way to provide other examples is this:

\begin{example}
Take $M_0$ and $M_1$ be Liouville manifolds, with functions as in Example \ref{th:liouville}. If $M_0$ admits a dilation, then so does $M_0 \times M_1$. This follows from the K{\"u}nneth formula \cite{oancea04}.
\end{example}

On the other hand, a look at some classes of Liouville manifolds arising from algebraic geometry indicates that the existence of dilations is a very restrictive condition.

\begin{example} \label{th:hypersurface}
Take a smooth hypersurface $\bar{M} \subset \bC P^{n+1}$ of degree $d \geq 3$, and remove its intersection with a generic hyperplane. The result is an affine hypersurface $M \subset \bC^{n+1}$ which, when equipped with the restriction of the standard symplectic form, is unique (depends only on $d$) up to symplectic isomorphism. It is Liouville and comes with a trivialization of its anticanonical bundle, hence belongs to the class of manifolds from Example \ref{th:liouville} (in the case $d = 2$, which we have excluded, this would lead to the previously studied example $M \iso T^*S^n$). Take $\bK$ to be of characteristic $0$. It then turns out that $M$ never admits a dilation. The case $n = 1$ is of an elementary topological nature \cite[Example 6.1]{seidel-solomon10}, and we won't discuss it further, but we will give a sketch of the proof in the higher-dimensional situation.

Start by taking $d = 3$ and $n \geq 4$. The contact hypersurface $N$ describing the structure of $M$ at infinity is a circle bundle over the cubic hypersurface in $\bC P^n$, hence (by weak Lefschetz) satisfies
\begin{equation} \label{eq:weak-lefschetz}
H^q(N;\bK) = 0 \;\; \text{for } q \neq 0,n-1,n,2n-1.
\end{equation}
There is a Morse-Bott spectral sequence converging to $\mathit{SH}^*(M)$, with starting page (similar to \cite[Equation (1)]{seidel02}, which would be the analogue for $d = n+2$)
\begin{equation}
E_1^{pq} = \begin{cases} H^q(M;\bK) & p = 0, \\
H^{(5-2n)p+q}(N;\bK) & p<0, \\ 
%
%
0 & p>0.
\end{cases}
\end{equation}
In view of \eqref{eq:weak-lefschetz} and the dimension assumption, it follows that $\mathit{SH}^1(M) = 0$, hence $M$ cannot admit a dilation (strictly speaking, this part of the argument also requires knowing that the symplectic cohomology of $M$ is not identically zero, but that it easy to show; for instance, by exhibiting a Lagrangian sphere inside $M$). 

Now consider the affine cubic threefold ($d = 3$ and $n = 3$). This is clearly a hyperplane section of the affine cubic fourfold. Hence, there is a Lefschetz fibration which has the threefold as a fibre and the fourfold as total space. By \cite[Proposition 7.3]{seidel-solomon10} (see also Section \ref{subsec:lef} below for related results), if the fibre admits a dilation, so must the total space. We apply this argument in reverse to show that the affine cubic threefold does not admit a dilation.
One can iterate the same argument to reach the corresponding conclusion for the affine cubic surface.

Finally, a degeneration argument shows that the affine hypersurface of degree $d$ can be symplectically embedded into that of degree $d+1$. These embeddings are automatically exact if $n>1$. An application of Viterbo functoriality now shows that since cubic hypersurfaces do not admit dilations, neither do those of higher degree. 
\end{example}

\begin{example} \label{th:surface-milnor-fibre}
Take a affine algebraic surface $\{p(z_1,z_2,z_3) = 0\} \subset \bC^3$ with an isolated singularity at the origin (the dimensional restriction is crucial, compare Example \ref{th:stabilized-milnor-fibre} below). From that singularity, one obtains a symplectic four-manifold, its Milnor fibre, which (after attaching a semi-infinite cone to the boundary) belongs to the class of Liouville manifolds from Example \ref{th:liouville}. It turns out that if the singularity is not one of the simple ones (not of ADE type), its Milnor fibre does not admit a dilation (for any choice of coefficient field $\bK$). This can be proved by combining a construction of Lagrangian tori from \cite{keating14} with results from singularity theory, as follows.

\cite{keating14} first considers the simple-elliptic singularities of type $\tilde{E}_6$, $\tilde{E}_7$, $\tilde{E}_8$ ($P_8$, $X_9$ and $J_{10}$ in Arnol'd's notation \cite{arnold-gusein-zade-varchenko}), and shows that each of their Milnor fibres contains an exact Lagrangian torus. This implies the desired result, since the existence of a dilation rules out having closed exact Lagrangian submanifolds which are Eilenberg-MacLane spaces \cite[Corollary 6.3]{seidel-solomon10}. Any singularity which is not simple is adjacent to one of the three we have considered (see e.g.\ \cite[Proposition 10.1 and Table 3]{durfee79}). Adjacence comes with a symplectic embedding of Milnor fibres \cite[Lemma 9.9]{keating12}, and one then argues as in Example \ref{th:hypersurface}. The remaining case of simple (ADE type) singularities is open at present (except for $(A_1)$, where the Milnor fibre is $T^*S^2$).
\end{example}

\subsection{Lefschetz fibrations\label{subsec:lef}}
We want to discuss one more way to construct dilations. This largely follows \cite[Section 7]{seidel-solomon10}, but we pay a little more attention to the quantitative aspect, which means to the choice of Hamiltonian functions. 

Let $(F,\omega_F,\theta_F,H_F,I_F)$ be as in Setup \ref{th:setup}. Suppose that $F$ is the fibre of an exact symplectic Lefschetz fibration
\begin{equation} \label{eq:lefschetz-fibration}
\pi: M \longrightarrow \bC.
\end{equation}
The definition of such a fibration is as in \cite[Definition 7.1]{seidel-solomon10} (except that we are a little less restrictive concerning the kinds of fibres that are allowed; in \cite{seidel-solomon10} the fibres were required to be Liouville manifolds).
Take the given function $H_F$ on the fibre. Only its behaviour at infinity really matters, so we can assume without loss of generality that $H_F$ vanishes on a large compact subset. Since the Lefschetz fibration is trivial at infinity in fibrewise direction, there is then a preferred way to extend $H_F$ to a function on $M$. Define
\begin{equation} \label{eq:epsilon-function}
H_M = H_F + \epsilon H_\bC \in \smooth(M,\bR),
\end{equation}
where: $H_F$ stands for the previously mentioned extension; $H_\bC(z) = |z-b|^2/2$ is a function on the base, pulled back to $M$ (where $b \in \bC$ is some base point, assumed to be close to infinity, so that the fibration is locally trivial near $\pi^{-1}(b) \iso F$); and $\epsilon>0$ is a positive constant. Similarly, given the almost complex structure $I_F$ on the fibre, one can construct (in a non-unique way) a compatible almost complex structure $I_M$ on the total space, which makes $\pi$ pseudo-holomorphic. One can show that $M$, equipped with these data and with its given exact symplectic form $\omega_M = d\theta_M$, again satisfies the conditions from Setup \ref{th:setup}. To distinguish notationally between fibre and total space, we denote the respective Floer cohomology groups by $\mathit{HF}^*(F,\lambda)$ and $\mathit{HF}^*(M,\lambda)$, and similarly for the underlying chain complexes.

\begin{lemma} \label{th:lefschetz-1}
Suppose that $2n = \mathrm{dim}(F) \geq 4$. Fix some $\mu>0$. Then, provided that the constant $\epsilon$ is chosen sufficiently small, we have
\begin{equation} \label{eq:fibre-versus-total-space}
\mathit{HF}^*(M,\lambda) \iso \mathit{HF}^*(F,\lambda), \quad \ast \leq 1,
\end{equation}
for all $|\lambda| < \mu$. More precisely, if $\lambda$ is in that range and the right hand side of \eqref{eq:fibre-versus-total-space} is defined, then so is the left hand side, and the isomorphism holds.
\end{lemma}

This is a simplified version of \cite[Lemma 7.2]{seidel-solomon10}. Briefly, the Hamiltonian vector field of \eqref{eq:epsilon-function} satisfies
\begin{equation}
D\pi(X_M) = \psi \epsilon X_{\bC}
\end{equation}
where: $X_{\bC} = i(z-b)\partial_z$; and $\psi \in \smooth(M,\bR^{\geq 0})$ is a function which vanishes precisely at the critical points of $\pi$, and equals $1$ outside a compact subset, as well as in a neighbourhood of $\pi^{-1}(b)$. By taking $\epsilon$ small, one ensures that all one-periodic orbits of $X_\lambda$ are either contained in $\pi^{-1}(b)$, or else constant orbits located at the critical points. After a suitable perturbation to achieve transversality, one finds that as graded vector spaces,
\begin{equation}
\mathit{CF}^*(M,\lambda) \iso \mathit{CF}^*(F,\lambda)\; \oplus\; \textstyle \bigoplus_{x \in \mathit{Crit}(\pi)} \bK[-1-n].
\end{equation}
The first summand comes from $1$-periodic orbits lying in $\pi^{-1}(b) \iso F$, and the second one from the critical points of $\pi$, whose Conley-Zehnder index is $n+1 \geq 3$. The main remaining point is an energy computation \cite[Equation (7.6)]{seidel-solomon10}, which shows that Floer trajectories with both limits lying in $\pi^{-1}(b)$ must entirely be contained in that fibre. A variation of the same argument shows:

\begin{lemma} \label{th:lefschetz-2}
In the situation of Lemma \ref{th:lefschetz-1}, the isomorphisms \eqref{eq:fibre-versus-total-space} are compatible with the BV operator, and with continuation maps. Moreover, for $0 < \lambda < \mu$ the following diagram commutes:
\begin{equation}
\begin{matrix}
\xymatrix{
\mathit{HF}^*(M,\lambda) \ar[r]^-{\iso} & \mathit{HF}^*(F,\lambda) \\
H^*(M;\bK) \ar[u] \ar[r]^-{\iso} & H^*(F;\bK) \ar[u]
} 
\end{matrix}
\qquad \ast \leq 1,
\end{equation}
where the $\uparrow$ maps are induced by \eqref{eq:morse-to-floer}, and the bottom $\rightarrow$ is the ordinary restriction map.
\end{lemma}


For later use, it is convenient to formalize one of the applications of these ideas.

\begin{definition}
Let $M$ be a manifold with an exact symplectic structure $\omega_M = d\theta_M$ and a trivialization of $K_M^{-1}$. We say that $M$ has property (H) if there is a compatible almost complex structure $I_M$, a function $H_M$, and a $\mu>0$, such that the following holds:

(i) The conditions of Setup \ref{th:setup} are satisfied;

(ii) $\mu,2\mu \notin \scrP_M$;

(iii) the map $H^*(M;\bK) \rightarrow \mathit{HF}^*(\mu)$ is an isomorphism;

(iv) there is an element $B \in \mathit{HF}^1(2\mu)$ satisfying $\Delta B = 1 \in \mathit{HF}^0(2\mu)$.
\end{definition}

As a consequence of Lemmas \ref{th:lefschetz-1} and \ref{th:lefschetz-2}, together with a version of \cite[Lemma 7.2]{seidel-solomon10}, one then has:

\begin{corollary} \label{th:property-h}
Given an exact symplectic Lefschetz fibration, if the fibre has property (H), then so does the total space.
\end{corollary}

\begin{example} \label{th:stabilized-milnor-fibre}
Take a hypersurface 
\begin{equation} \label{eq:hypersurface}
\{p(z_1,\dots,z_{n+1}) = 0\} \subset \bC^{n+1}
\end{equation}
with an isolated singular point at the origin. Let $H = \{h_1z_1 + \cdots + h_{n+1}z_{n+1} = 0\} \subset \bC^{n+1}$ be a generic hyperplane through the origin. Then, one can make the Milnor fibre of the original singularity into the total space of an exact symplectic Lefschetz fibration, whose fibre is the Milnor fibre of 
\begin{equation} \label{eq:hyperplane-section}
\{p(z_1,\dots,z_{n+1}) = 0\} \cap H \subset H. 
\end{equation}
Very roughly speaking, the map which makes up the Lefschetz fibration is constructed from the linear function $z \mapsto h_1z_1 + \cdots h_{n+1}z_{n+1}$.

Because of the genericity assumption, if the rank of the Hessian $(D^2p)_{z=0}$ is less than $n+1$, then its restriction to $H$ will have the same rank. In particular, if $\mathrm{rank}((D^2p)_{z=0}) = n$, then \eqref{eq:hyperplane-section} has a nondegenerate singularity at the origin, so its Milnor fibre is $T^*S^{n-1}$. By iterating this idea and applying Corollary \ref{th:property-h} as well as Example \ref{th:cotangent-sphere}, one obtains the following conclusion: the Milnor fibre of \eqref{eq:hypersurface} has property (H) provided that
\begin{equation}
\mathrm{rank}((D^2p)_{z=0}) \geq \begin{cases} 3 & \mathrm{char}(\bK) \neq 2, \\
4 & \mathrm{char}(\bK) = 2.
\end{cases}
\end{equation}
\end{example}

\subsection{Lagrangian submanifolds}
We return to background material in Floer theory, still within the general framework of Setup \ref{th:setup}. Throughout, we will consider Lagrangian submanifolds of the following kind:

\begin{setup} \label{th:setup-lagrangian}
$L \subset M$ is assumed to be closed, connected, exact, and graded (which implies that it is oriented). If $\mathrm{char}(\bK) \neq 2$, we also assume that $L$ is {\em Spin}. 
\end{setup}

\begin{remark} \label{th:universal-coefficient-theorem}
If $H^1(L;\bK) = 0$ for some coefficient field $\bK$, then also $H^1(L;\bZ) = 0$, hence $H^1(L;\bR) = 0$. This is easiest to see in converse direction: if $H^1(L;\bR) = \mathit{Hom}(H_1(L;\bZ),\bR)$ is nonzero, $H_1(L;\bZ)$ has a nontrivial free summand, hence $H^1(L;\bK) = \mathit{Hom}(H_1(L;\bZ);\bK)$ is nonzero for any $\bK$. Hence, a Lagrangian submanifold with $H^1(L;\bK) = 0$ automatically satisfies the exactness condition from Setup \ref{th:setup-lagrangian}, and also admits a grading.
\end{remark}

Given a pair $(L_0,L_1)$ of such submanifolds, there is a well-defined Floer cohomology group $\mathit{HF}^*(L_0,L_1)$, which is a finite-dimensional graded $\bK$-vector space. To define it, one chooses some $\lambda_{L_0,L_1} \in \bR$. Take a time-dependent Hamiltonian $H_{L_0,L_1}$, where the time parameter is now $t \in [0,1]$, and such that $H_{L_0,L_1,t} = \lambda_{L_0,L_1} H_M$ outside a compact subset. One additionally requires that all chords
\begin{equation} \label{eq:periodic-orbit-2}
\left\{\begin{aligned} & x: [0,1] \longrightarrow M, \\ & x(0) \in L_0, \;\; x(1) \in L_1, \\ &
dx/dt = X_{L_0,L_1,t}(x) \end{aligned} \right.
\end{equation}
should be nondegenerate. Correspondingly, take a time-dependent almost complex structure $J_{L_0,L_1}$ so that $J_{L_0,L_1,t} = J_M$ outside a compact subset, and which makes all solutions of Floer's perturbed pseudo-holomorphic disc equation (with boundary values on $L_0$, $L_1$) nondegenerate. One then obtains a Floer complex $\mathit{CF}^*(L_0,L_1) = \mathit{CF}^*(H_{L_0,L_1},J_{L_0,L_1})$, whose differential we denote by $\mu^1_{L_0,L_1}$ following standard practice for Fukaya categories. Its cohomology $\mathit{HF}^*(L_0,L_1)$ is independent of all choices up to canonical isomorphism. Its Euler characteristic recovers the intersection number:
\begin{equation} \label{eq:hf-euler}
\chi(\mathit{HF}^*(L_0,L_1)) = (-1)^{n(n+1)/2} [L_0] \cdot [L_1].
\end{equation}

In the special case $L_0 = L_1 = L$, there is a canonical isomorphism
\begin{equation} \label{eq:hf-h}
\mathit{HF}^*(L,L) \iso H^*(L;\bK).
\end{equation}
There is also an analogue of \eqref{eq:poincare-duality-hf}, namely
\begin{equation} \label{eq:poincare-duality-lag}
\mathit{HF}^*(L_1,L_0) \iso \mathit{HF}^{n-*}(L_0,L_1)^\vee.
\end{equation}

\begin{remark} \label{th:fail}
If $L_0 \neq L_1$, one can coordinate choices to realize \eqref{eq:poincare-duality-lag} on the chain level, as in \eqref{eq:chain-duality}. However, that becomes much harder for $L_0 = L_1$ (in view of \eqref{eq:hf-h}, this relates to the issue of implementing a strict chain level Poincar{\'e} duality in ordinary topology, with the additional constraint of having to do so within Floer theory). This will cause a few complications later on (see Section \ref{sec:operations}).
\end{remark}

Any Lagrangian submanifold (within the class of Setup \ref{th:setup-lagrangian}, as always) gives rise to an element
\begin{equation} \label{eq:improved-class}
\lbr L \rbr \in \mathit{HF}^n(\lambda), \;\; \text{for any } \lambda \in \bR \setminus \scrP_M.
\end{equation}
To define this, one considers the Riemann surface with boundary $S = (-\infty,0] \times S^1$, with coordinates $(s,t)$. Equip $S$ with a one-form $\nu_S$ such that $d\nu_S \leq 0$ everywhere, and $\nu_S = \lambda \mathit{dt}$ on the region where $s \ll 0$. Choose an inhomogeneous term $K_S$ as before, with $K_{S,s,t} = H_{\lambda,t} \, \mathit{dt}$ for $s \ll 0$, and which satisfies the following boundary condition:
\begin{equation} \label{eq:boundary-condition-k}
\text{The $\mathit{dt}$ component of $K_{S,0,t}$ vanishes on $L$.} 
\end{equation}
Additionally, choose a family of almost complex structures $J_S$. One then considers an equation of type \eqref{eq:cont-equation} for maps $u: S \rightarrow M$, but with boundary values $u(\partial S) \subset L$. Counting isolated solutions of this equation yields the cochain representative for \eqref{eq:improved-class}. 

We want to describe some of its properties. First of all, for a fixed $L$ but varying $\lambda$, the classes \eqref{eq:improved-class} are mapped to each other by continuation maps. Moreover, if $\lambda>0$, then $\lbr L \rbr$ is simply the image of the ordinary homology class $[L]$ under the map $H^n_{\mathit{cpt}}(M;\bK) \rightarrow H^n(M;\bK) \rightarrow \mathit{HF}^n(\lambda)$. On the other hand, for $\lambda<0$, $\lbr L \rbr$ maps to $[L]$ under the map $\mathit{HF}^n(\lambda) \rightarrow H^n_{\mathit{cpt}}(M;\bK)$, which means that it is a refinement of the ordinary homology class, in general containing additional information. Finally, these classes are always annihilated by the BV operator:
\begin{equation} \label{eq:delta-kills-the-map}
\Delta \lbr L \rbr = 0.
\end{equation}
Roughly speaking, \eqref{eq:delta-kills-the-map} is proved as follows. By a gluing argument, $\Delta \lbr L \rbr$ can be defined directly using a family $(S_r)$, $r \in S^1$, of Riemann surfaces. These surfaces are all isomorphic to our previous $S = (-\infty,0] \times S^1$, and the parametrization of the auxiliary data must be such that $K_{S_r,s,t} = H_{\lambda,t+r}$, $J_{S_r,s,t} = J_{\lambda,t+r}$ for $s \ll 0$. Within this class, one can achieve that the isomorphism $S_0 \iso S_r$ given by rotation with angle $2\pi r$ is compatible with all the data. Then, the parametrized moduli space associated to the family $(S_r)$ decomposes as a product with a circle, hence has no isolated points, which implies the desired vanishing result.

In fact, one can generalize the construction of \eqref{eq:improved-class} to yield a pair of mutually dual (after switching the sign of $\lambda$) open-closed string maps
\begin{align} \label{eq:open-closed-string-map}
& \mathit{HF}^*(\lambda) \longrightarrow \mathit{HF}^*(L,L), \\
& \mathit{HF}^*(L,L) \longrightarrow \mathit{HF}^{*+n}(\lambda). \label{eq:dual-open-closed-string-map}
\end{align}
If one thinks of $H^*(L;\bK)$ instead of $\mathit{HF}^*(L,L)$, then the construction can be carried out in a way similar to \eqref{eq:morse-to-floer} and \eqref{eq:dual-morse-to-floer}, involving an auxiliary Morse function on $L$. The image of the unit element in $\mathit{HF}^0(L,L) \iso H^0(L;\bK)$ under \eqref{eq:dual-open-closed-string-map} recovers \eqref{eq:improved-class}.

\subsection{Product structures}
Finally, and without any geometric details, we want to review the product structures on Floer cohomology, starting with the Hamiltonian version (these go back to \cite{schwarz95,piunikhin-salamon-schwarz94} for closed symplectic manifolds, with details for Liouville manifolds given in \cite{ritter10}). First, one has the pair-of-pants product
\begin{equation} \label{eq:product}
\mathit{HF}^*(\lambda_2) \otimes \mathit{HF}^*(\lambda_1) \longrightarrow
\mathit{HF}^*(\lambda_1+\lambda_2).
\end{equation}
This is associative and graded commutative. It is also symmetric with respect to the pairing induced by \eqref{eq:chain-duality}, in the sense that
\begin{equation} \label{eq:cyclic-symmetry-1}
\begin{aligned}
& \langle x_3, x_2 x_1 \rangle = (-1)^{|x_3|} \langle x_2, x_1 x_3 \rangle = (-1)^{|x_1|} \langle x_1 x_3, x_2 \rangle \\
& \text{for } x_1 \in \mathit{HF}^*(\lambda_1), \; x_2 \in \mathit{HF}^*(\lambda_2), \;
x_3 \in \mathit{HF}^*(-\lambda_1-\lambda_2).
\end{aligned}
\end{equation}
Moreover, the product with the element \eqref{eq:1-element} yields a continuation map. Using that and associativity, one can show that there is an induced product on $\mathit{HF}^*(\infty)$, which makes that space into a graded commutative unital ring. For $\lambda_1,\lambda_2>0$, the following diagram involving \eqref{eq:morse-to-floer} commutes:
\begin{equation} \label{eq:ring-str-1}
\xymatrix{
H^*(M;\bK) \otimes H^*(M;\bK) \ar[rr] \ar[d]_-{\text{cup-product}} && 
\mathit{HF}^*(\lambda_2) \otimes \mathit{HF}^*(\lambda_1) \ar[d]^-{\text{product}} \\
H^*(M;\bK) \ar[rr] && \mathit{HF}^*(\lambda_1+\lambda_2).
}
\end{equation}
Similarly, if $\lambda_1,\lambda_2<0$ one has a commutative diagram
\begin{equation} \label{eq:ring-str-2}
\xymatrix{ 
\mathit{HF}^*(\lambda_2) \otimes \mathit{HF}^*(\lambda_1) \ar[rr] \ar[d]_-{\text{product}} &&
H^*_{\mathit{cpt}}(M;\bK) \otimes H^*_{\mathit{cpt}}(M;\bK) \ar[d]^-{\text{cup-product}}
\\
\mathit{HF}^*(\lambda_1+\lambda_2) \ar[rr] &&
H^*_{\mathit{cpt}}(M;\bK).
}
\end{equation}
(There is also a mixed-sign version, which we omit.)

Maybe more interestingly, Hamiltonian Floer cohomology also carries a Lie bracket 
\begin{equation} \label{eq:lie-bracket}
[\cdot,\cdot]: \mathit{HF}^*(\lambda_2) \otimes \mathit{HF}^*(\lambda_1) \longrightarrow \mathit{HF}^{*-1}(\lambda_1+ \lambda_2).
\end{equation}
This can actually be expressed in terms of the product and BV operator, as
\begin{equation} \label{eq:bv-relation-1}
[x_2,x_1] = \Delta(x_2x_1) - (\Delta x_2)x_1 - (-1)^{|x_2|} x_2 (\Delta x_1).
\end{equation}
The bracket with \eqref{eq:1-element} vanishes (because of \eqref{eq:bv-relation-1}, $\Delta 1 = 0$, and the fact that $\Delta$ commutes with continuation maps; there is also an alternative more direct proof). Finally, one can show, either from \eqref{eq:bv-relation-1} or directly from the definition, that the bracket commutes with continuation maps, hence induces a bracket on $\mathit{HF}^*(\infty)$. In fact, that space acquires the structure of a BV algebra \cite[Section 8]{seidel07}.

Returning to (maybe) more familiar territory, the Lagrangian Floer cohomology groups come with an associative product, the triangle product
\begin{equation} \label{eq:lag-product}
\mathit{HF}^*(L_1,L_2) \otimes \mathit{HF}^*(L_0,L_1) \longrightarrow \mathit{HF}^*(L_0,L_2).
\end{equation}
For $L_0 = L_1 = L_2 = L$, this is compatible with the isomorphism \eqref{eq:hf-h}. More generally, the class in $\mathit{HF}^0(L,L)$ corresponding to $1 \in H^0(L;\bK)$ under that isomorphism is a two-sided unit for \eqref{eq:lag-product}. On the other hand, the product combined with integration over $L$ yields a nondegenerate (and graded symmetric) pairing
\begin{equation} \label{eq:poincare-pairing}
\mathit{HF}^*(L_1,L_0) \otimes \mathit{HF}^{n-*}(L_0,L_1)
\xrightarrow{\text{product}} \mathit{HF}^n(L_0,L_0) \iso H^n(L_0;\bK)
\xrightarrow{\int_{L_0}} \bK,
\end{equation}
which reproduces \eqref{eq:poincare-duality-lag}.

\section{The main argument\label{sec:strategy}}

\subsection{Degree one classes and twisting}
We will now explain how to prove our main results, Theorems \ref{th:1}--\ref{th:5}, using certain additional structures on Floer cohomology groups. Those structures build on the ones described in Section \ref{sec:dilations}, but go a little further. For the moment, we will state their existence and properties without proofs, and instead concentrate on what those properties imply. Some of the necessary details were carried out in \cite{seidel-solomon10}, and the rest will be provided in Sections \ref{sec:operations}--\ref{sec:technical} below. Throughout, we work in the context from Setup \ref{th:setup}. We fix some $\mu>0$ such that $\mu \notin \scrP_M$, $2\mu \notin \scrP_M$, and a class $B \in \mathit{HF}^1(2\mu)$, which can be arbitrary for now. 

Associated to $B$ is a twisted version of Floer cohomology, which we denote by $\tilde{H}^*$. It fits into a long exact sequence
\begin{equation} \label{eq:cone-les}
\cdots \rightarrow \mathit{HF}^{*-1}(-\mu) \stackrel{-B}{\longrightarrow} \mathit{HF}^*(\mu) \longrightarrow \tilde{H}^* \longrightarrow \mathit{HF}^*(-\mu) \stackrel{-B}{\longrightarrow} \mathit{HF}^{*+1}(\mu) \rightarrow \cdots
\end{equation}
Here, by $-B$ we mean the multiplication map $x \mapsto -Bx$ in the sense of \eqref{eq:product}. The existence of $\tilde{H}^*$ should not come as a surprise: on the chain complex level, one takes a chain map representing the multiplication map, and then $\tilde{H}^*$ is the cohomology of its mapping cone. More importantly, there is a pairing
\begin{equation} \label{eq:i-pairing}
I: \tilde{H}^* \otimes \tilde{H}^{2n-*} \longrightarrow \bK.
\end{equation}
This can be partly understood in terms of the natural duality \eqref{eq:chain-duality} between $\mathit{HF}^*(\pm \mu)$. More precisely, the statement is the following:

\begin{lemma} \label{th:pairing-1}
Take $\tilde{x}_0,\tilde{x}_1 \in \tilde{H}^*$. Suppose that in the long exact sequence \eqref{eq:cone-les}, $\tilde{x}_1$ is the image of some $x_1 \in \mathit{HF}^*(\mu)$, and that $\tilde{x}_0$ maps to $\xi_0 \in \mathit{HF}^*(-\mu)$. Then \parskip0em
\begin{equation}
I(\tilde{x}_0,\tilde{x}_1) = -\langle \xi_0,x_1 \rangle.
\end{equation}
(A parallel statement holds with the roles of $\tilde{x}_0$ and $\tilde{x}_1$ reversed.)
\end{lemma}

\begin{corollary}
$I$ is nondegenerate.
\end{corollary}

\begin{proof}
Suppose that $\tilde{x}_0$ lies in the nullspace, so $I(\tilde{x}_0,\tilde{x}_1) = 0$ for all $\tilde{x}_1$. By Lemma \ref{th:pairing-1}, the image of $\tilde{x}_0$ in $\mathit{HF}^*(-\mu)$ must vanish, which means that $\tilde{x}_0$ itself comes from a class $x_0 \in \mathit{HF}^*(\mu)$. But then, the corresponding formula with the two entries reversed shows that $\langle \xi,x_0 \rangle = 0$ for all $\xi \in \mathit{HF}^*(-\mu)$ which satisfy $B\xi = 0$. Because of the commutativity of the product and its cyclic symmetry \eqref{eq:cyclic-symmetry-1}, the multiplication maps
\begin{equation} \label{eq:b-duality}
\begin{aligned}
& B: \mathit{HF}^k(-\mu) \longrightarrow \mathit{HF}^{k+1}(\mu), \\
& B: \mathit{HF}^{2n-k-1}(-\mu) \longrightarrow \mathit{HF}^{2n-k}(\mu)
\end{aligned}
\end{equation}
are dual (up to signs) with respect to \eqref{eq:poincare-duality-hf}. It follows that $x_0$ itself lies in the image of multiplication by $B$, hence $\tilde{x}_0$ vanishes.
\end{proof}

Lemma \ref{th:pairing-1} does not describe $I$ fully, and indeed there is an extra ingredient that enters into its definition, which is related to the failure of \eqref{eq:product} to be commutative on the chain level. One can see an indirect sign of that extra ingredient in the following symmetry formula:

\begin{lemma} \label{th:pairing-2}
Take classes $\tilde{x}_0,\tilde{x}_1 \in \tilde{H}^*$, and denote their images in $\mathit{HF}^*(-\mu)$ by $\xi_0,\xi_1$. Then
\begin{equation} \label{eq:i-symmetry}
I(\tilde{x}_0,\tilde{x}_1) + (-1)^{|\tilde{x}_0|} I(\tilde{x}_1,\tilde{x}_0) =
\langle B, [\xi_0,\xi_1] \rangle.
\end{equation}
\end{lemma}

Looking slightly ahead to the case of dilations, this implies the following:

\begin{corollary}
In the situation of Lemma \ref{th:pairing-2}, suppose that $\Delta B = 1 \in \mathit{HF}^0(2\mu)$, and that $\Delta \xi_0 = 0$, $\Delta \xi_1 = 0$. Let $\lambda_0,\lambda_1 \in H_*(M;\bK)$ be the images of $\xi_0,\xi_1$ under \eqref{eq:dual-morse-to-floer}. Then
\begin{equation} \label{eq:i-symmetry-3}
I(\tilde{x}_0,\tilde{x}_1) + (-1)^{|\tilde{x}_0|} I(\tilde{x}_1,\tilde{x}_0) = \lambda_0 \cdot \lambda_1,
\end{equation}
where the right hand side is the standard intersection pairing.
\end{corollary}

\begin{proof}
In view of \eqref{eq:bv-relation}, one can rewrite \eqref{eq:i-symmetry} as
\begin{equation} \label{eq:i-symmetry-2}
I(\tilde{x}_0,\tilde{x}_1) + (-1)^{|\tilde{x}_0|} I(\tilde{x}_1,\tilde{x}_0) =
\langle B, \Delta(\xi_0\xi_1) \rangle - \langle B, (\Delta \xi_0)\xi_1 + (-1)^{|\xi_0|} \xi_0 (\Delta \xi_1) \rangle.
\end{equation}
Under our assumptions, the last two terms on the right hand side vanish. Because of the general compatibility of the BV operator with the pairing $\langle \cdot, \cdot \rangle$, the remaining term equals $\langle \Delta B, \xi_0\xi_1 \rangle = \langle 1, \xi_0\xi_1 \rangle$, which is also the integral of the image of $\xi_0\xi_1$ in $H^*_{\mathit{cpt}}(M;\bK)$. Using \eqref{eq:ring-str-2}, one now converts this into a standard intersection number.
\end{proof}

\subsection{Improved intersection numbers}
We will now summarize some of the material from \cite{seidel-solomon10}, and then simplify it further for our purpose. Under the open-closed string map \eqref{eq:open-closed-string-map}, the given $B \in \mathit{HF}^1(2\mu)$ gives rise to a class in $\mathit{HF}^1(L,L) \iso H^1(L;\bK)$. Suppose that this class vanishes. One can then equip $L$ with some additional structure, which makes it into a {\em $B$-equivariant Lagrangian submanifold}, denoted by $\tilde{L}$. To make the definition precise, one needs to choose a cochain representing $B$, and then consider its image under the chain level realization of the open-closed string map; the additional structure is a cochain bounding that image. The possible choices of $B$-equivariant structures on $L$ form an affine space over $\mathit{HF}^0(L,L) \iso H^0(L;\bK) = \bK$ (for more details, see \cite[Section 4]{seidel-solomon10} and Section \ref{sec:operations} below).

The $B$-equivariance property allows us to introduce a secondary open-closed string operation, which takes on the form of an endomorphism 
\begin{equation} \label{eq:tilde-phi}
\Phi_{\tilde{L}_0,\tilde{L}_1}: \mathit{HF}^*(L_0,L_1) \longrightarrow \mathit{HF}^*(L_0,L_1)
\end{equation}
for any pair $(\tilde{L}_0,\tilde{L}_1)$ of $B$-equivariant Lagrangian submanifolds. These endomorphisms were defined in \cite[Equation (4.4)]{seidel-solomon10}, where their basic properties were established. In particular \cite[Equation (4.9)]{seidel-solomon10}:

\begin{lemma} \label{th:derivation}
The maps \eqref{eq:tilde-phi} are derivations with respect to the product \eqref{eq:lag-product}.
\end{lemma}

\begin{corollary} \label{th:degree-0}
$\Phi_{\tilde{L},\tilde{L}}$ acts trivially on $\mathit{HF}^0(L,L) \iso H^0(L;\bK) \iso \bK$. \qed
\end{corollary}

These endomorphisms can be used to define a refined intersection number \cite[Definition 4.3]{seidel-solomon10}, written formally as a function of a parameter $q$. It assigns to a pair $(\tilde{L}_0,\tilde{L}_1)$ the expression
\begin{equation} \label{eq:q-function}
q \longmapsto \sum_\sigma q^\sigma \chi(\mathit{HF}^*(L_0,L_1)_{\sigma}),
\end{equation}
where $\sigma$ runs over elements of the algebraic closure $\bar{\bK}$; and $\mathit{HF}^*(L_0,L_1)_{\sigma}$ is the corresponding generalized eigenspace of \eqref{eq:tilde-phi} inside $\mathit{HF}^*(L_0,L_1) \otimes_{\bK} \bar{\bK}$, whose Euler characteristic we take. Setting $q = 1$ recovers the ordinary intersection number via \eqref{eq:hf-euler}. For our purpose, what is relevant is the first piece of information beyond that, expressed formally as the derivative of \eqref{eq:q-function} at $q = 1$:

\begin{definition}
For $B$-equivariant Lagrangian submanifolds $\tilde{L}_0$, $\tilde{L}_1$, the supertrace (Lefschetz trace) of \eqref{eq:tilde-phi} will be denoted by
\begin{equation} \label{eq:bullet}
\begin{aligned}
\tilde{L}_0 \bullet \tilde{L}_1 & = \mathrm{Str}(\Phi_{\tilde{L}_0,\tilde{L}_1}) \\
& = \sum_k (-1)^k \mathrm{Tr}(\Phi_{\tilde{L}_0,\tilde{L}_1}^k) = 
\sum_\sigma \sigma\, \chi(\mathit{HF}^*(L_0,L_1)_{\sigma}) \; \in \bK.
\end{aligned}
\end{equation}
\end{definition}

The first equality in \eqref{eq:bullet} is the definition (and shows that $\tilde{L}_0 \bullet \tilde{L}_1$ lies in $\bK \subset \bar{\bK}$). The others are reformulations, spelling out the notion of Lefschetz trace in two equivalent ways. The notation $\Phi_{\tilde{L}_0,\tilde{L}_1}^k$ stands for the degree $k$ part, meaning the action of $\Phi_{\tilde{L}_0,\tilde{L}_1}$ on $\mathit{HF}^k(L_0,L_1)$. Because the definition is based on Floer cohomology, the following is obvious:

\begin{corollary} \label{th:basic} 
If $L_0$ and $L_1$ are disjoinable by a Hamiltonian isotopy, then $\tilde{L}_0 \bullet \tilde{L}_1 = 0$. \qed
\end{corollary}

For a $B$-equivariant Lagrangian submanifold, there is a class 
\begin{equation} \label{eq:improved-class2}
\lbr\tilde{L}\rbr \in \tilde{H}^n,
\end{equation}
which maps to $\lbr L \rbr$ under the map $\tilde{H}^* \rightarrow \mathit{HF}^*(-\mu)$ from \eqref{eq:cone-les}. The details, like the definition of $\tilde{H}^*$ itself, are not difficult, see Section \ref{sec:operations} below. The following relation between \eqref{eq:bullet} and the pairing \eqref{eq:i-pairing} is the main insight of this paper:

\begin{theorem} \label{th:cardy2}
$I(\lbr\tilde{L}_0\rbr,\lbr\tilde{L}_1\rbr) = (-1)^{n(n+1)/2} \tilde{L}_0 \bullet \tilde{L}_1$.
\end{theorem}

This is in fact a form of the Cardy relation, similar in principle to that in \cite{abouzaid10} (see also \cite{fukaya-oh-ohta-ono10b} and \cite[Section 5b]{seidel12}), but involving the class $B$ and the equivariant structures on the Lagrangian submanifolds as auxiliary data.

\subsection{Consequences}
From now on, let's impose the condition that $B$ should be a dilation (meaning that one sets $\lambda = 2\mu$ in Definition \ref{th:dilation}, with some arbitrary $\tilde{\lambda}$). Then \cite[Equation (4.16)]{seidel-solomon10}:

\begin{lemma} \label{th:degree-n}
$\Phi_{\tilde{L},\tilde{L}}$ acts as the identity on $\mathit{HF}^n(L,L) \iso H^n(L;\bK) \iso \bK$.
\end{lemma}

By combining this with the derivation property (Lemma \ref{th:derivation}) and looking at the pairing \eqref{eq:poincare-pairing}, we get \cite[Corollary 4.6]{seidel-solomon10}:

\begin{corollary} \label{th:duality}
Under the duality isomorphism \eqref{eq:poincare-duality-lag}, $\Phi_{\tilde{L}_1,\tilde{L}_0}$ corresponds to the dual of $\mathit{id} - \Phi_{\tilde{L}_0,\tilde{L}_1}$.
\end{corollary}

This has the following implications for \eqref{eq:bullet}:

\begin{corollary} \label{th:new-main-properties}
(i) $\tilde{L}_1 \bullet \tilde{L}_0 = (-1)^{n+1} \tilde{L}_0 \bullet \tilde{L}_1 + (-1)^{n(n-1)/2} [L_0] \cdot [L_1]$.

(ii) If $n$ is even,  
$2(\tilde{L} \bullet \tilde{L}) = \chi(L)$, where for $\mathrm{char}(\bK) > 0$, the right hand side is the Euler characteristic reduced mod $\mathrm{char}(\bK)$. If $\mathrm{char}(\bK) = 2$, one finds that $\chi(L)$ must be even, and then there is a refined
equality $\tilde{L} \bullet \tilde{L} = \chi(L)/2 \in \bK$.

(iii) If $n$ is odd and $\mathrm{char}(\bK) = 2$, then $\tilde{L} \bullet \tilde{L} = \chi_{1/2}(L)$ is the semi-characteristic \eqref{eq:kervaire}.

(iv) If $L$ is a $\bK$-homology sphere, $\tilde{L} \bullet \tilde{L} = (-1)^n$. 
\end{corollary}

\begin{proof} 
(i) is an immediate consequence of Corollary \ref{th:duality} and \eqref{eq:hf-euler}.

(ii) The first statement follows from (i) by setting $\tilde{L}_1 = \tilde{L}_0 = \tilde{L}$. For the more refined statement in $\mathrm{char}(\bK) = 2$, note that Corollary \ref{th:duality} implies that the generalized eigenspaces of $\Phi_{\tilde{L},\tilde{L}}$ satisfy
\begin{equation} \label{eq:dim-equal}
\mathit{HF}^*(L,L)_{\sigma} \iso \mathit{HF}^*(L,L)_{1-\sigma},
\end{equation}
where the isomorphism is one of $\bZ/2$-graded vector spaces. Because we are in characteristic $2$, $\sigma \neq 1-\sigma$ for all $\sigma \in \bar{\bK}$, hence \eqref{eq:dim-equal} always relates different eigenspaces. This implies that the total dimension of Floer cohomology, hence also its Euler characteristic, must be even. More precisely, we have
\begin{equation}
(-1)^{n(n+1)/2} \chi(L)/2 = \sum_{\sigma \in S} \chi(\mathit{HF}^*(L,L)_{\sigma}) \in \bZ,
\end{equation}
where on the right hand side we use a subset $S$ of eigenvalues which contains exactly one out of each pair $\{\sigma,1-\sigma\}$. Reducing both sides mod $2$ removes the sign on the left hand side, and allows one to write the right hand side as
\begin{equation}
\sum_{\sigma \in S} \sigma\, \chi(\mathit{HF}^*(L,L)_{\sigma}) +
(1-\sigma) \chi(\mathit{HF}^*(L,L)_{1-\sigma}) = \tilde{L} \bullet \tilde{L}.
\end{equation}

(iii) Write $\Phi_{\tilde{L},\tilde{L}}^k$ for the action on degree $k$ cohomology. By Corollary \ref{th:duality}, 
\begin{equation}
\mathrm{Tr}(\Phi_{\tilde{L},\tilde{L}}^k) + \mathrm{Tr}(\Phi_{\tilde{L},\tilde{L}}^{n-k}) = \mathrm{dim}\, H^k(L;\bK) \text{ mod } 2. 
\end{equation}
Therefore,
\begin{equation}
\tilde{L} \bullet \tilde{L} = \sum_{k=0}^{(n-1)/2} \mathrm{Tr}(\Phi_{\tilde{L},\tilde{L}}^k) + \mathrm{Tr}(\Phi_{\tilde{L},\tilde{L}}^{n-k}) = \chi_{1/2}(L).
\end{equation}

(iv) Follows immediately from Corollary \ref{th:degree-0} and Lemma \ref{th:degree-n}.
\end{proof}

With Theorem \ref{th:cardy2} as well as Corollaries \ref{th:basic} and \ref{th:new-main-properties} as our main tools, we now proceed to establish the main results.

\begin{proof}[Proof of Theorem \ref{th:1}]
Let $(L_1,\dots,L_r)$ be a collection of Lagrangian submanifolds as in the statement of the theorem (in fact, we do not strictly speaking need them to be pairwise disjoinable, but only that $\mathit{HF}^*(L_i,L_j) = 0$ for $i \neq j$). Each $L_i$ is automatically exact, can be equipped with a grading, and can also be made $B$-equivariant; all this because $H^1(L_i;\bK) = 0$.

By Corollary \ref{th:basic} and Corollary \ref{th:new-main-properties}(iv), we have
\begin{equation} \label{eq:diagonal-matrix}
\tilde{L}_i \bullet \tilde{L}_j = \begin{cases} (-1)^n & i = j, \\ 0 & i \neq j. \end{cases}
\end{equation}
In view of Theorem \ref{th:cardy2}, this means that the classes $\lbr \tilde{L}_i \rbr$ must be linearly independent, which yields the desired bound $r \leq N$, with $N = \mathrm{dim}\, \tilde{H}^n$.
\end{proof}

It is worth while to discuss the bound we have just obtained a little further. Let $\rho$ be the rank of either of the maps \eqref{eq:b-duality} for $k = n$. Because they appear as connecting maps in \eqref{eq:cone-les}, we get
\begin{equation}
\mathrm{dim}\, \tilde{H}^n = 2(\mathrm{dim}\, \mathit{HF}^n(\mu) - \rho).
\end{equation}
Lemma \ref{th:pairing-1} shows that the image of $\mathit{HF}^n(\mu)$ inside $\tilde{H}^n$ is a half-dimensional subspace which is isotropic for $I$. Because \eqref{eq:diagonal-matrix} is nondegenerate, the intersection of the subspace spanned by the $\lbr \tilde{L}_i \rbr$ with any isotropic subspace can be at most of dimension $r/2$. Moreover, if $\bK \subset \bR$, then \eqref{eq:diagonal-matrix} is definite, hence any such intersection must be actually $0$. This implies that:

\begin{corollary} \label{th:linearly-independent}
In the situation of Theorem \ref{th:1}, the classes $(\lbr L_1 \rbr, \dots, \lbr L_r \rbr)$ must span a subspace of $\mathit{HF}^n(-\mu)$ of dimension $\geq r/2$. If $\bK \subset \bR$, that bound can be improved to $r$, meaning that $(\lbr L_1 \rbr,\dots, \lbr L_r \rbr)$ are actually linearly independent. \qed
\end{corollary}

\begin{proof}[Proof of Theorem \ref{th:2}]
This uses the same argument as Theorem \ref{th:1}, except that the non-triviality of the diagonal entries in \eqref{eq:diagonal-matrix} is now a consequence of Corollary \ref{th:basic}(iii).
\end{proof}

\begin{proof}[Proof of Theorem \ref{th:3}]
As explained in Example \ref{th:stabilized-milnor-fibre}, these manifolds $M$ have property (H). The argument from Corollary \ref{th:linearly-independent} applies, and shows that the $\lbr L_i \rbr$ are linearly independent in $\mathit{HF}^n(-\mu)$. But by the dual of Definition \ref{th:property-h}(iii), the map $\mathit{HF}^n(-\mu) \rightarrow H^n_{\mathit{cpt}}(M;\bK)$ is an isomorphism, and we know that it maps $\lbr L_i \rbr$ to the ordinary homology class $[L_i]$. Hence, the $[L_i]$ are themselves linearly independent.
\end{proof}

\begin{proof}[Proof of Theorem \ref{th:4}]
This is the same as for Theorem \ref{th:3}, with the weaker bound coming from Corollary \ref{th:linearly-independent}.
\end{proof}

\begin{proof}[Proof of Theorem \ref{th:5}]
This is the same as for Theorem \ref{th:4}, except that as in the proof of Theorem \ref{th:2}, we use Corollary \ref{th:basic}(iii) to obtain \eqref{eq:diagonal-matrix}
\end{proof}

\begin{example} \label{th:lens-space}
Let $M = T^*L$ be the cotangent bundle of a three-dimensional lens space $L = L(p,q)$, and take $\mathrm{char}(\bK) = 0$. As in Example \ref{th:cotangent-sphere}, one has
\begin{equation}
\mathit{SH}^*(M) \iso H_{3-*}(\scrL L;\bK).
\end{equation}
In particular, the direct summand of $\mathit{SH}^*(M)$ coming from contractible loops is isomorphic to $H_{3-*}(\scrL S^3;\bK)^{\bZ/p}$, and one can use that to show that $M$ admits a dilation, compare \cite[Example 6.4]{seidel-solomon10}. To be more precise, take a Hamiltonian such that $H_M(x) = \|x\|$ at infinity, with respect to the round metric. Then, again following Example \ref{th:cotangent-sphere}, for all sufficiently large $\mu$ one has
\begin{equation}
\mathit{HF}^3(\mu) \iso \mathit{SH}^3(M) \iso H_0(\scrL L;\bK) \iso \bK^p.
\end{equation}
The proof of Theorem \ref{th:1} then yields an upper bound of $2p$ for the number of pairwise disjoinable Lagrangian $\bK$-homology spheres in $M$. 

This behaviour (linear growth in $p$) may seem far off the mark, in view of the nearby Lagrangian conjecture, but it can be motivated as follows. One can enlarge the Fukaya category by admitting Lagrangian submanifolds (still as in Setup \ref{th:setup-lagrangian}) equipped with flat $\bK$-vector bundles $\xi_L \rightarrow L$. The analogue of \eqref{eq:hf-h} in that context is that Floer cohomology reproduces ordinary cohomology with local coefficients:
\begin{equation}
\mathit{HF}^*((L,\xi_0),(L,\xi_1)) \iso H^*(L;\mathit{Hom}_{\bK}(\xi_0,\xi_1)).
\end{equation}
All our results apply to such objects as well. In the current example, suppose that we take $\bK = \bC$. The zero-section $L$ admits flat complex line bundles $(\xi_0,\dots,\xi_{p-1})$, which have holonomy $1,e^{2\pi i/p},\dots,e^{2\pi i (p-1)/p}$ around a fixed generator of $\pi_1(L) \iso \bZ/p$. The resulting objects of the enlarged Fukaya category behave formally like pairwise disjoinable Lagrangian homology spheres:
\begin{equation} \label{eq:pseudo-disjoint}
\mathit{HF}^*((L,\xi_i),(L,\xi_j)) \iso \begin{cases} H^*(L;\bC) & i = j, \\ 0 & \text{otherwise}. 
\end{cases}
\end{equation}
This gets us within a factor of $2$ of the upper bound derived above. 

Alternatively, one could use $\bK = \bR$, where our upper bound can be improved to $p$ using Corollary \ref{th:linearly-independent}. However, $\xi_i$ is isomorphic to $\xi_{p-i}$ as a real flat vector bundle. Hence, the number of non-isomorphic indecomposable flat bundles is $[p/2]+1$, which is still not sharp (except for the case $p = 2$, which is $L = \bR P^3$).
\end{example}

At this point, our remaining tasks are: the precise definition of the pairing \eqref{eq:i-pairing}, with the proof of its properties (Lemmas \ref{th:pairing-1} and, less importantly for our purpose, Lemma \ref{th:pairing-2}); and the proof of Theorem \ref{th:cardy2}.

\section{Chain level operations\label{sec:operations}}

\subsection{General framework}
The structure of operations on Floer cochain complexes is governed by the geometry of a certain class of Riemann surfaces. To bring out this structure clearly, we temporarily suppress many of the technical details, and proceed in a style reminiscent of TCFTs (Topological Conformal Field Theories, compare e.g.\ \cite{costello05}). Of course, many examples have already occured in Section \ref{sec:dilations}, even though at that point we did not attempt to put them in a common framework.

\begin{setup} \label{th:setup-2}
Take $S = \bar{S} \setminus \Sigma$, where $\bar{S}$ is a compact Riemann surface with boundary, and $\Sigma$ is a finite set of points, which may include both
\begin{equation}
\left\{
\begin{aligned}
& \text{interior points } \;
\Sigma^{\mathit{cl}} = \Sigma \setminus \partial \bar{S} \;\text{ and} \\
& \text{boundary points }\; \Sigma^{\mathit{op}} = \Sigma \cap \partial \bar{S}.
\end{aligned}
\right.
\end{equation}
The case where $S$ itself is closed is excluded. At each point of $\Sigma^{\mathit{cl}}$, we want to have additional structure, namely a preferred tangent direction (a nonzero tangent vector on $\bar{S}$ determined up to positive real multiples). We further assume that the points of $\Sigma$ have been divided into {\em inputs} and {\em outputs}, denoted by $\Sigma^{\mathit{in}}$ and $\Sigma^{\mathit{out}}$. Write
\begin{equation}
\begin{aligned}
& \Sigma^{\mathit{cl,out}} = \Sigma^{\mathit{cl}} \cap \Sigma^{\mathit{out}},
&& \Sigma^{\mathit{cl,in}} = \Sigma^{\mathit{cl}} \cap \Sigma^{\mathit{in}}, \\
& \Sigma^{\mathit{op,out}} = \Sigma^{\mathit{op}} \cap \Sigma^{\mathit{out}},
&& \Sigma^{\mathit{op,in}} = \Sigma^{\mathit{op}} \cap \Sigma^{\mathit{in}}.
\end{aligned}
\end{equation}

Additionally, $S$ should come with a real one-form $\nu_S \in \Omega^1(S)$ such that 
\begin{equation} \label{eq:one-form}
\left\{
\begin{aligned}
& d\nu_S \leq 0 && \text{everywhere}, \\
& d\nu_S = 0 && \text{near $\Sigma$,} \\
& \nu_S|\partial S = 0 && \text{near $\Sigma^{\mathit{op}}$,}
\end{aligned}
\right.
\end{equation}
where the last condition means that $\nu_S|\partial S \in \Omega^1(\partial S)$ vanishes at points close enough to $\Sigma^{\mathit{op}}$. This allows us to associate to each point $\zeta \in \Sigma$ a real number $\lambda_\zeta \in \bR$. Namely, for $\zeta \in \Sigma^{\mathit{cl}}$, we define $\lambda_\zeta$ by integrating $\nu_S$ along a small loop around $\zeta$; this winds clockwise if $\zeta$ is an input, and anticlockwise if it is an output. For $\zeta \in \Sigma^{\mathit{op}}$, we define $\lambda_\zeta$ by integrating $\nu_S$ along a small path going from one component of $\partial S$ close to $\zeta$ to the other one; as before, the winding direction is clockwise for inputs and anticlockwise for outputs.
\end{setup}

So far, we have considered the Riemann surfaces by themselves. Adding target space data is done as follows:

\begin{setup} \label{th:setup-2.5}
Fix a target manifold as in Setup \ref{th:setup}. Moreover, for each pair $(L_0,L_1)$ of Lagrangian submanifolds as in Setup \ref{th:setup-lagrangian} fix a number $\lambda_{L_0,L_1} \in \bR$ (it is possible, but not necessary, to set all these numbers to $0$). 

For $\zeta \in \Sigma^{\mathit{cl}}$, we then require that $\lambda_\zeta \notin \scrP$. We also want to have a Lagrangian submanifold $L_C$ associated to each boundary component $C \subset \partial S$. For any point $\zeta \in \Sigma^{\mathit{op}}$, this determines a pair of Lagrangian submanifolds $(L_{\zeta,0},L_{\zeta,1})$. If $\zeta$ is an input, then $L_{\zeta,0}$ is associated to the component of $\partial S$ preceding $\zeta$ in the boundary orientation, and $L_{\zeta,1}$ to the successive one; whereas if $\zeta$ is an output, the convention is opposite. Given that, we require that 
\begin{equation}
\lambda_\zeta = \lambda_{L_{\zeta,0},L_{\zeta,1}} \;\; \text{ for $\zeta \in \Sigma^{\mathit{op}}$.}
\end{equation}
\end{setup}

Assume that Floer cochain complexes $\mathit{CF}^*(\lambda)$ and $\mathit{CF}^*(L_0,L_1)$ have been defined (the latter definition involves the constants $\lambda_{L_0,L_1}$, even though the outcome is independent of that choice up to quasi-isomorphism). The operation associated to a single $S$ as in Setup \ref{th:setup-2}, \ref{th:setup-2.5}, is a chain map
\begin{equation} \label{eq:operation}
\bigotimes_{\zeta \in \Sigma^{\mathit{cl,in}}} \mathit{CF}^*(\lambda_\zeta) \otimes \bigotimes_{\zeta \in \Sigma^{\mathit{op,in}}} \mathit{CF}^*(L_{\zeta,0},L_{\zeta,1})
\longrightarrow
\bigotimes_{\zeta \in \Sigma^{\mathit{cl,out}}} \mathit{CF}^*(\lambda_\zeta) \otimes \bigotimes_{\zeta \in \Sigma^{\mathit{op,out}}} \mathit{CF}^*(L_{\zeta,0},L_{\zeta,1})
\end{equation}
of degree $n(-\chi(\bar{S}) + 2|\Sigma^{\mathit{cl,out}}| + |\Sigma^{\mathit{op,out}}|)$. 

\begin{remark} \label{th:fail-2}
The distinction of points of $\Sigma^{\mathit{cl}}$ into inputs and outputs is only a formality: turning one into the other amounts to applying \eqref{eq:chain-duality} to the associated maps \eqref{eq:operation}. The same is not quite true for $\Sigma^{\mathit{op}}$, because we do not have a strict chain level duality underlying \eqref{eq:poincare-duality-lag}, see Remark \ref{th:fail}.
\end{remark}

As mentioned at the start of our discussion, several constructions from Section \ref{sec:dilations} fall into this general category (or rather, the underlying cochain level structures do):

{\em Continuation maps \eqref{eq:continuation-map}.} As already described in our original discussion of these maps, the relevant Riemann surface is $S = \bR \times S^1$, with $\zeta_0 = \{s = -\infty\}$ considered as output, and $\zeta_1 = \{s = +\infty\}$ as input. The one-form is chosen so that the real numbers associated to the two points of $\Sigma$ are the $\lambda_0,\lambda_1$ appearing in \eqref{eq:continuation-map}. The tangent directions at both points are chosen to point along the line $\{t = 0\}$ (since $S$ has a rotational symmetry, the coordinate-independent part of this choice is the fact that the two tangent directions point towards each other).

{\em Unit elements \eqref{eq:1-element}.} These have $S = (\bR \times S^1) \cup \{s = +\infty\} \iso \bC$, with the only remaining point $\zeta = \{s = -\infty\}$ being an output. As before, our convention is to choose the tangent vector at $\zeta$ to point along $\{t = 0\}$ (but this time, this choice has no coordinate-independent meaning).

{\em Duality for Lagrangian Floer cohomology \eqref{eq:poincare-duality-lag}.} Take $S = \bR \times [0,1]$ with boundary conditions $L_0,L_1$, but where both ends are taken to be inputs. This defines a chain map $\mathit{CF}^*(L_1,L_0) \otimes \mathit{CF}^{n-*}(L_0,L_1) \rightarrow \bK$, which gives rise to a nondegenerate pairing on cohomology.

{\em Hamiltonian Floer cohomology classes associated to Lagrangian submanifolds \eqref{eq:improved-class}.} For these, one takes $S = (-\infty,0] \times S^1$, which has a single boundary component labeled with the Lagrangian submanifold $L$. We again choose $\{t = 0\}$ as the tangent direction at the point $\zeta = \{s = -\infty\}$ (but this has no intrinsic meaning, just as in the case of the unit elements). Note that because there is no restriction on $\nu_S|\partial S$, the condition that $d\nu_S \leq 0$ everywhere does not place any constraint on the number $\lambda = \lambda_\zeta \in \bR \setminus \scrP_M$. We denote the resulting cochain representative of $\lbr L \rbr$ by
\begin{equation}
\phi^{1,0}_L \in \mathit{CF}^n(\lambda).
\end{equation}

{\em Open-closed string maps \eqref{eq:open-closed-string-map}, \eqref{eq:dual-open-closed-string-map}.} We denote the underlying cochain level maps by 
\begin{equation} \label{eq:oc-chain-maps}
\begin{aligned}
& \phi^{1,1}_L: \mathit{CF}^*(\lambda) \longrightarrow \mathit{CF}^*(L,L), \\
& \check{\phi}^{1,1}_L: \mathit{CF}^*(L,L) \longrightarrow \mathit{CF}^{*+n}(\lambda).
\end{aligned}
\end{equation}
Both of them are defined using isomorphic Riemann surfaces, namely a disc with one interior and one boundary puncture. For $\phi^{1,1}_L$, the convention is that the preferred tangent direction at the interior puncture (which is an input) points towards the boundary puncture (which is an output). For $\check{\phi}^{1,1}_L$, the roles of input and output are exchanged, and we also require that the tangent direction should point in the opposite way (away from the boundary puncture). 

\begin{remark} \label{th:fail-3}
Since the induced cohomology level maps \eqref{eq:open-closed-string-map}, \eqref{eq:dual-open-closed-string-map} are dual, by introducing separate chain level models we are allowing a certain amount of redundancy. The fact that we find it convenient to do so is an effect of the asymmetry first pointed out in Remarks \ref{th:fail} and Remark \ref{th:fail-2}.
\end{remark}

{\em Pair-of-pants product \eqref{eq:product}:} We denote the underlying cochain map by
\begin{equation} \label{eq:smile-product}
\smile\;: \mathit{CF}^*(\lambda_2) \otimes \mathit{CF}^*(\lambda_1) \longrightarrow \mathit{CF}^*(\lambda_1+\lambda_2).
\end{equation}
The underlying Riemann surface $S$ is a three-punctured sphere (two inputs, one output). To fix the conventions more precisely, we identify $\bar{S}$ with the standard round sphere, and then take the punctures to be equidistributed along the equator. The tangent direction at each puncture goes towards the next one, where the convention is that for the expression $\langle x_3, x_2 \smile x_1 \rangle$, the punctures are cyclically ordered in a way that corresponds to $(x_1,x_2,x_3)$. By definition,
the one-form $\nu_S$ is necessarily closed.

One can arrange the choices so that \eqref{eq:cyclic-symmetry-1} is strictly realized on the cochain level, meaning that 
\begin{equation} \label{eq:cyclic-symmetry-2}
\begin{aligned} &
\langle x_3, x_2 \smile x_1 \rangle = (-1)^{|x_3|} \langle x_2, x_1 \smile x_3 \rangle = (-1)^{|x_1|} 
\langle x_1 \smile x_3, x_2 \rangle \\
& \text{for } x_1 \in \mathit{CF}^*(\lambda_1), \; x_2 \in \mathit{CF}^*(\lambda_2), \;
x_3 \in \mathit{CF}^*(-\lambda_1-\lambda_2).
\end{aligned}
\end{equation}
This is technically unproblematic because cyclic permutations act freely on all possible triples $(\lambda_1,\lambda_2,\lambda_3)$ (which is a consequence of the requirement that $\lambda_k \neq 0$); hence, what \eqref{eq:cyclic-symmetry-2} is asking is merely that we should coordinate the choices associated to two different geometric situations, in parallel with \eqref{eq:swap-sign}.

{\em Triangle product \eqref{eq:lag-product}:} The chain map underlying this product is written as
 \begin{equation} \label{eq:mu2}
\mu^2_{L_0,L_1,L_2}: \mathit{CF}^*(L_1,L_2) \otimes \mathit{CF}^*(L_0,L_1) \longrightarrow \mathit{CF}^*(L_0,L_2).
\end{equation}
To define it, one takes $S$ to be a disc with three boundary punctures. 

\begin{remark}
The formalism built above does not accommodate maps which mix ordinary (Morse or singular) cochains and Floer complexes, such as \eqref{eq:morse-to-floer}, \eqref{eq:dual-morse-to-floer} and \eqref{eq:hf-h}. This is acceptable for our purpose, since standard properties of those maps are all we need. 
\end{remark}

The next level of sophistication concerns operations induced by $\smooth$-families of Riemann surfaces (each as in Setup \ref{th:setup-2}, \ref{th:setup-2.5}) parametrized by a compact oriented manifold $R$. The simplest situation is where $R$ is closed. In that case, the associated operation is a chain map as in \eqref{eq:operation}, but whose degree is lower by $\mathrm{dim}(R)$. Two examples of this have occurred previously:

{\em BV operator.} 
The chain level structure underlying \eqref{eq:bv-operator-1} will be written as
\begin{equation} \label{eq:bv-operator}
\delta: \mathit{CF}^*(\lambda) \longrightarrow \mathit{CF}^{*-1}(\lambda).
\end{equation}
Take a family $(S_r)$ parametrized by $r \in S^1 = \bR/\bZ$. As a Riemann surface, $S_r = \bR \times S^1$ is independent of $r$, and so is the one-form $\nu_{S_r} = \lambda\, \mathit{dt}$. The preferred tangent direction at $s = -\infty$ corresponds to $t = r$, while that at $s = +\infty$ corresponds to $t =0$. Equivalently, $\bar{S} = \bC P^1$ with two marked points, and the tangent direction at one of those points rotates anticlockwise with $r$, while the other is constant (up to isomorphism, the choice of which marked point to use for each behaviour is irrelevant). 

{\em Lie bracket \eqref{eq:lie-bracket}.}
One can define this in terms of a family of three-punctured spheres parametrized by $S^1$, where the tangent directions vary with the parameter. However, we prefer to break up the parameter space into two intervals, and correspondingly to write the bracket as a sum of two terms \eqref{eq:lie}, see below.

\begin{remark}
At this point, the importance of the choice of tangent directions at points of $\Sigma^{\mathit{cl}}$ has hopefully become clear: if those directions were irrelevant, both the BV operator and the bracket would have to vanish (which is indeed what happens for closed symplectic manifolds).
\end{remark}

A slightly more complicated situation occurs when the parameter space $R$ is a compact manifold with boundary (or corners). In that case, the resulting operation is still a map whose degree is $\mathrm{dim}(R)$ lower than in \eqref{eq:operation}. However, instead of being a chain map, it is a nullhomotopy for the operation associated to $\partial R$ (or, in the case where $R$ has corners, for the sum of operations associated to the codimension one boundary faces). Here is one important example of such an operation:

{\em Homotopy commutativity.} Consider a family $(S_r)$ of Riemann surfaces parametrized by $r \in R = [0,1]$. Each $S_r$ is the standard sphere with three marked points equidistributed along the equator, exactly as in the definition of the pair-of-pants-product. In fact, for $r = 1$ we use exactly the setup (one-forms, tangent directions) from \eqref{eq:smile-product}. For $r = 0$ we use the pullback of that by the automorphism of the sphere which swaps the two input points and preserves the output point. This clearly defines the operation $(x_1,x_2) \mapsto (-1)^{|x_2| \cdot |x_1|} x_1 \smile x_2$. As $r$ varies from $0$ to $1$, all tangent directions perform a half-turn: anticlockwise for the output, clockwise for the inputs (Figure \ref{fig:star}).
\begin{figure}
\begin{centering}
\begin{picture}(0,0)%
\includegraphics{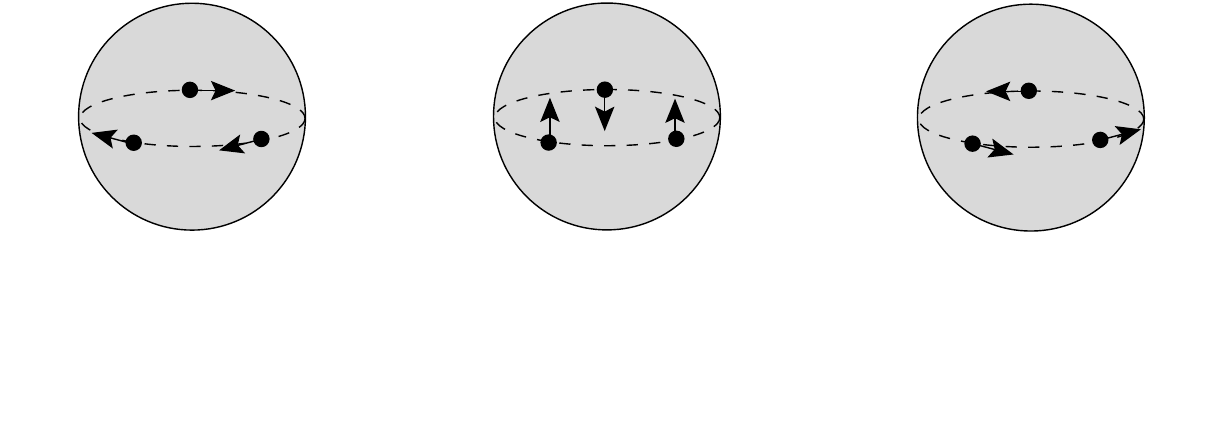}%
\end{picture}%
\setlength{\unitlength}{3552sp}%
\begingroup\makeatletter\ifx\SetFigFont\undefined%
\gdef\SetFigFont#1#2#3#4#5{%
  \reset@font\fontsize{#1}{#2pt}%
  \fontfamily{#3}\fontseries{#4}\fontshape{#5}%
  \selectfont}%
\fi\endgroup%
\begin{picture}(6494,2232)(-1173,-1829)
\put(-374,-1186){\makebox(0,0)[lb]{\smash{{\SetFigFont{11}{13.2}{\rmdefault}{\mddefault}{\updefault}{\color[rgb]{0,0,0}$r=0$}%
}}}}
\put(1230,-1085){\makebox(0,0)[lb]{\smash{{\SetFigFont{11}{13.2}{\rmdefault}{\mddefault}{\updefault}{\color[rgb]{0,0,0}tangent directions }%
}}}}
\put(4437, 12){\makebox(0,0)[lb]{\smash{{\SetFigFont{11}{13.2}{\rmdefault}{\mddefault}{\updefault}{\color[rgb]{0,0,0}output}%
}}}}
\put(4783,-550){\makebox(0,0)[lb]{\smash{{\SetFigFont{11}{13.2}{\rmdefault}{\mddefault}{\updefault}{\color[rgb]{0,0,0}input 2}%
}}}}
\put(-37, 17){\makebox(0,0)[lb]{\smash{{\SetFigFont{11}{13.2}{\rmdefault}{\mddefault}{\updefault}{\color[rgb]{0,0,0}output}%
}}}}
\put(301,-530){\makebox(0,0)[lb]{\smash{{\SetFigFont{11}{13.2}{\rmdefault}{\mddefault}{\updefault}{\color[rgb]{0,0,0}input 2}%
}}}}
\put(4141,-1178){\makebox(0,0)[lb]{\smash{{\SetFigFont{11}{13.2}{\rmdefault}{\mddefault}{\updefault}{\color[rgb]{0,0,0}$r=1$}%
}}}}
\put(3312,-549){\makebox(0,0)[lb]{\smash{{\SetFigFont{11}{13.2}{\rmdefault}{\mddefault}{\updefault}{\color[rgb]{0,0,0}input 1}%
}}}}
\put(-1158,-547){\makebox(0,0)[lb]{\smash{{\SetFigFont{11}{13.2}{\rmdefault}{\mddefault}{\updefault}{\color[rgb]{0,0,0}input 1}%
}}}}
\put(1230,-1310){\makebox(0,0)[lb]{\smash{{\SetFigFont{11}{13.2}{\rmdefault}{\mddefault}{\updefault}{\color[rgb]{0,0,0}rotate anticlockwise}%
}}}}
\put(1230,-1535){\makebox(0,0)[lb]{\smash{{\SetFigFont{11}{13.2}{\rmdefault}{\mddefault}{\updefault}{\color[rgb]{0,0,0}at output, clockwise}%
}}}}
\put(1230,-1760){\makebox(0,0)[lb]{\smash{{\SetFigFont{11}{13.2}{\rmdefault}{\mddefault}{\updefault}{\color[rgb]{0,0,0}at inputs}%
}}}}
\end{picture}%
\caption{\label{fig:star}}
\end{centering}
\end{figure}

From our general setup, it then follows that the resulting operation 
\begin{equation} \label{eq:star-product}
\ast : \mathit{CF}^*(\lambda_2) \otimes \mathit{CF}^*(\lambda_1) \longrightarrow \mathit{CF}^{*-1}(\lambda_1+\lambda_2)
\end{equation}
satisfies
\begin{equation} \label{eq:homotopy-commutativity}
d(x_2 \ast x_1) + dx_2 \ast x_1 + (-1)^{|x_2|} x_2 \ast dx_1 = x_2 \smile x_1 - (-1)^{|x_2| \, |x_1|} x_1 \smile x_2.
\end{equation}
In other words, this is a secondary product which ensures the homotopy commutativity of $\smile$. 
One can symmetrize $\ast$ to produce a bilinear chain map of degree $-1$, which is the chain level version of \eqref{eq:lie-bracket}:
\begin{equation} \label{eq:lie}
[x_2,x_1] = x_2 \ast x_1 + (-1)^{|x_1|\,|x_2|} x_1 \ast x_2.
\end{equation}
%
%


\begin{remark} \label{th:caution}
As this example demonstrates, the presence of an operation \eqref{eq:bv-operator} of degree $-1$ means that special care must be taken when defining secondary operations. There are many more families which interpolate between $S_0$ and $S_1$, but where the tangent directions rotate by different amounts (in $\pi + 2\pi \bZ$). This corresponds to adding integer multiples of $\delta (x_2 \smile x_1)$, $(\delta x_2) \smile x_1$ or $(-1)^{|x_2|} x_2 \smile \delta x_1$ to \eqref{eq:star-product}. Each of those modified operations would satisfy \eqref{eq:bv-relation-1}, but they are substantially different from each other.
\end{remark}

Finally, one can allow non-compact parameter spaces $R$ as long as the behaviour at infinity is tightly controlled. Namely, there must be a compactification $\bar{R}$ to a manifold with corners. As one approaches the boundary, the Riemann surfaces $S_r$ must stretch along tubular or strip-like ends, resulting in ``broken surfaces'' associated to points of $\bar{R} \setminus R$. A general description of the allowed behaviour would be quite lengthy, and we prefer a case-by-case discussion. For the most part, we will only encounter the unproblematic special case where $R$ is one-dimensional. There are also two two-dimensional instances relevant to us, one of which we will mention now (though its role in our overall argument is relatively minor).

{\em Relation between BV operator and Lie bracket.} Consider the cochain level version of \eqref{eq:bv-relation-1},
\begin{equation} \label{eq:bv-relation}
\delta( x_2 \smile x_1) - (\delta x_2) \smile x_1 - (-1)^{|x_2|} x_2 \smile (\delta x_1) - [x_2,x_1] = \text{nullhomotopic}.
\end{equation}
Geometrically, let's combine two copies of the interval defining \eqref{eq:star-product} (one with reversed orientation) to get a family over $S^1$, which defines \eqref{eq:lie}. Having done that, each term on the left hand side of \eqref{eq:bv-relation-1} corresponds to a particular family (three of them ``broken''; these are drawn in Figure \ref{fig:bv-relation}) parametrized by a circle. Unsurprisingly, to relate them, one can use a two-dimensional compactified parameter space $\bar{R}$ which is a genus zero surface with four boundary components. Removing three of the boundary circles yields a non-compact surface $R$. One can construct a family $(S_r)$ of surfaces parametrized by $r \in R$, with the following properties. Over $\partial R \iso S^1$ this reproduces the family underlying \eqref{eq:lie}; while, as one approaches any one of the three components of $\partial \bar{R} \setminus \partial R$, the surfaces $S_r$ are stretched along tubular ends. 

To show the existence of this two-dimensional family, consider all possible ways of choosing tangent directions at the three marked points on the equator of the sphere. The moduli space of all such choices can be identified with $(S^1)^3$, in particular its first homology is $\bZ^3$. If we move slightly inwards from each of three circles of $\bar{R} \setminus R$, which means gluing together the components of the ``broken'' surfaces from Figure \ref{fig:bv-relation}, we get three circles in our moduli space, whose homology classes are $(0,1,0)$, $(1,0,0)$ and $(0,0,-1)$, respectively. On the other hand, combining two copies of the interval from Figure \ref{fig:star} yields a circle in the homology class $(1,1,-1)$. Hence, there is a surface in moduli space which has these four circles (with the orientation of the last one reversed) as boundaries. It is not difficult to see, given the known topology of the moduli space, that the topology of this surface can be taken to be the $\bar{R}$ described above (actually, a surface with a different topology would do as well for our purpose). This argument is not new, being part of the general derivation of BV structures as operads associated to moduli spaces \cite{getzler94b}.
\begin{figure}
\begin{centering}
\begin{picture}(0,0)%
\includegraphics{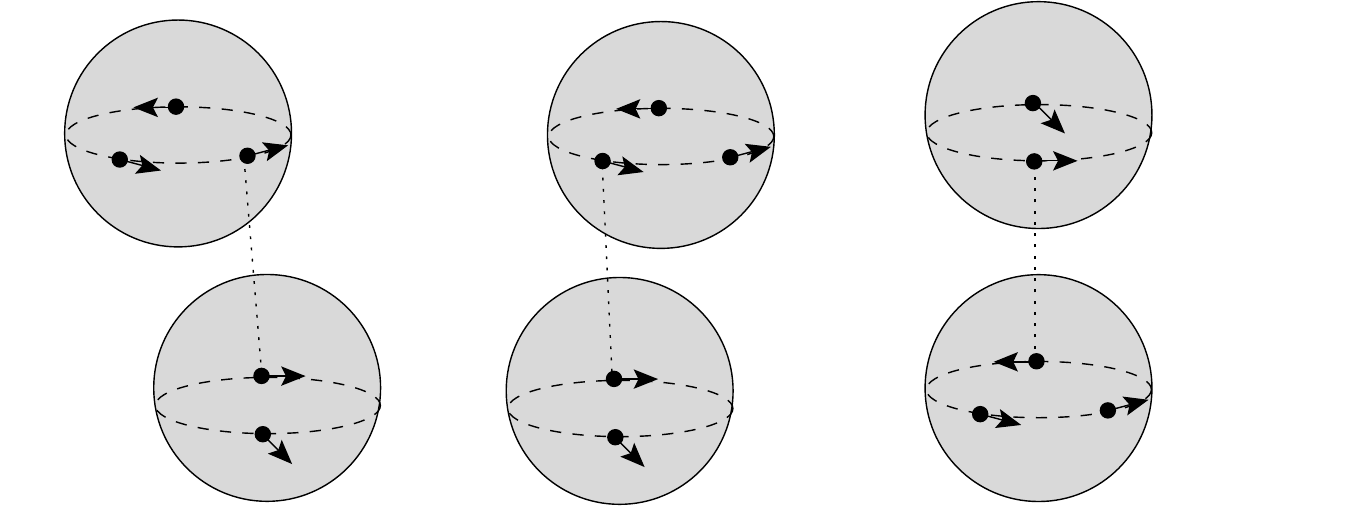}%
\end{picture}%
\setlength{\unitlength}{3552sp}%
\begingroup\makeatletter\ifx\SetFigFont\undefined%
\gdef\SetFigFont#1#2#3#4#5{%
  \reset@font\fontsize{#1}{#2pt}%
  \fontfamily{#3}\fontseries{#4}\fontshape{#5}%
  \selectfont}%
\fi\endgroup%
\begin{picture}(7256,2703)(-914,-850)
\put(491,-781){\makebox(0,0)[lb]{\smash{{\SetFigFont{11}{13.2}{\rmdefault}{\mddefault}{\updefault}{\color[rgb]{0,0,0}input points rotate clockwise}%
}}}}
\put(526,-61){\makebox(0,0)[lb]{\smash{{\SetFigFont{11}{13.2}{\rmdefault}{\mddefault}{\updefault}{\color[rgb]{0,0,0}output}%
}}}}
\put(2401,-61){\makebox(0,0)[lb]{\smash{{\SetFigFont{11}{13.2}{\rmdefault}{\mddefault}{\updefault}{\color[rgb]{0,0,0}output}%
}}}}
\put(526,839){\makebox(0,0)[lb]{\smash{{\SetFigFont{11}{13.2}{\rmdefault}{\mddefault}{\updefault}{\color[rgb]{0,0,0}input}%
}}}}
\put(-749,839){\makebox(0,0)[lb]{\smash{{\SetFigFont{11}{13.2}{\rmdefault}{\mddefault}{\updefault}{\color[rgb]{0,0,0}input}%
}}}}
\put( 76,1364){\makebox(0,0)[lb]{\smash{{\SetFigFont{11}{13.2}{\rmdefault}{\mddefault}{\updefault}{\color[rgb]{0,0,0}output}%
}}}}
\put(-899,1214){\makebox(0,0)[lb]{\smash{{\SetFigFont{11}{13.2}{\rmdefault}{\mddefault}{\updefault}{\color[rgb]{0,0,0}$\smile$}%
}}}}
\put(1726,1214){\makebox(0,0)[lb]{\smash{{\SetFigFont{11}{13.2}{\rmdefault}{\mddefault}{\updefault}{\color[rgb]{0,0,0}$\smile$}%
}}}}
\put(1801,839){\makebox(0,0)[lb]{\smash{{\SetFigFont{11}{13.2}{\rmdefault}{\mddefault}{\updefault}{\color[rgb]{0,0,0}input}%
}}}}
\put(3076,839){\makebox(0,0)[lb]{\smash{{\SetFigFont{11}{13.2}{\rmdefault}{\mddefault}{\updefault}{\color[rgb]{0,0,0}input}%
}}}}
\put(2626,1364){\makebox(0,0)[lb]{\smash{{\SetFigFont{11}{13.2}{\rmdefault}{\mddefault}{\updefault}{\color[rgb]{0,0,0}output}%
}}}}
\put(3751,-286){\makebox(0,0)[lb]{\smash{{\SetFigFont{11}{13.2}{\rmdefault}{\mddefault}{\updefault}{\color[rgb]{0,0,0}$\smile$}%
}}}}
\put(4651,1589){\makebox(0,0)[lb]{\smash{{\SetFigFont{11}{13.2}{\rmdefault}{\mddefault}{\updefault}{\color[rgb]{0,0,0}output point}%
}}}}
\put(4651,1439){\makebox(0,0)[lb]{\smash{{\SetFigFont{11}{13.2}{\rmdefault}{\mddefault}{\updefault}{\color[rgb]{0,0,0}rotates anticlockwise}%
}}}}
\put(4651,764){\makebox(0,0)[lb]{\smash{{\SetFigFont{11}{13.2}{\rmdefault}{\mddefault}{\updefault}{\color[rgb]{0,0,0}input}%
}}}}
\put(3826,-511){\makebox(0,0)[lb]{\smash{{\SetFigFont{11}{13.2}{\rmdefault}{\mddefault}{\updefault}{\color[rgb]{0,0,0}input}%
}}}}
\put(5026,-511){\makebox(0,0)[lb]{\smash{{\SetFigFont{11}{13.2}{\rmdefault}{\mddefault}{\updefault}{\color[rgb]{0,0,0}input}%
}}}}
\put(4651, 14){\makebox(0,0)[lb]{\smash{{\SetFigFont{11}{13.2}{\rmdefault}{\mddefault}{\updefault}{\color[rgb]{0,0,0}output}%
}}}}
\put(-299,-286){\makebox(0,0)[lb]{\smash{{\SetFigFont{11}{13.2}{\rmdefault}{\mddefault}{\updefault}{\color[rgb]{0,0,0}$\delta$}%
}}}}
\put(1576,-286){\makebox(0,0)[lb]{\smash{{\SetFigFont{11}{13.2}{\rmdefault}{\mddefault}{\updefault}{\color[rgb]{0,0,0}$\delta$}%
}}}}
\put(3826,1214){\makebox(0,0)[lb]{\smash{{\SetFigFont{11}{13.2}{\rmdefault}{\mddefault}{\updefault}{\color[rgb]{0,0,0}$\delta$}%
}}}}
\end{picture}%
\caption{\label{fig:bv-relation}}
\end{centering}
\end{figure}

\subsection{Secondary open-closed string maps}
We now turn to the operations which are central to our argument, and which go beyond those previously described in Section \ref{sec:dilations}. The first two of these have the following form:
\begin{align}
& \phi^{2,0}_L: \mathit{CF}^*(\lambda_1) \longrightarrow \mathit{CF}^{*+n-1}(\lambda_0), \label{eq:phi20} \\ 
& \phi^{1,2}_{L_0,L_1}: \mathit{CF}^*(\lambda) \otimes \mathit{CF}^*(L_0,L_1) \longrightarrow \mathit{CF}^{*-1}(L_0,L_1).
\label{eq:phi12}
\end{align}
Their basic properties are:
\begin{align}
\label{eq:phi-phi} & 
d \phi^{2,0}_L(x) + (-1)^n \phi^{2,0}_L(dx) = \check{\phi}^{1,1}_L(\phi^{1,1}_L(x)) - (-1)^{n|x|} x \smile \phi^{1,0}_L, \\
& 
\begin{aligned}
& \mu^1_{L_0,L_1}(\phi^{1,2}_{L_0,L_1}(x,a)) + \phi^{1,2}_{L_0,L_1}(dx,a) + (-1)^{|x|}\phi^{1,2}_{L_0,L_1}(x,\mu^1_{L_0,L_1}(a)) \\ & =
\mu^2_{L_0,L_1,L_1}(\phi^{1,1}_{L_1}(x),a) - (-1)^{|a| \cdot |x|} \mu^2_{L_0,L_0,L_1}(a,\phi^{1,1}_{L_0}(x)).
\end{aligned}
\end{align}

\begin{remark}
The analogue of the cohomology level relation expressed by \eqref{eq:phi-phi} in ordinary topology would be the fact that for $i: L \hookrightarrow M$, pullback followed by pushforward is cup product with the Poincar{\'e} dual class of $L$:
\begin{equation}
i_! i^* = [L] : H^*(M) \longrightarrow H^{*+n}(M). 
\end{equation}
\end{remark}

Let's look at $\phi^{2,0}_L$ in more detail. On the right hand side of \eqref{eq:phi-phi}, $\phi^{1,1}_L$ uses $\mathit{CF}^*(\lambda_1)$; $\check{\phi}^{1,1}_L$ uses $\mathit{CF}^*(\lambda_0)$; and $\phi^{1,0}_L \in \mathit{CF}^n(\lambda_0-\lambda_1)$. Of course, we assume tacitly that all those Floer complexes are defined, but otherwise the situation is unlike \eqref{eq:continuation-map} in that no relation between $\lambda_0$ and $\lambda_1$ needs to hold. The underlying family of Riemann surfaces is parametrized by $R = \bR$. Each $S_r$ is a disc with two interior punctures, carrying a closed one-form $\nu_{S_r}$ which has integral $\lambda_0$ around the input puncture, and $\lambda_1$ around the output puncture (hence $\int_{\partial S_r} \nu_{S_r} = \lambda_0-\lambda_1$). Moreover, we assume that the preferred tangent direction at the input puncture points towards the output puncture, while the tangent direction at the output puncture points away from the input puncture. As $r \rightarrow -\infty$, the two punctures collide, leading to bubbling off of a three-punctured holomorphic sphere. We equip that sphere with the additional structures associated to $\smile$, and the remaining once-punctured disc with those for $\phi^{1,0}_L$. In the other limit $r \rightarrow \infty$, the punctures move apart, and the disc is therefore stretched out into two discs joined at the boundary; we equip those with the structures associated to $\phi^{1,1}_L$ and $\check{\phi}^{1,1}_L$, respectively. See Figure \ref{fig:phi20} for a more precise picture of the limits (the dotted line merely indicates in which way the interior punctures are paired up). 

\begin{figure}
\begin{centering}
\begin{picture}(0,0)%
\includegraphics{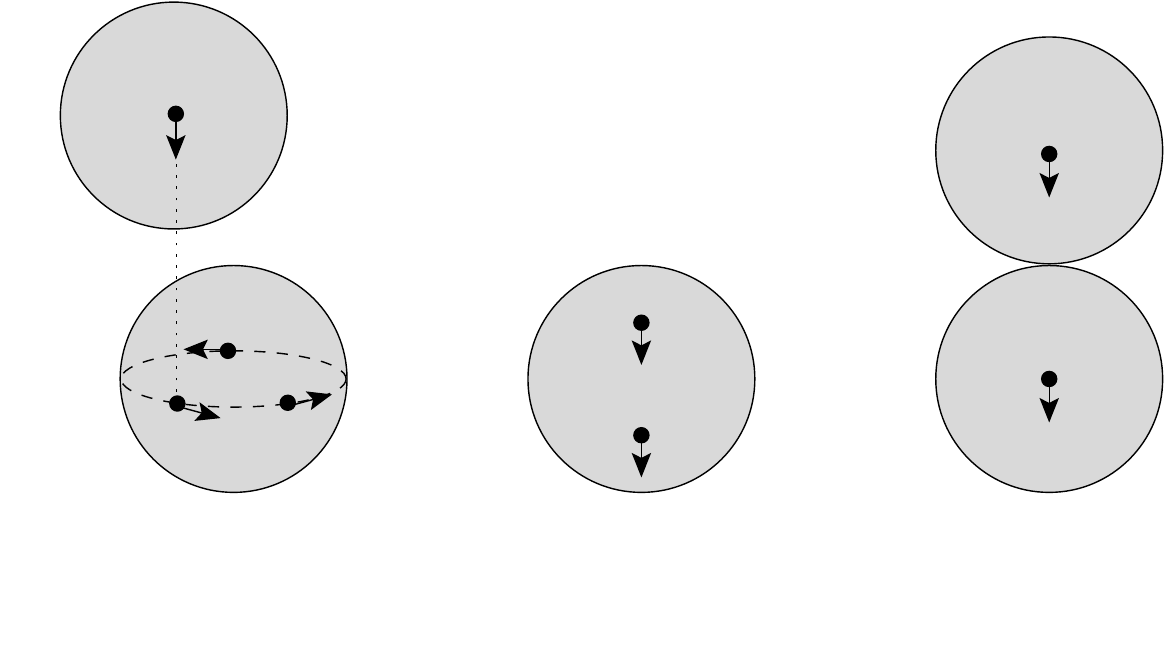}%
\end{picture}%
\setlength{\unitlength}{3552sp}%
\begingroup\makeatletter\ifx\SetFigFont\undefined%
\gdef\SetFigFont#1#2#3#4#5{%
  \reset@font\fontsize{#1}{#2pt}%
  \fontfamily{#3}\fontseries{#4}\fontshape{#5}%
  \selectfont}%
\fi\endgroup%
\begin{picture}(6209,3516)(136,-3269)
\put(1726,-2086){\makebox(0,0)[lb]{\smash{{\SetFigFont{11}{13.2}{\rmdefault}{\mddefault}{\updefault}{\color[rgb]{0,0,0}input}%
}}}}
\put(601,-2086){\makebox(0,0)[lb]{\smash{{\SetFigFont{11}{13.2}{\rmdefault}{\mddefault}{\updefault}{\color[rgb]{0,0,0}input}%
}}}}
\put(1126,-286){\makebox(0,0)[lb]{\smash{{\SetFigFont{11}{13.2}{\rmdefault}{\mddefault}{\updefault}{\color[rgb]{0,0,0}output}%
}}}}
\put(1082,-2746){\makebox(0,0)[lb]{\smash{{\SetFigFont{11}{13.2}{\rmdefault}{\mddefault}{\updefault}{\color[rgb]{0,0,0}$r=-\infty$}%
}}}}
\put(5507,-2746){\makebox(0,0)[lb]{\smash{{\SetFigFont{11}{13.2}{\rmdefault}{\mddefault}{\updefault}{\color[rgb]{0,0,0}$r=+\infty$}%
}}}}
\put(1374,-1527){\makebox(0,0)[lb]{\smash{{\SetFigFont{11}{13.2}{\rmdefault}{\mddefault}{\updefault}{\color[rgb]{0,0,0}output}%
}}}}
\put(2567,-2746){\makebox(0,0)[lb]{\smash{{\SetFigFont{11}{13.2}{\rmdefault}{\mddefault}{\updefault}{\color[rgb]{0,0,0}Marked points move apart as}%
}}}}
\put(2567,-2971){\makebox(0,0)[lb]{\smash{{\SetFigFont{11}{13.2}{\rmdefault}{\mddefault}{\updefault}{\color[rgb]{0,0,0}$r \rightarrow +\infty$, collide as}%
}}}}
\put(2567,-3196){\makebox(0,0)[lb]{\smash{{\SetFigFont{11}{13.2}{\rmdefault}{\mddefault}{\updefault}{\color[rgb]{0,0,0}$r \rightarrow -\infty$}%
}}}}
\put(5701,-1111){\makebox(0,0)[lb]{\smash{{\SetFigFont{11}{13.2}{\rmdefault}{\mddefault}{\updefault}{\color[rgb]{0,0,0}output}%
}}}}
\put(5701,-1411){\makebox(0,0)[lb]{\smash{{\SetFigFont{11}{13.2}{\rmdefault}{\mddefault}{\updefault}{\color[rgb]{0,0,0}input}%
}}}}
\put(5776,-511){\makebox(0,0)[lb]{\smash{{\SetFigFont{11}{13.2}{\rmdefault}{\mddefault}{\updefault}{\color[rgb]{0,0,0}input}%
}}}}
\put(5776,-1711){\makebox(0,0)[lb]{\smash{{\SetFigFont{11}{13.2}{\rmdefault}{\mddefault}{\updefault}{\color[rgb]{0,0,0}output}%
}}}}
\put(3601,-2011){\makebox(0,0)[lb]{\smash{{\SetFigFont{11}{13.2}{\rmdefault}{\mddefault}{\updefault}{\color[rgb]{0,0,0}output}%
}}}}
\put(3601,-1411){\makebox(0,0)[lb]{\smash{{\SetFigFont{11}{13.2}{\rmdefault}{\mddefault}{\updefault}{\color[rgb]{0,0,0}input}%
}}}}
\put(4726,-586){\makebox(0,0)[lb]{\smash{{\SetFigFont{11}{13.2}{\rmdefault}{\mddefault}{\updefault}{\color[rgb]{0,0,0}$\phi^{1,1}_L$}%
}}}}
\put(4726,-1711){\makebox(0,0)[lb]{\smash{{\SetFigFont{11}{13.2}{\rmdefault}{\mddefault}{\updefault}{\color[rgb]{0,0,0}$\check{\phi}^{1,1}_L$}%
}}}}
\put(526,-1711){\makebox(0,0)[lb]{\smash{{\SetFigFont{11}{13.2}{\rmdefault}{\mddefault}{\updefault}{\color[rgb]{0,0,0}$\smile$}%
}}}}
\put(151,-211){\makebox(0,0)[lb]{\smash{{\SetFigFont{11}{13.2}{\rmdefault}{\mddefault}{\updefault}{\color[rgb]{0,0,0}$\phi^{1,0}_L$}%
}}}}
\end{picture}%
\caption{\label{fig:phi20}}
\end{centering}
\end{figure}%

The map $\phi^{1,2}$ is maybe more familiar. It is the second component (after $\phi^{1,1}$) of a chain map, first mentioned in \cite{seidel02}, from $\mathit{CF}^*(\lambda)$ to the Hochschild cochain complex of the Fukaya category of $M$. As such, if $x$ is a cocycle then $\phi^{1,2}(x,\cdot)$ expresses the ``centrality'' of the elements $\phi^{1,1}(x)$. The underlying family of Riemann surfaces is again parametrized by $R = \bR$. Each $S_r$ is a disc with two boundary punctures and one interior puncture. We prefer to think of it as $S_r = (\bR \times [0,1]) \setminus \{(0,t_r)\}$, where the position of the interior puncture satisfies $t_r \rightarrow 0$ for $r \rightarrow -\infty$ and $t_r \rightarrow 1$ for $r \rightarrow +\infty$. The preferred tangent direction at that puncture is vertical up for $r \ll 0$, vertical down for $r \gg 0$, and rotates anticlockwise by $\pi$ between those two extremes. The one-forms $\nu_{S_r}$ are again closed. As $r \rightarrow \pm\infty$, the surfaces split into two components: one of them is a disc with three boundary punctures (with the same data used to define $\mu^2_{L_0,L_1,L_1}$ respectively $\mu^2_{L_0,L_0,L_1}$), and the other one is a disc with one interior and one boundary puncture (exactly as in $\phi^{1,1}_{L_0}$ respectively $\phi^{1,1}_{L_1}$); see Figure \ref{fig:phi12}.
\begin{figure}
\begin{centering}
\begin{picture}(0,0)%
\includegraphics{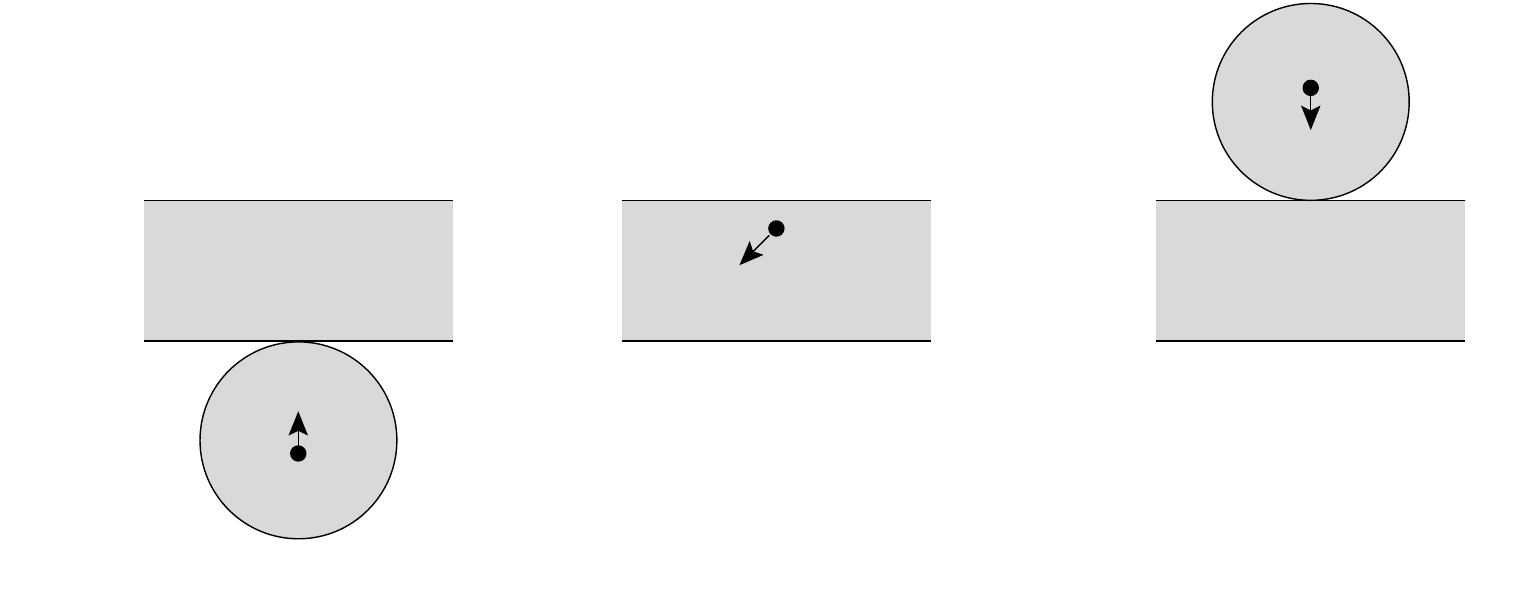}%
\end{picture}%
\setlength{\unitlength}{3552sp}%
\begingroup\makeatletter\ifx\SetFigFont\undefined%
\gdef\SetFigFont#1#2#3#4#5{%
  \reset@font\fontsize{#1}{#2pt}%
  \fontfamily{#3}\fontseries{#4}\fontshape{#5}%
  \selectfont}%
\fi\endgroup%
\begin{picture}(8064,3231)(-239,-2684)
\put(6376,-1636){\makebox(0,0)[lb]{\smash{{\SetFigFont{11}{13.2}{\rmdefault}{\mddefault}{\updefault}{\color[rgb]{0,0,0}$r=+\infty$}%
}}}}
\put(5701,-961){\makebox(0,0)[lb]{\smash{{\SetFigFont{11}{13.2}{\rmdefault}{\mddefault}{\updefault}{\color[rgb]{0,0,0}output}%
}}}}
\put(7426,-961){\makebox(0,0)[lb]{\smash{{\SetFigFont{11}{13.2}{\rmdefault}{\mddefault}{\updefault}{\color[rgb]{0,0,0}input}%
}}}}
\put(6451,164){\makebox(0,0)[lb]{\smash{{\SetFigFont{11}{13.2}{\rmdefault}{\mddefault}{\updefault}{\color[rgb]{0,0,0}input}%
}}}}
\put(6751,-436){\makebox(0,0)[lb]{\smash{{\SetFigFont{11}{13.2}{\rmdefault}{\mddefault}{\updefault}{\color[rgb]{0,0,0}output}%
}}}}
\put(6751,-736){\makebox(0,0)[lb]{\smash{{\SetFigFont{11}{13.2}{\rmdefault}{\mddefault}{\updefault}{\color[rgb]{0,0,0}input}%
}}}}
\put(5626,-61){\makebox(0,0)[lb]{\smash{{\SetFigFont{11}{13.2}{\rmdefault}{\mddefault}{\updefault}{\color[rgb]{0,0,0}$\phi^{1,1}_{L_1}$}%
}}}}
\put(5176,-736){\makebox(0,0)[lb]{\smash{{\SetFigFont{11}{13.2}{\rmdefault}{\mddefault}{\updefault}{\color[rgb]{0,0,0}$\mu^2_{L_0,L_1,L_1}$}%
}}}}
\put(2626,-2011){\makebox(0,0)[lb]{\smash{{\SetFigFont{11}{13.2}{\rmdefault}{\mddefault}{\updefault}{\color[rgb]{0,0,0}moves up, and the tangent direction}%
}}}}
\put(2626,-2236){\makebox(0,0)[lb]{\smash{{\SetFigFont{11}{13.2}{\rmdefault}{\mddefault}{\updefault}{\color[rgb]{0,0,0}rotates anticlockwise}%
}}}}
\put(2626,-1786){\makebox(0,0)[lb]{\smash{{\SetFigFont{11}{13.2}{\rmdefault}{\mddefault}{\updefault}{\color[rgb]{0,0,0}as $r$ increases, the marked point}%
}}}}
\put(3676,-436){\makebox(0,0)[lb]{\smash{{\SetFigFont{11}{13.2}{\rmdefault}{\mddefault}{\updefault}{\color[rgb]{0,0,0}$L_1$}%
}}}}
\put(3676,-1486){\makebox(0,0)[lb]{\smash{{\SetFigFont{11}{13.2}{\rmdefault}{\mddefault}{\updefault}{\color[rgb]{0,0,0}$L_0$}%
}}}}
\put(2701,-961){\makebox(0,0)[lb]{\smash{{\SetFigFont{11}{13.2}{\rmdefault}{\mddefault}{\updefault}{\color[rgb]{0,0,0}output}%
}}}}
\put(1351,-1186){\makebox(0,0)[lb]{\smash{{\SetFigFont{11}{13.2}{\rmdefault}{\mddefault}{\updefault}{\color[rgb]{0,0,0}input}%
}}}}
\put(2026,-961){\makebox(0,0)[lb]{\smash{{\SetFigFont{11}{13.2}{\rmdefault}{\mddefault}{\updefault}{\color[rgb]{0,0,0}input}%
}}}}
\put(-224,-736){\makebox(0,0)[lb]{\smash{{\SetFigFont{11}{13.2}{\rmdefault}{\mddefault}{\updefault}{\color[rgb]{0,0,0}$\mu^2_{L_0,L_0,L_1}$}%
}}}}
\put(151,-961){\makebox(0,0)[lb]{\smash{{\SetFigFont{11}{13.2}{\rmdefault}{\mddefault}{\updefault}{\color[rgb]{0,0,0}output}%
}}}}
\put(4576,-961){\makebox(0,0)[lb]{\smash{{\SetFigFont{11}{13.2}{\rmdefault}{\mddefault}{\updefault}{\color[rgb]{0,0,0}input}%
}}}}
\put(901,-2611){\makebox(0,0)[lb]{\smash{{\SetFigFont{11}{13.2}{\rmdefault}{\mddefault}{\updefault}{\color[rgb]{0,0,0}$r=-\infty$}%
}}}}
\put(1426,-2011){\makebox(0,0)[lb]{\smash{{\SetFigFont{11}{13.2}{\rmdefault}{\mddefault}{\updefault}{\color[rgb]{0,0,0}input}%
}}}}
\put(1351,-1486){\makebox(0,0)[lb]{\smash{{\SetFigFont{11}{13.2}{\rmdefault}{\mddefault}{\updefault}{\color[rgb]{0,0,0}output}%
}}}}
\put(226,-1861){\makebox(0,0)[lb]{\smash{{\SetFigFont{11}{13.2}{\rmdefault}{\mddefault}{\updefault}{\color[rgb]{0,0,0}$\phi^{1,1}_{L_0}$}%
}}}}
\end{picture}%
\caption{\label{fig:phi12}}
\end{centering}
\end{figure}%

\subsection{Cardy relations}
Next, we consider operations induced by annuli. Those for a single annulus are not particularly interesting, since they can be decomposed into products and open-closed string maps up to chain homotopy. Instead, we want to look at a one-parameter family of annuli, which leads to maps
\begin{align} \label{eq:psi}
& \psi^{0,1}_{L_0,L_1}: \mathit{CF}^1(L_1,L_1) \longrightarrow \bK, \\ 
\label{eq:dual-psi}
& \check{\psi}^{0,1}_{L_0,L_1}: \mathit{CF}^1(L_0,L_0) \longrightarrow \bK,
\end{align}
satisfying
\begin{align} \label{eq:phi-check-phi}
& \psi^{0,1}_{L_0,L_1}(\mu^1(a)) = (-1)^{n(n-1)/2} \mathrm{Str}(\mu^2_{L_0,L_1,L_1}(a,\cdot)) - (-1)^n \langle \phi^{1,0}_{L_0}, \check{\phi}^{1,1}_{L_1}(a) \rangle, \\ \label{eq:phi-check-phi-2}
& \check{\psi}^{0,1}_{L_0,L_1}(\mu^1(a)) = (-1)^{n(n-1)/2} \mathrm{Str}(\mu^2_{L_0,L_0,L_1}(\cdot,a)) - \langle \phi^{1,0}_{L_1}, \check{\phi}^{1,1}_{L_0}(a) \rangle.
\end{align}
In the last term on the right hand side of \eqref{eq:phi-check-phi}, we take $\phi^{1,0}_{L_0} \in \mathit{CF}^n(-\lambda)$ for some $\lambda \in \bR$, and pair it with the result of applying $\check{\phi}^{1,1}_{L_1}: \mathit{CF}^*(L_1,L_1) \rightarrow \mathit{CF}^{*+n}(\lambda)$. Note that even though $\lambda$ does not appear in the notation \eqref{eq:psi}, the definition of $\psi^{0,1}_{L_0,L_1}$ depends on it. The same applies to \eqref{eq:dual-psi}. 

\begin{remark}
As a heuristic check on the signs, let $a \in \mathit{CF}^0(L_1,L_1)$ be a cocycle representing the identity. Then the supertrace of multiplication is the Euler characteristic of Floer cohomology, and on the other hand, $[\check{\phi}^{1,1}_{L_k}(a)] = [\phi^{1,0}_{L_k}]$. Vanishing of the right hand side of \eqref{eq:phi-check-phi} then recovers \eqref{eq:euler}. A similar argument applies to \eqref{eq:phi-check-phi-2}.
\end{remark}

\begin{remark}
The two operations are related in the following way. Suppose that $L_0 \neq L_1$. One can then arrange that $\mathit{CF}^*(L_0,L_1)$ and $\mathit{CF}^*(L_1,L_0)$ are strictly dual, see Remark \ref{th:fail}. Moreover, one can arrange that a limited analogue of \eqref{eq:cyclic-symmetry-2} holds for the triangle product, namely
\begin{multline}
\langle a_3, \mu^2_{L_0,L_0,L_1}(a_2,a_1) \rangle = (-1)^{|a_1|(n-|a_1|)} \langle \mu^2_{L_1,L_0,L_0}(a_1,a_3), a_2 \rangle: \\
\mathit{CF}^*(L_1,L_0) \otimes \mathit{CF}^*(L_0,L_1) \otimes \mathit{CF}^*(L_0,L_0) \longrightarrow \bK.
\end{multline}
This implies that for any $a \in \mathit{CF}^*(L_0,L_0)$, 
\begin{equation}
\mathrm{Str}(\mu^2_{L_1,L_0,L_0}(a,\cdot)) = (-1)^n \mathrm{Str}(\mu^2_{L_0,L_0,L_1}(\cdot,a)).
\end{equation}
Once one has arranged that, it follows that $(-1)^n \psi^{0,1}_{L_1,L_0} - \check{\psi}^{0,1}_{L_0,L_1}$ is a chain map. One can show that it is actually nullhomotopic.
\end{remark}

$\psi^{0,1}_{L_0,L_1}$ is defined using a family $(S_r)$ of Riemann surfaces parametrized by $R = \bR$. Each $S_r$ is an annulus with the two boundary sides labeled $(L_0,L_1)$. The $L_1$ side carries a boundary puncture, and the one-form $\nu_{S_r}$ is closed. As usual in Cardy-type relations, the conformal structure of the annulus varies with $r$. As $r \rightarrow -\infty$, we have $\bar{S}_r \iso [0,l_r] \times S^1$ with $l_r \rightarrow \infty$. In the limit, we get a degeneration with two punctured disc components. We equip those components with the structures used to define $\phi^{1,0}_{L_0}$ and $\check{\phi}^{1,1}_{L_1}$, respectively. In particular, these surfaces carry nontrivial closed one-forms, and therefore $\nu_{S_r}$ is still nontrivial for $r \ll 0$. As $r \rightarrow +\infty$, $S_r$ degenerates into a disc with three boundary punctures, corresponding to $\mu^2$, and that disc is glued to itself by matching the first input to the output (see Figure \ref{fig:psi01}). The idea for $\check{\psi}^{0,1}_{L_0,L_1}$ is the same: we will not write down the details, but see Figure \ref{fig:psi01-check}.

\begin{figure}
\begin{centering}
\begin{picture}(0,0)%
\includegraphics{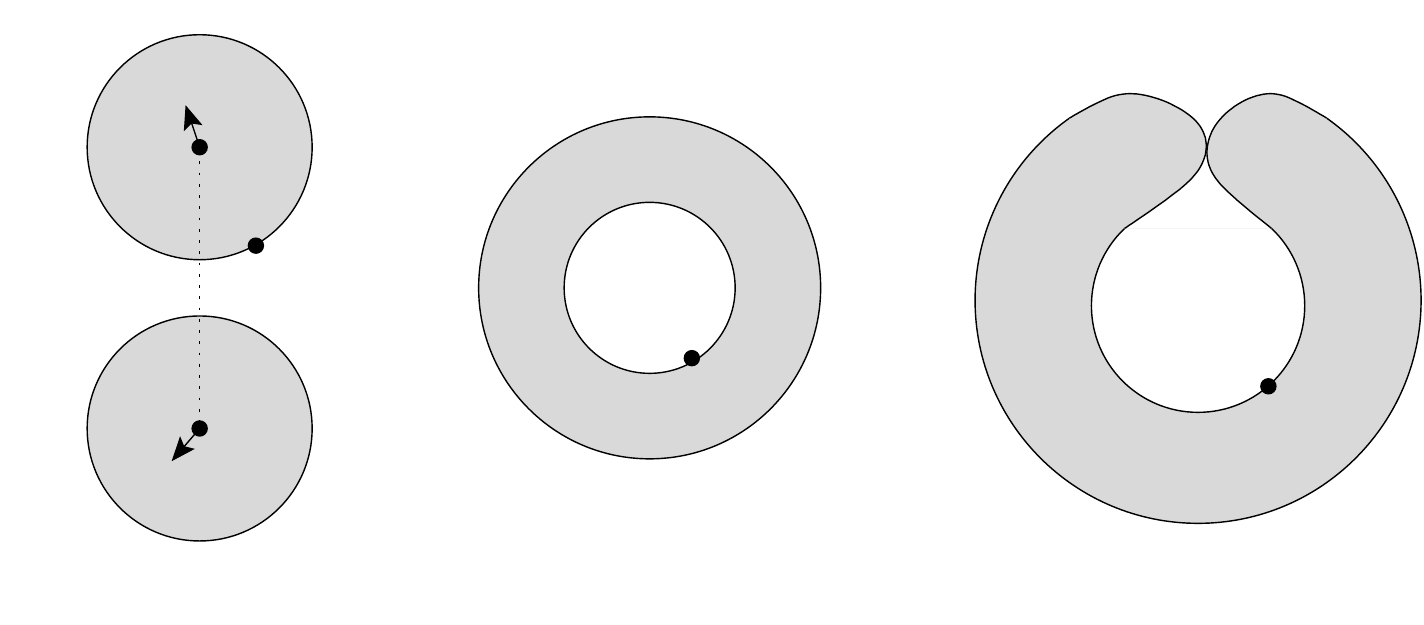}%
\end{picture}%
\setlength{\unitlength}{3552sp}%
\begingroup\makeatletter\ifx\SetFigFont\undefined%
\gdef\SetFigFont#1#2#3#4#5{%
  \reset@font\fontsize{#1}{#2pt}%
  \fontfamily{#3}\fontseries{#4}\fontshape{#5}%
  \selectfont}%
\fi\endgroup%
\begin{picture}(7588,3319)(436,-2909)
\put(4501,-2011){\makebox(0,0)[lb]{\smash{{\SetFigFont{11}{13.2}{\rmdefault}{\mddefault}{\updefault}{\color[rgb]{0,0,0}$L_0$}%
}}}}
\put(3526,-1111){\makebox(0,0)[lb]{\smash{{\SetFigFont{11}{13.2}{\rmdefault}{\mddefault}{\updefault}{\color[rgb]{0,0,0}$L_1$}%
}}}}
\put(1201,-2836){\makebox(0,0)[lb]{\smash{{\SetFigFont{11}{13.2}{\rmdefault}{\mddefault}{\updefault}{\color[rgb]{0,0,0}$r=-\infty$}%
}}}}
\put(1576,-1786){\makebox(0,0)[lb]{\smash{{\SetFigFont{11}{13.2}{\rmdefault}{\mddefault}{\updefault}{\color[rgb]{0,0,0}output}%
}}}}
\put(1576,-286){\makebox(0,0)[lb]{\smash{{\SetFigFont{11}{13.2}{\rmdefault}{\mddefault}{\updefault}{\color[rgb]{0,0,0}output}%
}}}}
\put(1876,-1036){\makebox(0,0)[lb]{\smash{{\SetFigFont{11}{13.2}{\rmdefault}{\mddefault}{\updefault}{\color[rgb]{0,0,0}input}%
}}}}
\put(4201,-1636){\makebox(0,0)[lb]{\smash{{\SetFigFont{11}{13.2}{\rmdefault}{\mddefault}{\updefault}{\color[rgb]{0,0,0}input}%
}}}}
\put(1801,239){\makebox(0,0)[lb]{\smash{{\SetFigFont{11}{13.2}{\rmdefault}{\mddefault}{\updefault}{\color[rgb]{0,0,0}$L_1$}%
}}}}
\put(1951,-1411){\makebox(0,0)[lb]{\smash{{\SetFigFont{11}{13.2}{\rmdefault}{\mddefault}{\updefault}{\color[rgb]{0,0,0}$L_0$}%
}}}}
\put(451,-2311){\makebox(0,0)[lb]{\smash{{\SetFigFont{11}{13.2}{\rmdefault}{\mddefault}{\updefault}{\color[rgb]{0,0,0}$\phi^{1,0}_{L_0}$}%
}}}}
\put(6526,-2836){\makebox(0,0)[lb]{\smash{{\SetFigFont{11}{13.2}{\rmdefault}{\mddefault}{\updefault}{\color[rgb]{0,0,0}$r=+\infty$}%
}}}}
\put(6751,-1561){\makebox(0,0)[lb]{\smash{{\SetFigFont{11}{13.2}{\rmdefault}{\mddefault}{\updefault}{\color[rgb]{0,0,0}input}%
}}}}
\put(6976,-436){\makebox(0,0)[lb]{\smash{{\SetFigFont{11}{13.2}{\rmdefault}{\mddefault}{\updefault}{\color[rgb]{0,0,0}input}%
}}}}
\put(6226,-436){\makebox(0,0)[lb]{\smash{{\SetFigFont{11}{13.2}{\rmdefault}{\mddefault}{\updefault}{\color[rgb]{0,0,0}output}%
}}}}
\put(5176,-2086){\makebox(0,0)[lb]{\smash{{\SetFigFont{11}{13.2}{\rmdefault}{\mddefault}{\updefault}{\color[rgb]{0,0,0}$\mu^2_{L_0,L_1,L_1}$}%
}}}}
\put(7126,-1186){\makebox(0,0)[lb]{\smash{{\SetFigFont{11}{13.2}{\rmdefault}{\mddefault}{\updefault}{\color[rgb]{0,0,0}$L_1$}%
}}}}
\put(7951,-1936){\makebox(0,0)[lb]{\smash{{\SetFigFont{11}{13.2}{\rmdefault}{\mddefault}{\updefault}{\color[rgb]{0,0,0}$L_0$}%
}}}}
\put(451,-811){\makebox(0,0)[lb]{\smash{{\SetFigFont{11}{13.2}{\rmdefault}{\mddefault}{\updefault}{\color[rgb]{0,0,0}$\check{\phi}^{1,1}_{L_1}$}%
}}}}
\end{picture}%
\caption{\label{fig:psi01}}
\end{centering}
\end{figure}

\begin{figure}
\begin{centering}
\begin{picture}(0,0)%
\includegraphics{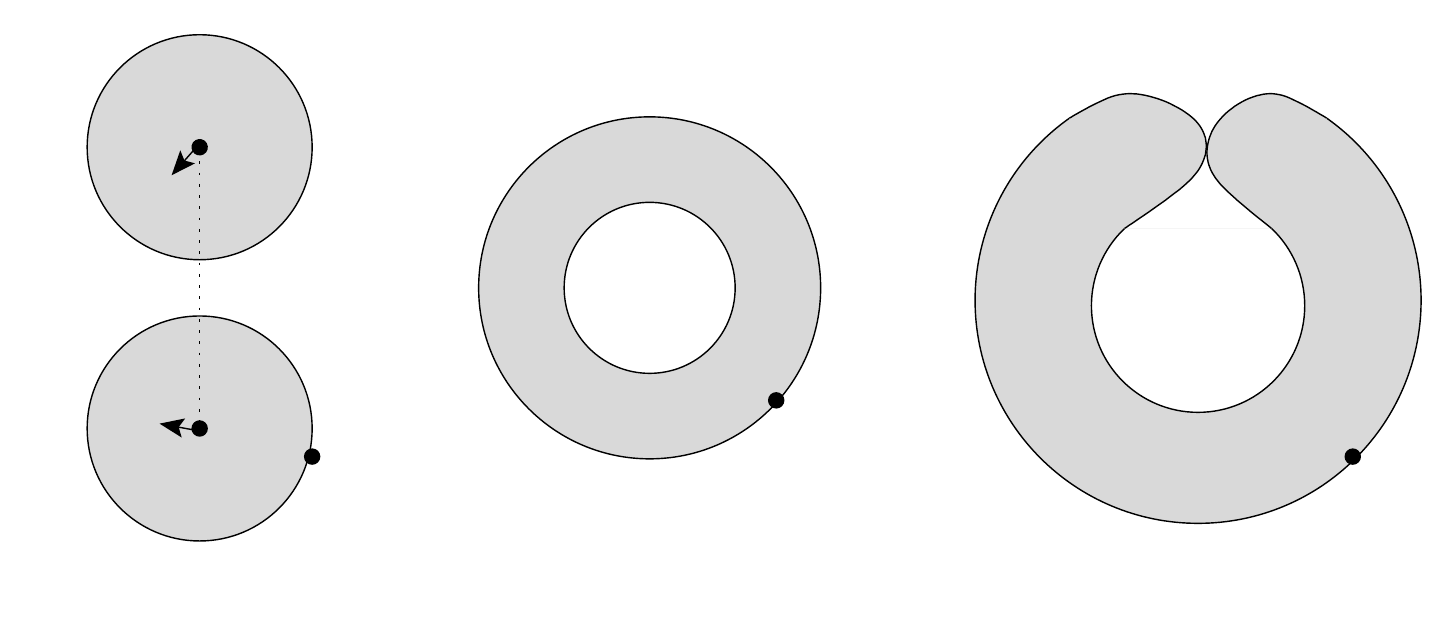}%
\end{picture}%
\setlength{\unitlength}{3552sp}%
\begingroup\makeatletter\ifx\SetFigFont\undefined%
\gdef\SetFigFont#1#2#3#4#5{%
  \reset@font\fontsize{#1}{#2pt}%
  \fontfamily{#3}\fontseries{#4}\fontshape{#5}%
  \selectfont}%
\fi\endgroup%
\begin{picture}(7614,3319)(436,-2909)
\put(3526,-1111){\makebox(0,0)[lb]{\smash{{\SetFigFont{11}{13.2}{\rmdefault}{\mddefault}{\updefault}{\color[rgb]{0,0,0}$L_1$}%
}}}}
\put(1201,-2836){\makebox(0,0)[lb]{\smash{{\SetFigFont{11}{13.2}{\rmdefault}{\mddefault}{\updefault}{\color[rgb]{0,0,0}$r=-\infty$}%
}}}}
\put(1576,-1786){\makebox(0,0)[lb]{\smash{{\SetFigFont{11}{13.2}{\rmdefault}{\mddefault}{\updefault}{\color[rgb]{0,0,0}output}%
}}}}
\put(1576,-286){\makebox(0,0)[lb]{\smash{{\SetFigFont{11}{13.2}{\rmdefault}{\mddefault}{\updefault}{\color[rgb]{0,0,0}output}%
}}}}
\put(1801,239){\makebox(0,0)[lb]{\smash{{\SetFigFont{11}{13.2}{\rmdefault}{\mddefault}{\updefault}{\color[rgb]{0,0,0}$L_1$}%
}}}}
\put(1951,-1411){\makebox(0,0)[lb]{\smash{{\SetFigFont{11}{13.2}{\rmdefault}{\mddefault}{\updefault}{\color[rgb]{0,0,0}$L_0$}%
}}}}
\put(6526,-2836){\makebox(0,0)[lb]{\smash{{\SetFigFont{11}{13.2}{\rmdefault}{\mddefault}{\updefault}{\color[rgb]{0,0,0}$r=+\infty$}%
}}}}
\put(6976,-436){\makebox(0,0)[lb]{\smash{{\SetFigFont{11}{13.2}{\rmdefault}{\mddefault}{\updefault}{\color[rgb]{0,0,0}input}%
}}}}
\put(6226,-436){\makebox(0,0)[lb]{\smash{{\SetFigFont{11}{13.2}{\rmdefault}{\mddefault}{\updefault}{\color[rgb]{0,0,0}output}%
}}}}
\put(7126,-1186){\makebox(0,0)[lb]{\smash{{\SetFigFont{11}{13.2}{\rmdefault}{\mddefault}{\updefault}{\color[rgb]{0,0,0}$L_1$}%
}}}}
\put(7951,-1936){\makebox(0,0)[lb]{\smash{{\SetFigFont{11}{13.2}{\rmdefault}{\mddefault}{\updefault}{\color[rgb]{0,0,0}$L_0$}%
}}}}
\put(2176,-2161){\makebox(0,0)[lb]{\smash{{\SetFigFont{11}{13.2}{\rmdefault}{\mddefault}{\updefault}{\color[rgb]{0,0,0}input}%
}}}}
\put(451,-2311){\makebox(0,0)[lb]{\smash{{\SetFigFont{11}{13.2}{\rmdefault}{\mddefault}{\updefault}{\color[rgb]{0,0,0}$\check{\phi}^{1,1}_{L_0}$}%
}}}}
\put(451,-811){\makebox(0,0)[lb]{\smash{{\SetFigFont{11}{13.2}{\rmdefault}{\mddefault}{\updefault}{\color[rgb]{0,0,0}$\phi^{1,0}_{L_1}$}%
}}}}
\put(5176,-2086){\makebox(0,0)[lb]{\smash{{\SetFigFont{11}{13.2}{\rmdefault}{\mddefault}{\updefault}{\color[rgb]{0,0,0}$\mu^2_{L_0,L_0,L_1}$}%
}}}}
\put(7651,-2236){\makebox(0,0)[lb]{\smash{{\SetFigFont{11}{13.2}{\rmdefault}{\mddefault}{\updefault}{\color[rgb]{0,0,0}input}%
}}}}
\put(4576,-436){\makebox(0,0)[lb]{\smash{{\SetFigFont{11}{13.2}{\rmdefault}{\mddefault}{\updefault}{\color[rgb]{0,0,0}$L_0$}%
}}}}
\put(4651,-1861){\makebox(0,0)[lb]{\smash{{\SetFigFont{11}{13.2}{\rmdefault}{\mddefault}{\updefault}{\color[rgb]{0,0,0}input}%
}}}}
\end{picture}%
\caption{\label{fig:psi01-check}}
\end{centering}
\end{figure}

\subsection{A higher relation\label{subsec:hexagon}}
We continue with $L_0,L_1$ as before. Define a map 
\begin{equation} \label{eq:h-map}
\mathit{CF}^1(2\mu) \longrightarrow \bK
\end{equation}
as the sum of the following six expressions:
\begin{equation} \label{eq:6-expressions}
\begin{aligned}
& \text{(i)} && \mathit{CF}^1(2\mu) \xrightarrow{\phi^{2,0}_{L_1}} \mathit{CF}^n(\mu) \xrightarrow{\langle \phi^{1,0}_{L_0}, \cdot \rangle} \bK, \\
& \text{(ii)} && \mathit{CF}^1(2\mu) \xrightarrow{\phi^{1,1}_{L_1}} \mathit{CF}^1(L_1,L_1) \xrightarrow{\psi^{0,1}_{L_0,L_1}} \bK, \\
& \text{(iii)} && x \longmapsto \mathrm{Str}(\phi^{1,2}_{L_0,L_1}(x,\cdot)) \quad \text{multiplied by $(-1)^{n(n-1)/2+1},$} \\
& \text{(iv)} && \mathit{CF}^1(2\mu) \xrightarrow{\phi^{1,1}_{L_0}} \mathit{CF}^1(L_0,L_0) \xrightarrow{\check{\psi}^{0,1}_{L_0,L_1}} \bK \quad \text{multiplied by $(-1)$}, \\
& \text{(v)} && \mathit{CF}^1(2\mu) \xrightarrow{\phi^{2,0}_{L_0}} \mathit{CF}^n(\mu) \xrightarrow{\langle \phi^{1,0}_{L_1},\cdot \rangle} \bK \quad \text{multiplied by $(-1)^{n+1}$,} \\ 
& \text{(vi)} && \langle \cdot, \phi^{1,0}_{L_0} \ast \phi^{1,0}_{L_1} \rangle \quad \text{multiplied by $(-1)^n$.}
\end{aligned}
\end{equation}
These map $dx$ to, respectively
\begin{equation}
\begin{aligned}
& \text{(i)} && (-1)^n \langle \phi^{1,0}_{L_0}, \check{\phi}^{1,1}_{L_1}(\phi^{1,1}_{L_1}(x)) \rangle - \langle \phi^{1,0}_{L_1} \smile \phi^{1,0}_{L_0}, x \rangle, \\
& \text{(ii)} && (-1)^{n+1} \langle \phi^{1,0}_{L_0}, \check{\phi}^{1,1}_{L_1}(\phi^{1,1}_{L_1}(x)) \rangle + (-1)^{n(n-1)/2} \mathrm{Str}(\mu^2_{L_0,L_1,L_1}(\phi^{1,1}_{L_1}(x),\cdot)), \\
& \text{(iii)} && (-1)^{n(n-1)/2} \mathrm{Str}(\mu^2_{L_0,L_0,L_1}(\cdot,\phi^{1,1}_{L_0}(x))) - (-1)^{n(n-1)/2} \mathrm{Str}(\mu^2_{L_0,L_1,L_1}(\phi^{1,1}_{L_1}(x),\cdot)), \\
& \text{(iv)} && \langle \phi^{1,0}_{L_1}, \check{\phi}^{1,1}_{L_0}(\phi^{1,1}_{L_0}(x)) \rangle - (-1)^{n(n-1)/2} 
\mathrm{Str}(\mu^2_{L_0,L_0,L_1}(\cdot,\phi^{1,1}_{L_0}(x))), \\
& \text{(v)} && (-1)^n \langle \phi^{1,0}_{L_0} \smile \phi^{1,0}_{L_1}, x \rangle - \langle \phi^{1,0}_{L_1}, \check{\phi}^{1,1}_{L_0}(\phi^{1,1}_{L_0}(x)) \rangle, \\
& \text{(vi)} && \langle x, \phi^{1,0}_{L_1} \smile \phi^{1,0}_{L_0} \rangle + (-1)^{n+1} \langle x, \phi^{1,0}_{L_0} \smile \phi^{1,0}_{L_1} \rangle.
\end{aligned}
\end{equation}
Hence their sum maps $dx$ to zero.

\begin{proposition} \label{th:hexagon}
\eqref{eq:h-map} is nullhomotopic.
\end{proposition}

Unsurprisingly, the proof of this involves constructing a two-parameter family of Riemann surfaces, which are annuli with one interior puncture (rather than a boundary puncture as in $\psi^{0,1}$, $\check{\psi}^{0,1}$). The compactified parameter space is a hexagon, whose sides correspond to the six expressions above. We have represented the degenerations associated to the six sides graphically in Figure \ref{fig:hex-sides}, and those that happen at the corners in Figure \ref{fig:hex-corners}. In principle, the construction of such a family is not hard: suppose that we go around the boundary of the hexagon and glue together all the components in Figure \ref{fig:hex-sides}. This yields a family of annuli over with an interior puncture and tangent direction at that puncture. Identify each such annulus topologically with $S^1 \times [0,1]$, so that the interior puncture is at $(0,1/2)$ (this identification is unique up to isotopy). Then, as we go around the boundary of the hexagon, the total rotation number of the tangent direction is zero (we have a rotation by $\pi$ in (vi) and by $-\pi$ in (iii), because for the latter the boundary orientation is opposite to the one chosen for the parameter space of $\phi^{1,2}$). This allows one to fill in the family over the interior of the hexagon. There are a few noteworthy points, in particular concerning the boundary side (vi), which unlike all others consists of surfaces with three components. We postpone further discussion of this to Section \ref{subsec:ksv}.

\begin{figure}
\begin{centering}
\begin{picture}(0,0)%
\includegraphics{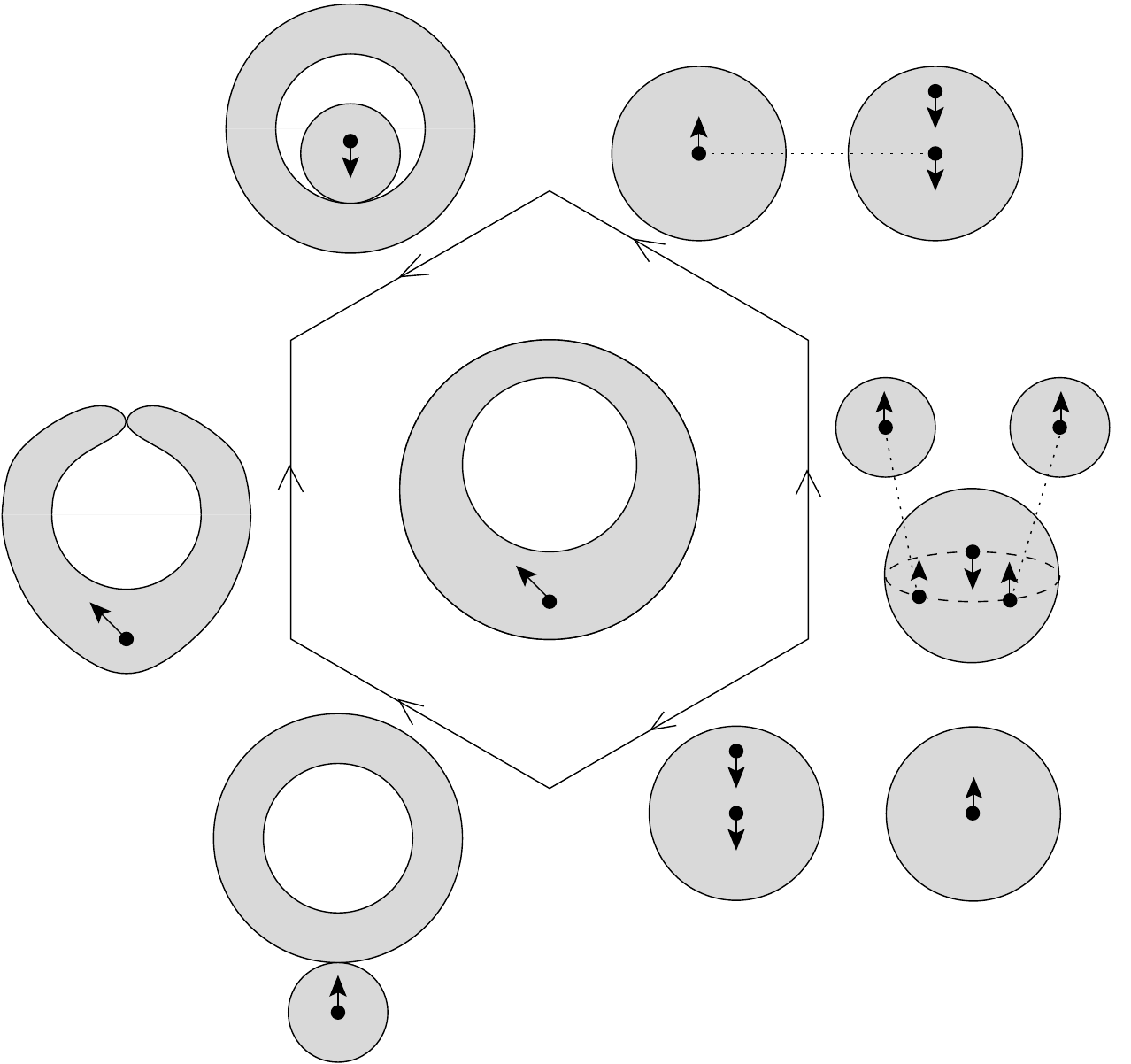}%
\end{picture}%
\setlength{\unitlength}{3552sp}%
\begingroup\makeatletter\ifx\SetFigFont\undefined%
\gdef\SetFigFont#1#2#3#4#5{%
  \reset@font\fontsize{#1}{#2pt}%
  \fontfamily{#3}\fontseries{#4}\fontshape{#5}%
  \selectfont}%
\fi\endgroup%
\begin{picture}(6814,6391)(664,-6368)
\put(2551,-2986){\makebox(0,0)[lb]{\smash{{\SetFigFont{11}{13.2}{\rmdefault}{\mddefault}{\updefault}{\color[rgb]{0,0,0}(iii)}%
}}}}
\put(3226,-5761){\makebox(0,0)[lb]{\smash{{\SetFigFont{11}{13.2}{\rmdefault}{\mddefault}{\updefault}{\color[rgb]{0,0,0}$L_0$}%
}}}}
\put(2551,-5386){\makebox(0,0)[lb]{\smash{{\SetFigFont{11}{13.2}{\rmdefault}{\mddefault}{\updefault}{\color[rgb]{0,0,0}$L_1$}%
}}}}
\put(1313,-3399){\makebox(0,0)[lb]{\smash{{\SetFigFont{11}{13.2}{\rmdefault}{\mddefault}{\updefault}{\color[rgb]{0,0,0}$L_1$}%
}}}}
\put(1801,-2386){\makebox(0,0)[lb]{\smash{{\SetFigFont{11}{13.2}{\rmdefault}{\mddefault}{\updefault}{\color[rgb]{0,0,0}$L_0$}%
}}}}
\put(2551,-511){\makebox(0,0)[lb]{\smash{{\SetFigFont{11}{13.2}{\rmdefault}{\mddefault}{\updefault}{\color[rgb]{0,0,0}$L_1$}%
}}}}
\put(3376,-211){\makebox(0,0)[lb]{\smash{{\SetFigFont{11}{13.2}{\rmdefault}{\mddefault}{\updefault}{\color[rgb]{0,0,0}$L_0$}%
}}}}
\put(5754,-4411){\makebox(0,0)[rb]{\smash{{\SetFigFont{11}{13.2}{\rmdefault}{\mddefault}{\updefault}{\color[rgb]{0,0,0}$L_0$}%
}}}}
\put(7088,-5484){\makebox(0,0)[rb]{\smash{{\SetFigFont{11}{13.2}{\rmdefault}{\mddefault}{\updefault}{\color[rgb]{0,0,0}$L_1$}%
}}}}
\put(5176,-2986){\makebox(0,0)[lb]{\smash{{\SetFigFont{11}{13.2}{\rmdefault}{\mddefault}{\updefault}{\color[rgb]{0,0,0}(vi)}%
}}}}
\put(3076,-4111){\makebox(0,0)[lb]{\smash{{\SetFigFont{11}{13.2}{\rmdefault}{\mddefault}{\updefault}{\color[rgb]{0,0,0}(iv)}%
}}}}
\put(4651,-4111){\makebox(0,0)[lb]{\smash{{\SetFigFont{11}{13.2}{\rmdefault}{\mddefault}{\updefault}{\color[rgb]{0,0,0}(v)}%
}}}}
\put(5101,-361){\makebox(0,0)[lb]{\smash{{\SetFigFont{11}{13.2}{\rmdefault}{\mddefault}{\updefault}{\color[rgb]{0,0,0}$L_0$}%
}}}}
\put(6601,-361){\makebox(0,0)[lb]{\smash{{\SetFigFont{11}{13.2}{\rmdefault}{\mddefault}{\updefault}{\color[rgb]{0,0,0}$L_1$}%
}}}}
\put(3076,-1861){\makebox(0,0)[lb]{\smash{{\SetFigFont{11}{13.2}{\rmdefault}{\mddefault}{\updefault}{\color[rgb]{0,0,0}(ii)}%
}}}}
\put(4726,-1861){\makebox(0,0)[lb]{\smash{{\SetFigFont{11}{13.2}{\rmdefault}{\mddefault}{\updefault}{\color[rgb]{0,0,0}(i)}%
}}}}
\put(6399,-2169){\makebox(0,0)[rb]{\smash{{\SetFigFont{11}{13.2}{\rmdefault}{\mddefault}{\updefault}{\color[rgb]{0,0,0}$L_1$}%
}}}}
\put(7463,-2169){\makebox(0,0)[rb]{\smash{{\SetFigFont{11}{13.2}{\rmdefault}{\mddefault}{\updefault}{\color[rgb]{0,0,0}$L_0$}%
}}}}
\end{picture}%

\caption{\label{fig:hex-sides}}
\end{centering}
\end{figure}%
\begin{figure}
\begin{centering}
\begin{picture}(0,0)%
\includegraphics{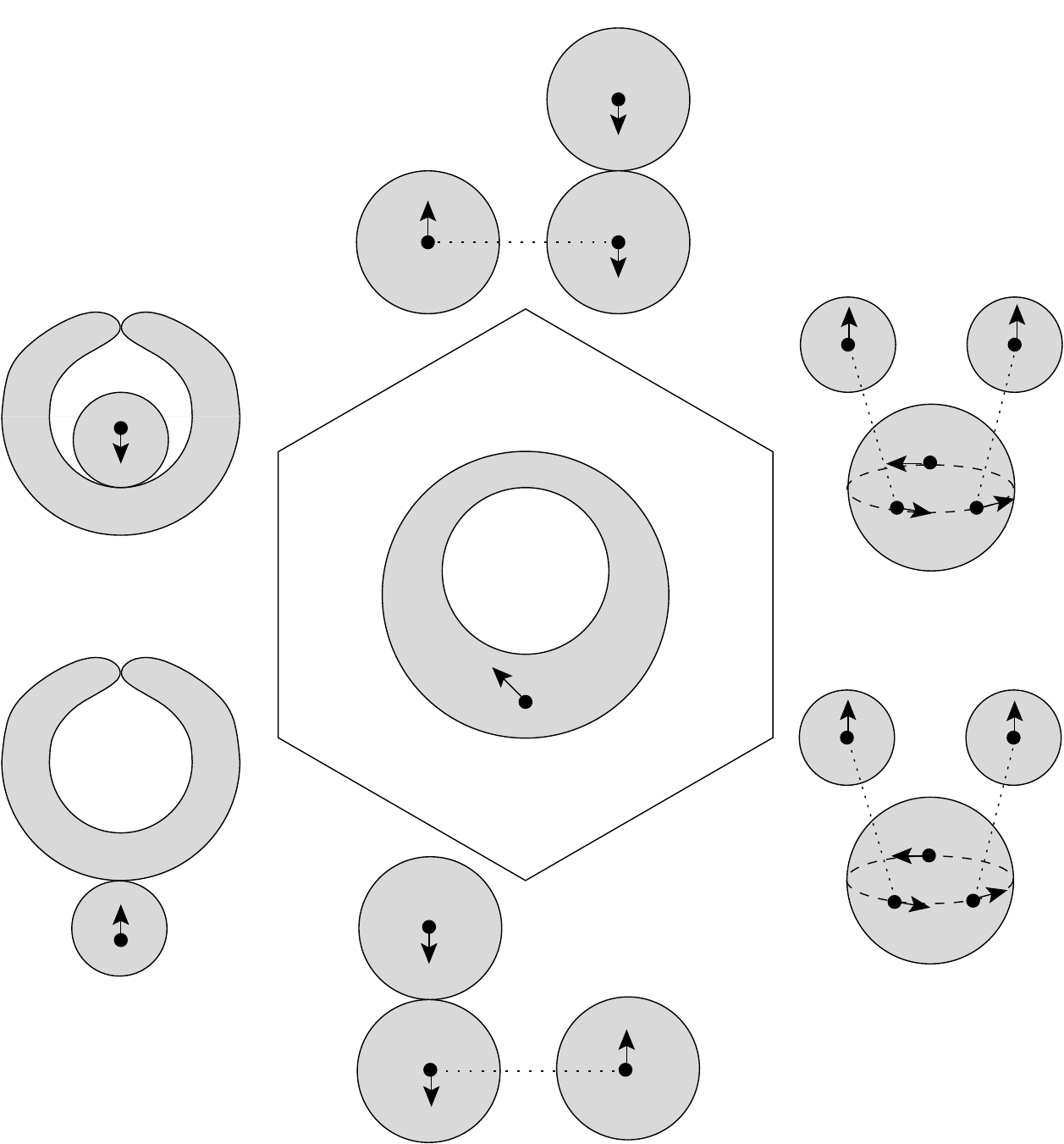}%
\end{picture}%
\setlength{\unitlength}{3552sp}%
\begingroup\makeatletter\ifx\SetFigFont\undefined%
\gdef\SetFigFont#1#2#3#4#5{%
  \reset@font\fontsize{#1}{#2pt}%
  \fontfamily{#3}\fontseries{#4}\fontshape{#5}%
  \selectfont}%
\fi\endgroup%
\begin{picture}(6702,7198)(1354,-6638)
\put(6548,-1231){\makebox(0,0)[lb]{\smash{{\SetFigFont{11}{13.2}{\rmdefault}{\mddefault}{\updefault}{\color[rgb]{0,0,0}$L_1$}%
}}}}
\put(7673,-1231){\makebox(0,0)[lb]{\smash{{\SetFigFont{11}{13.2}{\rmdefault}{\mddefault}{\updefault}{\color[rgb]{0,0,0}$L_0$}%
}}}}
\put(5543,-5671){\makebox(0,0)[lb]{\smash{{\SetFigFont{11}{13.2}{\rmdefault}{\mddefault}{\updefault}{\color[rgb]{0,0,0}$L_1$}%
}}}}
\put(7666,-3706){\makebox(0,0)[lb]{\smash{{\SetFigFont{11}{13.2}{\rmdefault}{\mddefault}{\updefault}{\color[rgb]{0,0,0}$L_1$}%
}}}}
\put(6541,-3706){\makebox(0,0)[lb]{\smash{{\SetFigFont{11}{13.2}{\rmdefault}{\mddefault}{\updefault}{\color[rgb]{0,0,0}$L_0$}%
}}}}
\put(2491,-3556){\makebox(0,0)[lb]{\smash{{\SetFigFont{11}{13.2}{\rmdefault}{\mddefault}{\updefault}{\color[rgb]{0,0,0}$L_0$}%
}}}}
\put(1966,-4569){\makebox(0,0)[lb]{\smash{{\SetFigFont{11}{13.2}{\rmdefault}{\mddefault}{\updefault}{\color[rgb]{0,0,0}$L_1$}%
}}}}
\put(2716,-1531){\makebox(0,0)[lb]{\smash{{\SetFigFont{11}{13.2}{\rmdefault}{\mddefault}{\updefault}{\color[rgb]{0,0,0}$L_0$}%
}}}}
\put(1891,-1831){\makebox(0,0)[lb]{\smash{{\SetFigFont{11}{13.2}{\rmdefault}{\mddefault}{\updefault}{\color[rgb]{0,0,0}$L_1$}%
}}}}
\put(3376,-661){\makebox(0,0)[lb]{\smash{{\SetFigFont{11}{13.2}{\rmdefault}{\mddefault}{\updefault}{\color[rgb]{0,0,0}$L_0$}%
}}}}
\put(5401,389){\makebox(0,0)[lb]{\smash{{\SetFigFont{11}{13.2}{\rmdefault}{\mddefault}{\updefault}{\color[rgb]{0,0,0}$L_1$}%
}}}}
\put(4613,-5288){\makebox(0,0)[lb]{\smash{{\SetFigFont{11}{13.2}{\rmdefault}{\mddefault}{\updefault}{\color[rgb]{0,0,0}$L_0$}%
}}}}
\end{picture}%
\caption{\label{fig:hex-corners}}
\end{centering}
\end{figure}%

\begin{remark}
The observation made in Remark \ref{th:caution} applies here as well. For instance, suppose that we replace part (vi) of \eqref{eq:6-expressions} with the a priori equally plausible 
\begin{equation} \label{eq:plausible-but-wrong}
-\langle \cdot, \phi^{1,0}_{L_1} \ast \phi^{1,0}_{L_0} \rangle. 
\end{equation}
The resulting version of \eqref{eq:h-map} would still be a chain map, but the analogue of Proposition \ref{th:hexagon} fails. Instead, the map $\mathit{CF}^1(2\mu) \rightarrow \bK$ defined in this way would be chain homotopic to 
\begin{equation}
\pm \langle [\phi^{1,0}_{L_0},\phi^{1,0}_{L_1}], \cdot \rangle \htp \pm \langle \phi^{1,0}_{L_0} \smile \phi^{1,0}_{L_1}, 
\delta(\cdot) \rangle.
\end{equation}
In particular, if we plug in a cocycle representing a dilation, the outcome would be the ordinary intersection number $L_0 \cdot L_1$. Geometrically, choosing \eqref{eq:plausible-but-wrong} corresponds to taking a family of surfaces over the boundary of the hexagon for which the preferred tangent directions rotate by a nonzero degree, hence which can't be extended over the interior of the hexagon.
\end{remark}

\subsection{Defining the pairing}
We now have all the ingredients necessary to flesh out the discussion from Section \ref{sec:strategy}. We choose $B \in \mathit{HF}^1(2\mu)$ and a representing cocycle $\beta \in \mathit{CF}^1(2\mu)$. Define a chain complex
\begin{equation} \label{eq:tilde-c}
\begin{aligned}
& \tilde{C}^* = \mathit{CF}^*(-\mu) \oplus \mathit{CF}^*(\mu), \\
& \tilde{d}(\xi,x) = \big(d\xi, dx - \beta \smile \xi \big).
\end{aligned}
\end{equation}
The cohomology of \eqref{eq:tilde-c} is the graded vector space previously denoted by $\tilde{H}^*$. The long exact sequence \eqref{eq:cone-les} is obvious from the definition. Next we introduce a pairing
\begin{equation} \label{eq:iota-pairing}
\begin{aligned}
& \iota : \tilde{C}^* \otimes \tilde{C}^{2n-*} \longrightarrow \bK, \\
& \iota((\xi_0,x_0), (\xi_1,x_1)) = \langle x_0, \xi_1 \rangle - (-1)^{|\xi_0|}\langle x_1, \xi_0 \rangle + \langle \beta, \xi_0 \ast \xi_1 \rangle.
\end{aligned}
\end{equation}
Let's show that this is a chain map: for any $(\xi_1,x_1)$ and $(\xi_0,x_0)$ whose degrees add up to $2n-1$,
\begin{equation}
\begin{aligned}
& \iota((d\xi_0,dx_0 - \beta \smile \xi_0),(\xi_1,x_1)) + (-1)^{|\xi_0|} \iota((\xi_0,x_0),(d\xi_1,dx_1 - \beta \smile \xi_1)) \\ 
& = \langle dx_0, \xi_1 \rangle + (-1)^{|\xi_0|} \langle x_0, d\xi_1 \rangle - (-1)^{|\xi_1|}\langle x_1, d\xi_0 \rangle -
\langle dx_1, \xi_0\rangle \\
& \qquad - \langle \beta \smile \xi_0, \xi_1 \rangle + \langle \beta \smile \xi_1, \xi_0 \rangle + \langle \beta, d\xi_0 \ast \xi_1 + (-1)^{|\xi_0|} \xi_0 \ast d\xi_1 \rangle \\
& = \langle \beta, -\xi_0 \smile \xi_1 + \xi_1 \smile \xi_0 + d\xi_0 \ast \xi_1 + (-1)^{|\xi_0|} \xi_0 \ast d\xi_1 \rangle \\
& = \langle \beta, -d(\xi_0 \ast \xi_1) \rangle = 0.
\end{aligned}
\end{equation}
The induced cohomology pairing is \eqref{eq:i-pairing}, and we will now establish its basic properties as stated there.

\begin{proof}[Proof of Lemma \ref{th:pairing-1}] The assumption means that $\tilde{x}_1$ can be represented by a cochain of the form $(0,x_1)$, in which case indeed $\iota((\xi_0,x_0),(0,x_1)) = -(-1)^{|\xi_0|} \langle x_1,\xi_0 \rangle = -\langle \xi_0, x_1 \rangle$.
\end{proof}

\begin{proof}[Proof of Lemma \ref{th:pairing-2}] This is clear from the definition:
\begin{equation} \label{eq:unsymmetry-1}
\begin{aligned}
\iota((\xi_0,x_0),(\xi_1,x_1)) + (-1)^{|\xi_0|} \iota((\xi_1,x_1),(\xi_0,x_0)) & = \langle \beta, \xi_0 \ast \xi_1 + (-1)^{|\xi_0|} \xi_1 \ast \xi_0 \rangle \\ & = \langle \beta, [\xi_0,\xi_1] \rangle,
\end{aligned}
\end{equation}
where the second equality holds by our definition of the bracket \eqref{eq:lie}.
\end{proof}

Next, we recall the precise definition of $B$-equivariant Lagrangian submanifold from \cite[Definition 4.2]{seidel-solomon10}. This is a pair $\tilde{L} = (L,\gamma_L)$ consisting of a Lagrangian submanifold $L$ (as usual, with the conditions from Setup \ref{th:setup-lagrangian}) together with an element $\gamma_L \in \mathit{CF}^0(L,L)$ satisfying 
\begin{equation} \label{eq:gamma-cobounds}
\mu^1(\gamma_L) = \phi^{1,1}_L(\beta). 
\end{equation}
Two $B$-equivariant structures on a fixed $L$ are considered to be equivalent if the $\gamma_L$ differ by a coboundary. We can associate to each $B$-equivariant Lagrangian submanifold a cocycle
\begin{equation} \label{eq:equivariant-cocycle}
(\xi,x) = \big(\phi^{1,0}_L, (-1)^{n+1} \phi^{2,0}_L(\beta) + \check{\phi}^{1,1}_L(\gamma_L)\big) \in \tilde{C}^n.
\end{equation}
To see that this is closed under the differential \eqref{eq:tilde-c}, one uses: that $\phi^{1,0}_L$ is a cocycle; the basic property \eqref{eq:phi-phi} of $\phi^{2,0}_L$; that $\beta$ is a cocycle; that $\check{\phi}^{1,1}_L$ is a chain map of degree $n$; and finally \eqref{eq:gamma-cobounds}. The cohomology class of \eqref{eq:equivariant-cocycle} is the previously introduced \eqref{eq:improved-class2}. As stated there, the image of $\lbr \tilde{L} \rbr$ under the map $\tilde{H}^n \rightarrow \mathit{HF}^n(-\mu)$ is indeed $\lbr L \rbr = [\phi^{1,0}_L]$. 

\begin{remark}
Suppose that we change $[\gamma_L]$ by a multiple of the identity class in $\mathit{HF}^0(L,L)$. Then \eqref{eq:equivariant-cocycle} changes by a cocycle homologous to the corresponding multiple of $(0,\phi^{1,0}_L)$. On the cohomology level, this means that $\lbr \tilde{L} \rbr$ changes by a multiple of the image of the standard fundamental class $[L]$ under $H^n_{\mathit{cpt}}(M;\bK) \rightarrow H^n(M;\bK) \rightarrow \mathit{HF}^n(\mu)$.
\end{remark}

Given two $B$-equivariant Lagrangian submanifolds, we have a chain map \cite[Equation (4.4)]{seidel-solomon10}
\begin{equation} \label{eq:tilde-phi-map}
\begin{aligned}
& \phi_{\tilde{L}_0,\tilde{L}_1}: \mathit{CF}^*(L_0,L_1) \longrightarrow \mathit{CF}^*(L_0,L_1), \\
& \phi_{\tilde{L}_0,\tilde{L}_1}(a) = \phi^{1,2}_{L_0,L_1}(\beta,a) - \mu^2_{L_0,L_1,L_1}(\gamma_{L_1},a) +
\mu^2_{L_0,L_0,L_1}(a,\gamma_{L_0}).
\end{aligned}
\end{equation}
The induced map on cohomology is the endomorphism $\Phi_{\tilde{L}_0,\tilde{L}_1}$ from \eqref{eq:tilde-phi}. 

\begin{proof}[Proof of Theorem \ref{th:cardy2}]
From the definition \eqref{eq:tilde-phi-map} and \eqref{eq:phi-check-phi}, \eqref{eq:phi-check-phi-2}, we get
\begin{equation} \label{eq:str-main}
\begin{aligned}
& (-1)^{n(n+1)/2} \mathrm{Str}(\phi_{\tilde{L}_0,\tilde{L}_1}) \\
& = (-1)^{n(n+1)/2} \Big( \mathrm{Str}(\phi^{1,2}_{L_0,L_1}(\beta,\cdot)) 
- \mathrm{Str}(\mu^2_{L_0,L_1,L_1}(\gamma_{L_1},\cdot)) + \mathrm{Str}(\mu^2_{L_0,L_0,L_1}(\cdot,\gamma_{L_0})) \Big) \\
& = (-1)^{n(n+1)/2} \mathrm{Str}(\phi^{1,2}_{L_0,L_1}(\beta,\cdot)) - (-1)^n \psi^{0,1}_{L_0,L_1}(\phi^{1,1}_{L_1}(\beta)) - \langle \phi^{1,0}_{L_0}, \check{\phi}^{1,1}_{L_1}(\gamma_{L_1}) \rangle \\
& \qquad \qquad + (-1)^n \check{\psi}^{0,1}_{L_0,L_1}(\phi^{1,1}_{L_0}(\beta)) + (-1)^n \langle \phi^{1,0}_{L_1}, \check{\phi}^{1,1}_{L_0}(\gamma_{L_0}) \rangle.
\end{aligned}
\end{equation}
We know from Proposition \ref{th:hexagon} that the image of $\beta$ under \eqref{eq:h-map} is zero, which means that
\begin{multline}
(-1)^{n(n+1)/2} \mathrm{Str}(\phi^{1,2}_{L_0,L_1}(\beta,\cdot)) - (-1)^n \psi^{0,1}_{L_0,L_1}(\phi^{1,1}_{L_1}(\beta)) + (-1)^{n} \check{\psi}^{0,1}_{L_0,L_1}(\phi^{1,1}_{L_0}(\beta)) \\ =
-\langle \phi^{1,0}_{L_1}, \phi^{2,0}_{L_0}(\beta) \rangle + \langle \beta, \phi^{1,0}_{L_0} \ast \phi^{1,0}_{L_1} \rangle + (-1)^n \langle \phi^{1,0}_{L_0}, \phi^{2,0}_{L_1}(\beta) \rangle.
\end{multline}
With that in mind, one rewrites \eqref{eq:str-main} as
\begin{equation}
\begin{aligned}
(-1)^{n(n+1)/2} \mathrm{Str}(\phi_{\tilde{L}_0,\tilde{L}_1}) & = 
  \langle (-1)^{n+1} \phi^{2,0}_{L_0}(\beta) + \check{\phi}^{1,1}_{L_0}(\gamma_{L_0}), \phi^{1,0}_{L_1} \rangle \\ & \qquad
- (-1)^n \langle (-1)^{n+1} \phi^{2,0}_{L_1}(\beta) + \check{\phi}^{1,1}_{L_1}(\gamma_{L_1}), \phi^{1,0}_{L_0} \rangle
+ \langle \beta, \phi^{1,0}_{L_0} \ast \phi^{1,0}_{L_1} \rangle.
\end{aligned}
\end{equation}
The right hand side is exactly the result of applying \eqref{eq:iota-pairing} to the cocycles \eqref{eq:equivariant-cocycle}.
\end{proof}

\section{Selected technical aspects\label{sec:technical}}


\subsection{Pseudo-holomorphic map equations}
We want to describe briefly how the Floer-theoretic apparatus from Section \ref{sec:dilations} should be extended in order to cover the operations introduced in Section \ref{sec:operations}. There is nothing particularly original about this. Besides the classical references for operations in the Hamiltonian \cite{schwarz95,piunikhin-salamon-schwarz94, seidel07, ritter10} and Lagrangian \cite{fukaya93,desilva98} flavours of Floer theory, there is now a considerable amount of literature concerning the combination of the two \cite{seidel02, albers08, biran-cornea09, abouzaid-seidel07, abouzaid10,ritter-smith12}. Among the last-mentioned group, \cite{abouzaid10} is particularly close to our concerns. 

To begin, let's slightly rigidify the class of Riemann surfaces under consideration (this will not make any essential difference, since all the families which we have considered previously can be adapted without any issues to this framework; indeed, there is a general fact ensuring that this can be done, which however would take too long to formulate properly).

\begin{setup} \label{th:setup-3}
Take a Riemann surface as in Setup \ref{th:setup-2}. We want to make additional choices of distinguished coordinates near the punctures. For the interior punctures, these choices are tubular ends
\begin{equation} \label{eq:tubular-ends}
\left\{
\begin{aligned}
& \epsilon_\zeta: (-\infty,0] \times S^1 \longrightarrow S, && \zeta \in \Sigma^{\mathit{cl,out}}, \\
& \epsilon_\zeta: [0,\infty) \times S^1 \longrightarrow S, && \zeta \in \Sigma^{\mathit{cl,in}}.
\end{aligned}
\right.
\end{equation}
More precisely, the $\epsilon_\zeta$ are proper holomorphic embeddings with $\lim_{s \rightarrow \pm \infty} \epsilon_\zeta(s,\cdot) = \zeta$, chosen in such a way that the distinguished tangent direction points along the arc $\{\epsilon_\zeta(s,0)\}$. We introduce a slightly stricter version of \eqref{eq:one-form}, requiring that near infinity on each tubular end,
\begin{equation} \label{eq:nu-over-the-ends}
\epsilon_\zeta^*\nu_S = \lambda_\zeta\, \mathit{dt}.
\end{equation}
The analogue for the boundary punctures are strip-like ends
\begin{equation}
\left\{
\begin{aligned}
& \epsilon_\zeta: (-\infty,0] \times [0,1] \longrightarrow S, && \zeta \in \Sigma^{\mathit{op,out}}, \\
& \epsilon_\zeta: [0,\infty) \times [0,1] \longrightarrow S, && \zeta \in \Sigma^{\mathit{op,in}},
\end{aligned}
\right.
\end{equation}
and we again require that \eqref{eq:nu-over-the-ends} should hold.
\end{setup}

Additional data are now chosen as in the construction of continuation maps \eqref{eq:continuation-map}. This means first of all, a section $K_S$ of the pullback bundle $T^*S \rightarrow S \times M$. It should satisfy \eqref{eq:k-infty}, and reduce to $H_{\lambda_\zeta,t} \mathit{dt}$ on each tubular end, and to $H_{L_{\zeta,0},L_{\zeta,1},t} \mathit{dt}$ on each strip-like end. Additionally, we require that if $\xi$ is tangent to some component $C \subset \partial S$, then $K_S(\xi)|L_C = 0$; this is the generalization of \eqref{eq:boundary-condition-k}. The second piece of data is a family $J_S$ of compatible almost complex structures parametrized by $S$, again satisfying the analogue of \eqref{eq:j-infty}, and reducing to $J_{\lambda_\zeta,t}$ respectively to $J_{L_{\zeta,0},L_{\zeta,1},t}$ on the ends. By using the section $Y_S$ of $\mathit{Hom}(TS,TM) \rightarrow S \times M$ associated to $K_S$, one writes down the appropriate generalization of \eqref{eq:cont-equation}:
\begin{equation} \label{eq:generalized-floer}
\left\{ 
\begin{aligned}
& u: S \longrightarrow M, \\
& u(C) \subset L_C \quad \text{for each component $C \subset \partial S$}, \\
& (du - Y_{S,z})^{0,1} = 0, \\
& \textstyle \lim_{s \rightarrow \pm\infty} u(\epsilon_\zeta(s,t)) = x_\zeta(t).
\end{aligned}
\right.
\end{equation}
The limits $x_\zeta$ are appropriate trajectories (either $1$-periodic orbits or chords). Exactness guarantees an a priori bound on the energy, and \eqref{eq:convexity} provides the necessary $C^0$-bound (showing that solutions $u$ cannot escape to infinity in $M$). Moreover, for a generic choice of $(J_S,K_S)$ (in fact, even a generic choice of $K_S$ with arbitrary fixed $J_S$), the moduli space of solutions of \eqref{eq:generalized-floer} is regular.

For a single surface $S$, this is all one needs: counting points in the zero-dimensional moduli spaces of solutions of \eqref{eq:generalized-floer} defines the chain map \eqref{eq:operation}. In the case of a family $(S_r)_{r \in R}$ with compact parameter space $R$ is similar, one proceeds in the same way, but where all the auxiliary structures are chosen to be smooth in $R$. Of course, in applications where $R$ has boundary, the structures associated to points $r \in \partial R$ are usually related to previously defined operations, but that does not interfere with regularity.

\begin{example}
In the definition of \eqref{eq:star-product}, the Riemann surface $S_r = S$ itself is the same for all $r$, but the tubular ends rotate in dependence of the parameter value $r \in R = [0,1]$, and hence $(K_{S_r},J_{S_r})$ must vary with $r$. To obtain the required equation \eqref{eq:homotopy-commutativity}, one has to fix the choices at the endpoints: for $r = 0$ one uses the auxiliary data which enter the construction of \eqref{eq:smile-product}, and for $r = 1$ the pullback of the same data by an automorphism of $S$.
\end{example}

The case of non-compact $R$ is a little more difficult. For a single parameter $r \in R = \bR$, one has points $r = \pm \infty \in \bar{R}$ which correspond to ``broken'' surfaces, and the $S_r$ for $|r| \gg 0$ are obtained from those by a gluing construction (gluing together either strip-like or tubular ends). A simple solution would be to choose the data $(K_{S_r},J_{S_r})$ for $|r|$ large to be ones inherited from those for $r = \pm\infty$ through the gluing process. However, to ensure transversality, one generally has to allow a further perturbation, which however needs to decay sufficiently swiftly in the limit $r \rightarrow \pm\infty$. For technical simplicity, it is convenient to use perturbations that are zero on the ends as well as the ``thin'' pieces of $S_r$, which are what remains from the ends used up in the gluing process. For a description of the necessary conditions, we refer to \cite[Section 9]{seidel04} (this covers only discs with boundary punctures, but the overall strategy is the same in all cases). For higher-dimensional $R$, the main additional difficulty is to properly understand the gluing processes which describe the structure of $R$ near infinity. This is straightforward for \eqref{eq:bv-relation}. The remaining case is sufficiently important for our purpose to merit a more detailed discussion, which is our next task.

\subsection{The two-parameter family\label{subsec:ksv}}
We will now construct in detail the parameter space for the family of Riemann surfaces which appears in Proposition \ref{th:hexagon} (see Figures \ref{fig:hex-sides}, \ref{fig:hex-corners}). The construction combines ideas of \cite{kimura-stasheff-voronov95} (the Kimura-Stasheff-Voronov compactification) and \cite{liu02} (real Gromov-Witten theory). 

To begin with, take the moduli space $\scrR$ of annuli with one interior marked point, and let $\bar\scrR$ be its Kimura-Stasheff-Voronov (KSV) compactification. Points of $\bar\scrR$ parametrize objects (up to isomorphism) of the following kind. Let $\bar{S}$ be a nodal marked Riemann surface with boundary, which is a degeneration of an annulus with one interior marked point. Here, ``nodal Riemann surface with boundary'' means that 
\begin{equation} \label{eq:bar-s}
\bar{S} = \bigcup_{i \in I} \bar{S}_i,
\end{equation}
where the irreducible components $\bar{S}_i$ are compact Riemann surfaces, possibly with boundary. These components are glued together at boundary nodes and interior nodes, to form \eqref{eq:bar-s}. In addition, we require that one of these components should carry an interior marked point, which is not one of the nodes. The KSV structure is the following: let $\zeta \in \bar{S}$ be an interior node, with preimages $\zeta_\pm \in \bar{S}_{i_\pm}$. Then, we want to have a distinguished direction (a nonzero vector specified up to positive real multiples),
\begin{equation} \label{eq:node-direction}
\bR^+ \cdot \delta_\zeta \subset T_{\zeta_-}\bar{S}_{i_-} \otimes_{\bC} T_{\zeta_+}\bar{S}_{i_+}.
\end{equation}
Finally, there is a stability condition, which says that the group of automorphisms of $\bar{S}$ preserving all the preferred directions \eqref{eq:node-direction} should be finite; in our particular case, this implies that the automorphism groups are actually trivial. An isomorphism class of stable $\bar{S}$ determines a point of $\bar\scrR$. As usual, there is a stratification by topological type, with each stratum being an open manifold; the main stratum (an annulus) and the codimension $1$ strata are shown in Figure \ref{fig:pentagon}.

One general feature of KSV-type spaces is that they come with a continuous map to the corresponding Deligne-Mumford space. This map is defined by forgetting \eqref{eq:node-direction} and then collapsing the components that become unstable. In our case, a stratum-by-stratum analysis shows that no information is lost either by forgetting \eqref{eq:node-direction} or by collapsing components. Hence, the map to the corresponding real Deligne-Mumford space is bijective, and therefore a homeomorphism. Conversely, one can take this as a definition of the topology of $\bar{\scrR}$, starting with the familiar topology of the real Deligne-Mumford space as described in \cite{liu02} (see also \cite[Figure 10]{devadoss-heath-vipismakul10}). In fact, while we're about it, we will carry over the differentiable structure (as a two-dimensional manifold with corners) as well. 
%
 
\begin{figure}
\begin{centering}
\begin{picture}(0,0)%
\includegraphics{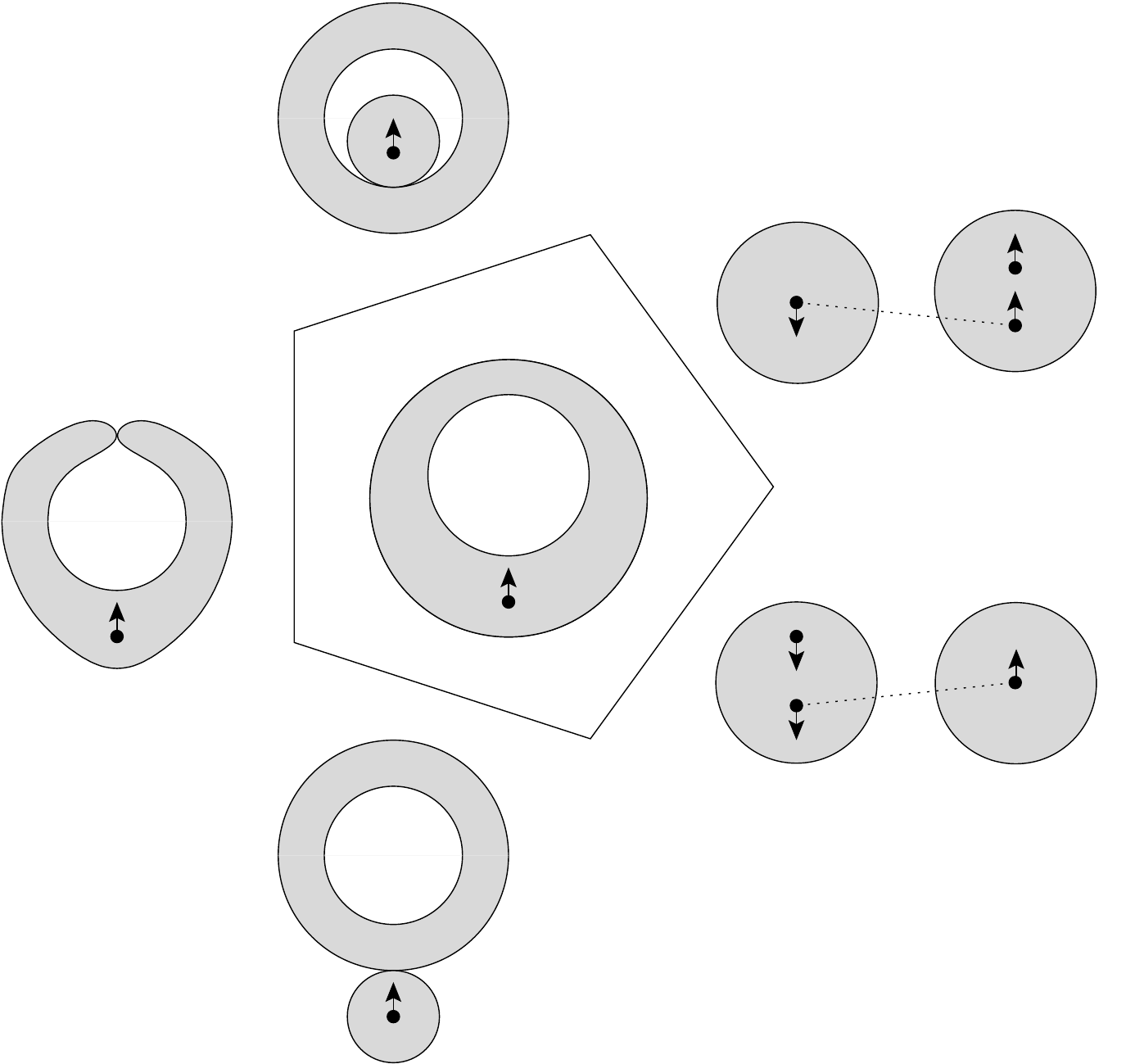}%
\end{picture}%
\setlength{\unitlength}{3552sp}%
\begingroup\makeatletter\ifx\SetFigFont\undefined%
\gdef\SetFigFont#1#2#3#4#5{%
  \reset@font\fontsize{#1}{#2pt}%
  \fontfamily{#3}\fontseries{#4}\fontshape{#5}%
  \selectfont}%
\fi\endgroup%
\begin{picture}(7377,6916)(664,-6593)
\put(3451,-1711){\makebox(0,0)[lb]{\smash{{\SetFigFont{11}{13.2}{\rmdefault}{\mddefault}{\updefault}{\color[rgb]{0,0,0}(ii)}%
}}}}
\put(3751,-5986){\makebox(0,0)[lb]{\smash{{\SetFigFont{11}{13.2}{\rmdefault}{\mddefault}{\updefault}{\color[rgb]{0,0,0}$L_0$}%
}}}}
\put(3076,-5611){\makebox(0,0)[lb]{\smash{{\SetFigFont{11}{13.2}{\rmdefault}{\mddefault}{\updefault}{\color[rgb]{0,0,0}$L_1$}%
}}}}
\put(3001,-211){\makebox(0,0)[lb]{\smash{{\SetFigFont{11}{13.2}{\rmdefault}{\mddefault}{\updefault}{\color[rgb]{0,0,0}$L_1$}%
}}}}
\put(3826, 89){\makebox(0,0)[lb]{\smash{{\SetFigFont{11}{13.2}{\rmdefault}{\mddefault}{\updefault}{\color[rgb]{0,0,0}$L_0$}%
}}}}
\put(1313,-3399){\makebox(0,0)[lb]{\smash{{\SetFigFont{11}{13.2}{\rmdefault}{\mddefault}{\updefault}{\color[rgb]{0,0,0}$L_1$}%
}}}}
\put(1801,-2386){\makebox(0,0)[lb]{\smash{{\SetFigFont{11}{13.2}{\rmdefault}{\mddefault}{\updefault}{\color[rgb]{0,0,0}$L_0$}%
}}}}
\put(7576,-1111){\makebox(0,0)[lb]{\smash{{\SetFigFont{11}{13.2}{\rmdefault}{\mddefault}{\updefault}{\color[rgb]{0,0,0}$L_1$}%
}}}}
\put(6076,-1111){\makebox(0,0)[lb]{\smash{{\SetFigFont{11}{13.2}{\rmdefault}{\mddefault}{\updefault}{\color[rgb]{0,0,0}$L_0$}%
}}}}
\put(6526,-4786){\makebox(0,0)[rb]{\smash{{\SetFigFont{11}{13.2}{\rmdefault}{\mddefault}{\updefault}{\color[rgb]{0,0,0}$L_0$}%
}}}}
\put(8026,-4786){\makebox(0,0)[rb]{\smash{{\SetFigFont{11}{13.2}{\rmdefault}{\mddefault}{\updefault}{\color[rgb]{0,0,0}$L_1$}%
}}}}
\put(2626,-2836){\makebox(0,0)[lb]{\smash{{\SetFigFont{11}{13.2}{\rmdefault}{\mddefault}{\updefault}{\color[rgb]{0,0,0}(iii)}%
}}}}
\put(3451,-4036){\makebox(0,0)[lb]{\smash{{\SetFigFont{11}{13.2}{\rmdefault}{\mddefault}{\updefault}{\color[rgb]{0,0,0}(iv)}%
}}}}
\put(4876,-3586){\makebox(0,0)[lb]{\smash{{\SetFigFont{11}{13.2}{\rmdefault}{\mddefault}{\updefault}{\color[rgb]{0,0,0}(v)}%
}}}}
\put(4876,-2236){\makebox(0,0)[lb]{\smash{{\SetFigFont{11}{13.2}{\rmdefault}{\mddefault}{\updefault}{\color[rgb]{0,0,0}(i)}%
}}}}
\end{picture}%
\caption{\label{fig:pentagon}}
\end{centering}
\end{figure}%

We can choose the following additional data smoothly over $\bar\scrR$: a tangent direction at the marked point, and moreover at each interior node, tangent directions to the preimages $\zeta_\pm$, whose tensor product equals $\bR^+ \cdot \delta_\zeta$. More precisely, we first want to make the choices on the boundary strata exactly as indicated in Figure \ref{fig:pentagon}, and then extend the choice of tangent direction at the marked point smoothly over the interior. 

The hexagon in Figure \ref{fig:hex-sides}, which we denote by $\hat\scrR$ from now on, is the real blowup of $\bar\scrR$ at the corner between boundary sides (i) and (v). We pull back the family of (nodal) surfaces by the projection $\hat\scrR \rightarrow \bar\scrR$, and then change the previously introduced tangent directions at the interior marked point as follows. Along boundary side (iii) (with its boundary orientation), rotate the tangent direction gradually by a total angle of $-\pi$. As a result, we get the opposite tangent direction along boundary sides (iv) and (v), and finally compensate this by rotating by $+\pi$ along the new boundary side (vi) created by the blowup. The outcome can then again be extended over the interior (this extension is obviously not the pullback of the previous one on $\bar\scrR$), yielding tangent directions as in Figure \ref{fig:hex-sides}.

We need to consider the geometry near the new boundary side (vi) in more detail. Let $S_0$ be the three-punctured sphere, with tubular ends 
\begin{equation}
\begin{aligned}
& \epsilon_{+,0}: (-\infty,0] \times S^1 \longrightarrow S_0, \\
& \epsilon_{+1}, \; \epsilon_{+,2}: [0,\infty) \times S^1 \longrightarrow S_0
\end{aligned}
\end{equation}
which are compatible with the choices of tangent directions used to construct \eqref{eq:smile-product}. Let $S_1$ be a disc with one interior puncture, with its tubular end $\epsilon_{-,1}: (-\infty,0] \times S^1 \rightarrow S_1$. We assume that this is rotationally symmetric, meaning that each nontrivial automorphisms of $S_1$ maps $\epsilon_{-,1}(s,t)$ to $\epsilon_{-,1}(s,t+\theta)$, for some constant $\theta$. Let $S_2$ be another disc of the same kind, with its end $\epsilon_{-,2}$. Gluing together these three surfaces with length parameters $(l_1,l_2)$ and angle parameters $(\theta_1,\theta_2)$ means to identify
\begin{equation} \label{eq:double-glue}
\begin{aligned}
& \epsilon_{+,1}(s,t) \sim \epsilon_{-,1}(s-l_1,t+\theta_1) \quad \text{for $(s,t) \in [0,l_1] \times S^1$,} \\
& \epsilon_{+,2}(s,t) \sim \epsilon_{-,2}(s-l_2,t+\theta_2) \quad \text{for $(s,t) \in [0,l_2] \times S^1$.}
\end{aligned}
\end{equation}
It is understood that we have first removed all but a finite piece of the ends in question. If we are only interested in the resulting Riemann surfaces, the choice of $\theta_k$ is irrelevant because of the rotational symmetry of the disc. However, later on when the surfaces come equipped with additional data $(J_{S_k},K_{S_k})$ breaking that symmetry, the angle parameters will matter.

In Figure \ref{fig:pentagon}, a coordinate neighbourhood of the corner point between sides (i) and (v) is parametrized by $(w_1,w_2) \in (-\delta,0]^2$ for some small $\delta > 0$, corresponding to choices of gluing lengths $l_k = -\log(-w_k)$. The logarithms come from the standard complex coordinates on the smoothing of a nodal Riemann surface, compare e.g.\ \cite[Section 9e]{seidel04}. We now pass to the real blowup, which near the new boundary side (vi) has coordinates $(w,r) \in (-\delta,0] \times [0,1]$ related to the previous ones by $w_1 = w\cos(\pi r/2)$, $w_2 = w\sin(\pi r/2)$. Keeping this and the $r$-dependent rotation of the strip-like ends from Figure \ref{fig:star} in mind, one finds that the gluing parameters are now given by
\begin{equation} \label{eq:rotate-glue}
\begin{aligned}
& l_1 = -\log(-w \cos(\pi r/2)), && \theta_1 = r/2-1/2, \\ 
& l_2 = -\log(-w \sin(\pi r/2)), && \theta_2 = r/2-1/2.
\end{aligned}
\end{equation}
In words, as we go from the boundary side (vi) inwards, both nodes are being smoothed, but at rates which reflect where on the boundary we started. Moreover, there is a rotational twist which varies linearly along the boundary. In our application, $S_0$ comes with the $r$-dependent datum $(J_{S_0,r},K_{S_0,r})$ underlying \eqref{eq:star-product}; and $S_1$, $S_2$ with the data that define $\phi^{1,0}_{L_1}$ and $\phi^{1,0}_{L_0}$, respectively. The choice of $\theta_1,\theta_2$ in \eqref{eq:rotate-glue} ensures that these are compatible with the gluing process. One chooses corresponding data on the glued surfaces as before. Given that, (by now) standard compactness and gluing arguments apply to the resulting parametrized moduli space. This forms the core of the proof of Proposition \ref{th:hexagon}.

\subsection{Signs and supertraces}
The last technical topic which we wish to expand on is a sign issue. Again, this is not new: it represents a small part of the more comprehensive discussion in \cite[Section 25]{fukaya-oh-ohta-ono10b}.

We start by recalling some index theory that enters into the construction of signs in Floer theory. Our presentation follows \cite[Section 11]{seidel04}, but an equivalent account (with different terminology) can be found in \cite[Chapter 8]{fooo}. Let $(V^{2n},\omega_V)$ be a symplectic vector space, with a compatible complex structure $J_V$ and complex volume form $\eta_V$. A linear Lagrangian brane $\Lambda$ in $V$ is an oriented linear Lagrangian subspace equipped with a grading and a {\em Spin} structure. The grading is a number $\alpha_\Lambda \in \bR$ such that $\exp(\pi i\alpha_\Lambda) = \eta_V(v_1 \wedge \cdots \wedge v_n) \in S^1$ for any oriented orthonormal basis $(v_1,\dots,v_n)$ of $\Lambda$. The {\em Spin} structure is a principal homogeneous space for the group $\mathit{Spin}_n$ together with an isomorphism between its reduction to $\mathit{SO}_n$ and the standard frame bundle of $\Lambda$. In the same way, one defines the notion of family of linear Lagrangian branes parametrized by some space.

Let $S$ be a compact Riemann surface with boundary, together with an ordering of its boundary circles. Suppose that this surface comes with a family $\Lambda = \{\Lambda_z\}$ of linear Lagrangian branes parametrized by $z \in \partial S$. Consider the standard Cauchy-Riemann operator with totally real boundary conditions:
\begin{equation} \label{eq:dbar}
\begin{aligned}
& \bar\partial_S: \scrE^0_S \longrightarrow \scrE^1_S, \\
& \scrE^0_S = \{v \in W^{k,2}(S,V) \;:\; v(z) \in \Lambda_z \text{ for $z \in \partial S$}\}, \\ 
& \scrE^1_S = W^{k-1,2}(S,V)
\end{aligned}
\end{equation}
for some $k \geq 1$. This is elliptic, and its determinant line
\begin{equation}
\mathit{det}(\bar\partial_S) = \lambda^{\mathit{top}}(\mathit{coker}(\bar\partial_S)^\vee) \otimes
\lambda^{\mathit{top}}(\mathit{ker}(\bar\partial_S))
\end{equation}
is independent of the choice of $k$ up to canonical isomorphism (as usual for determinant lines, we do not distinguish between two isomorphisms whose quotient is a positive number). Following \cite[Section 11]{seidel04} one can construct a preferred trivialization
\begin{equation} \label{eq:det-trivial}
\mathit{det}(\bar\partial_S) \iso \bR
\end{equation}
as follows. The prototypical case is where $S$ is a disc, and $\Lambda$ is constant. In that case, $\bar\partial_S$ is onto and its kernel consists of constant sections, yielding $\mathit{det}(\bar\partial_S) = \lambda^{\mathit{top}}(\Lambda)$. If the {\em Spin} structure is trivial, one uses the given orientation of $\Lambda$ to get \eqref{eq:det-trivial}, while for the other {\em Spin} structure one uses the opposite trivialization. The case where $S$ is a disc and $\Lambda$ is arbitrary can be reduced to this by deformation. Finally, one deals with general $S$ by degeneration, more concretely by shrinking boundary-parallel circles to nodes. All but one of the irreducible components of the resulting nodal surface are discs, to which the previous argument applies (using their ordering to put the trivializations together). The remaining irreducible component is closed, and one uses the canonical orientation of complex vector spaces to trivialize its determinant line (as well as for the remaining data needed to glue the components together). 

\begin{remark} \label{th:koszul-exchange}
Applying an even permutation to the ordering of boundary circles does not affect \eqref{eq:det-trivial}, but an odd permutation changes it by a Koszul sign obtained by exchanging the order of two operators \eqref{eq:dbar} associated to discs. Since those operators both have index $n$, the sign is $(-1)^n$.
\end{remark}

\begin{example} \label{th:two-discs}
Take two discs $D_0,D_1$ with marked boundary points $z_0 \in \partial D_0$, $z_1 \in \partial D_1$, and boundary conditions $\Lambda_0, \Lambda_1$ which are locally constant near the marked points, and such that $\Lambda_{0,z_0} = \Lambda_{1,z_1}$. One can then glue together the two to a disc $D$ with boundary conditions $\Lambda$. Gluing theory for elliptic operators identifies the determinant line of the $\bar\partial_D$ with that of the restriction of $\bar\partial_{D_1} \oplus \bar\partial_{D_0}$ to the subspace where $\{v_0(z_0) = v_1(z_1)\}$ (which makes sense for $k>1$). We can deform that condition linearly to $\{v_0(z_0) = 0\}$, which yields
\begin{equation} \label{eq:glue-discs}
\mathit{det}(\bar\partial_S) \iso \mathit{det}(\bar\partial_{S_1}) \otimes \lambda^{\mathit{top}}(\Lambda_{0,z_0}^\vee) \otimes \mathit{det}(\bar\partial_{S_0}).
\end{equation}
By the general convention for determinant lines \cite[Equation (11.3)]{seidel04}, if $(v_1,\dots,v_n)$ is an oriented basis of $\Lambda_{0,z_0}$, then $\lambda^{\mathit{top}}(\Lambda_{0,z_0}^\vee)$ should be trivialized using $v_n^\vee \wedge \cdots \wedge v_1^\vee$. If one uses that to cancel the middle term on the right hand side of \eqref{eq:glue-discs}, the outcome is an isomorphism $\mathit{det}(\bar\partial_S) \iso \mathit{det}(\bar\partial_{S_1}) \otimes \mathit{det}(\bar\partial_{S_0})$ compatible with \eqref{eq:det-trivial}.
\end{example}

\begin{example} \label{th:annulus-1}
Let $S$ be an annulus, with constant boundary conditions $(\Lambda_0,\Lambda_1)$ along the boundary circles, where the two Lagrangian subspaces intersect transversally. The operator \eqref{eq:dbar} is invertible, hence $\mathit{det}(\bar\partial_S) \iso \bR$ tautologically. Now suppose that we have another pair $(\Lambda_0',\Lambda_1')$ satisfying the same transversality condition. If one deforms one pair into the other, the resulting isomorphism between determinant lines differs from the tautological trivializations by a sign which comes from the Maslov index for paths \cite{robbin-salamon93} of the deformation reduced mod $2$. Equivalently, if we choose orientations, the sign is given by the difference of intersection numbers,
\begin{equation} \label{eq:mod2-maslov}
(-1)^{\Lambda_0 \cdot \Lambda_1 - \Lambda_0' \cdot \Lambda_1'}.
\end{equation}
Supposing that all boundary conditions carry linear brane structures (with trivial {\em Spin} structures), let's compare the tautological trivialization with the one defined previously, as part of our general process. To do that, one degenerates $S$ to two discs $S_0,S_1$, with constant boundary conditions $\Lambda_0,\Lambda_1$, and which are identified with each other along interior points $z_0,z_1$. Linear gluing theory yields an isomorphism
\begin{equation} \label{eq:linear-gluing}
\bR = \mathit{det}(\bar\partial_S) \iso \lambda^{\mathit{top}}(V^\vee) \otimes \mathit{det}(\bar\partial_{S_0}) \otimes \mathit{det}(\bar\partial_{S_1}),
\end{equation}
which is induced by the isomorphism $\mathit{ker}(\bar\partial_{S_0}) \oplus \mathit{ker}(\bar\partial_{S_1}) \rightarrow \lambda^{\mathit{top}}(V)$, $(v_0,v_1) \mapsto v_0-v_1$. The ordering of the terms in \eqref{eq:linear-gluing} reflects the fact that the boundary components of $\partial S$ are assumed to be ordered according to $(\Lambda_0,\Lambda_1)$, and the general conventions for \eqref{eq:linear-gluing} (in contrast, the ordering in \eqref{eq:glue-discs} was a choice we made freely when introducing that gluing process). If we use the given orientations of the $\Lambda_k$ and the complex orientation of $V$, then the determinant of that isomorphism is
\begin{equation} \label{eq:01-dot}
(-1)^{n + \Lambda_0 \cdot \Lambda_1}.
\end{equation}
Hence, this is the sign by which the tautological trivialization of $\mathit{det}(\bar\partial_S)$ differs from the one constructed in \eqref{eq:det-trivial}. Note that this is compatible with \eqref{eq:mod2-maslov}.
\end{example}

\begin{example} \label{th:annulus-2}
Take $S$ again to be an annulus, with the same constant boundary condition $\Lambda$ along both boundary components (and trivial {\em Spin} structures). We want to consider two ways of trivializing $\mathit{det}(\bar\partial_\Lambda)$. The first one is to deform the boundary conditions to $\Lambda_0 = e^{a J_V}\Lambda$, $\Lambda_1 = \Lambda$ for some small $a>0$. This deforms the operator into an invertible one, and we use the tautological trivialization of its determinant line. Note that this choice of deformation amounts to
\begin{equation} \label{eq:01-dot-2}
\Lambda_0 \cdot \Lambda_1 = (-1)^{n(n+1)/2}.
\end{equation}
The second possibility is to think of the annulus as obtained by starting with a disc $S'$ and gluing together two boundary points $z_0'$ and $z_1'$. By proceeding as in Example \ref{th:two-discs}, one gets 
\begin{equation}
\mathit{det}(\bar\partial_S) \iso \lambda^{\mathit{top}}(\Lambda^\vee) \otimes \mathit{det}(\bar\partial_{S'}) \iso \bR,
\end{equation}
where the last isomorphism uses the fact that the $\bar\partial$-operator on $S'$ becomes invertible when restricted to the subspace where $v(z_0') = 0$; or equivalently, it uses the orientation of $\Lambda$ and \eqref{eq:det-trivial}. Explicit computation (compare \cite[Section 25.4.1]{fukaya-oh-ohta-ono10b}) shows that these two trivializations agree.
\end{example}

Comparing Examples \ref{th:annulus-1} and \ref{th:annulus-2} yields the following. Suppose that we have an annulus $S$ with general boundary conditions. One can trivialize $\bar\partial_S$ either by degenerating it to two discs glued to each other along interior points, or else by degenerating it to a disc glued to itself along boundary points. The first method recovers \eqref{eq:det-trivial}; the second one yields a result differing from that by $(-1)^{n(n-1)/2}$, which comes from inserting \eqref{eq:01-dot-2} into \eqref{eq:01-dot}. This discrepancy (which appears in \cite{fukaya-oh-ohta-ono10b}, and was pointed out to the author by Abouzaid) is at the root of the signs appearing in connection with supertraces in all our formulae.

More explicitly, consider the right hand side of \eqref{eq:phi-check-phi} and \eqref{eq:phi-check-phi-2}. We get a $(-1)$ in front of $\langle \phi^{1,0}, \check{\phi}^{1,1}(a) \rangle$ because this is the boundary point $r = -\infty$ of the compactified parameter space $\bar{R} = \bR \cup \{\pm\infty\}$, hence counts negatively. The supertrace term carries the sign $(-1)^{n(n-1)/2}$ explained above. The final Koszul sign $(-1)^n$ is obviously due to the ordering of $L_0$ and $L_1$, as in Remark \ref{th:koszul-exchange}, but we can't fully satisfactorily explain its role without going into the geometric orientation conventions that enter into the definition of the simpler operations ($\phi^{1,0}$ and $\check{\phi}^{1,1}$), which is beyond our scope here. Similar remarks apply to \eqref{eq:6-expressions}, with the added complication that the boundary terms themselves are given by families of Riemann surfaces over suitable one-dimensional spaces $R$, giving rise to additional Koszul signs associated to the position of $TR$ in the ordering of the various operators. Luckily, the conceptually most important term (iii) is also the one whose signs can be explained very simply: it comes with a $(-1)^{n(n-1)/2}$ as before, plus an additional $(-1)$ because the identification of the parameter space underlying $\phi^{1,2}_{L_0,L_1}$ with the real line is the opposite of the boundary orientation of the hexagon. The last-mentioned $(-1)$ also appears in terms (iv) and (v), see the arrows in Figure \ref{fig:hex-sides}. Finally, we should point out that the (geometric) construction of the hexagon shows that the sum of the six boundary terms must be a chain map. We know this to be true algebraically (by relying on the properties of the simpler operations), and that provides a nontrivial consistency check on \eqref{eq:6-expressions}.
%
%
%
%

\section{Analogies\label{sec:motivation}}

\subsection{Topology}
Within classical topology, one can introduce a construction which shares some of the formal properties of our main theory (however, the BV operator is trivial in this context, so there are no analogues of dilations). We outline this briefly, hoping that it can be useful in providing additional intuition. To make the similarities stand out, we will allow overlaps in the notation, which is therefore inconsistent with the rest of the paper (as well as with Section \ref{subsec:alg-geom}); the reader has been warned!

Let $M$ be an oriented manifold, together with a class $B \in H^1(M;\bK)$ for some field $\bK$, represented by a singular $1$-cocycle $\beta$. We consider closed connected oriented submanifolds $L \subset M$ equipped with a $0$-cochain $\gamma_L: L \rightarrow \bK$ such that $d\gamma_L = \beta|L$. The pair $(L,\gamma_L)$ is written as $\tilde{L}$. Given two such submanifolds of complementary dimension and which intersect transversally, define
\begin{equation} \label{eq:bullet-trivial}
\tilde{L}_0 \bullet \tilde{L}_1 = \sum_{p \in L_0 \cap L_1} \pm (\gamma_{L_1}(p) - \gamma_{L_0}(p)) \in \bK,
\end{equation}
where $\pm$ is the local sign with which $p$ contributes to the ordinary intersection number $L_0 \cdot L_1$. Changing $\gamma_{L_k}$ by a constant will change \eqref{eq:bullet-trivial} by a corresponding multiple of $L_0 \cdot L_1$. Moreover, the following symmetry formula holds:
\begin{equation} \label{eq:bullet-trivial-symmetry}
\tilde{L}_1 \bullet \tilde{L}_0 + (-1)^{\mathrm{codim}(L_0) \, \mathrm{codim}(L_1)} \tilde{L}_0 \bullet \tilde{L}_1 = 0.
\end{equation}
One can show that \eqref{eq:bullet-trivial} is invariant under isotopies (if one carries over the $\gamma_{L_k}$ appropriately), hence the transverse intersection assumption is not really needed. Along the same lines, one does not need embedded submanifolds: smooth manifolds mapped to $M$ in an arbitrary way will do. 

\begin{example}
Let $M$ be the two-punctured plane, and $L$ the immersion of the circle drawn in Figure \ref{fig:8}. There is an obvious nontrivial choice of $B$ whose pullback to $L$ vanishes on the cohomology level, but not on the cochain level. Namely, take the integral of $B$ around a loop winding around either puncture to be $1$. Then, the values of $\gamma_L$ at the two preimages of the self-intersection point differ by $1$. If one considers the intersection of $L$ with a slightly perturbed version of itself, the contribution of each intersection point to \eqref{eq:bullet-trivial} vanishes unless the intersection comes from two different branches of $L$ crossing each other. Therefore, there are two points which give a nontrivial contribution. Those two points have different signs, but at the same time they have the opposite sign of $\gamma_{L_1}(p) - \gamma_{L_0}(p)$. Hence, their contributions add up to
\begin{equation}
\tilde{L} \bullet \tilde{L} = 2.
\end{equation}
\end{example}
\begin{figure}
\begin{centering}
\begin{picture}(0,0)%
\includegraphics{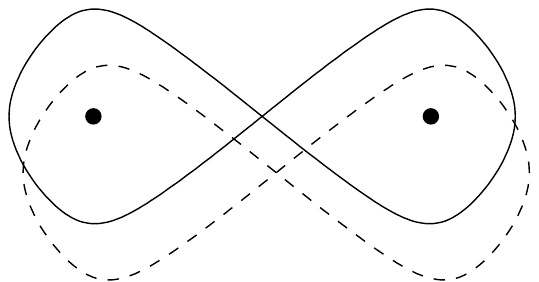}%
\end{picture}%
\setlength{\unitlength}{3552sp}%
\begingroup\makeatletter\ifx\SetFigFont\undefined%
\gdef\SetFigFont#1#2#3#4#5{%
  \reset@font\fontsize{#1}{#2pt}%
  \fontfamily{#3}\fontseries{#4}\fontshape{#5}%
  \selectfont}%
\fi\endgroup%
\begin{picture}(3105,1468)(1186,-1395)
\end{picture}%
\caption{\label{fig:8}}
\end{centering}
\end{figure}

One can think of \eqref{eq:bullet-trivial} as a twisted intersection pairing. This is most easily explained for closed $M$ and a finite field $\bK = \bF_p$. Consider the cyclic $p$-fold covering $\tilde{M} \rightarrow M$ associated to $B$. Choose a triangulation of $M$, and lift it to $\tilde{M}$. The simplicial cochain complex $C^*(\tilde{M})$ with $\bK$-coefficients comes with an action of the covering group, which turns it into a free module over $\bK[q]/q^p-1 = \bK[q]/(q-1)^p$. Consider
\begin{equation} \label{eq:trivial-c-tilde}
\tilde{C}^* \stackrel{\mathrm{def}}{=} C^*(\tilde{M}) \otimes_{\bK[q]/q^p-1} \bK[q]/(q-1)^2.
\end{equation}
Bearing in mind that $C^*(\tilde{M}) \otimes_{\bK[q]/q^p-1} \bK[q]/(q-1) = C^*(M)$ is the analogous complex for $M$, we find that the cohomology $\tilde{H}^*$ of \eqref{eq:trivial-c-tilde} sits in a long exact coefficient sequence
\begin{equation}
\cdots \rightarrow H^*(M;\bK) \longrightarrow \tilde{H}^* \longrightarrow H^*(M;\bK) \rightarrow \cdots
\end{equation}
where the boundary map is the product with $B$. Define a pairing $\iota$ on $C^*(\tilde{M})$ by
\begin{equation}
\iota(\tilde{x}_0,\tilde{x}_1) = \sum_j j \, \textstyle\int_{\tilde{M}} (q^{-j} \tilde{x}_0) \tilde{x}_1,
\end{equation}
where: the sum is over all $j \in \bF_p$; $q^{-j} \tilde{x}_0$ is the pullback by the covering transformation corresponding to $-j$; $(q^{-j} \tilde{x}_0) \tilde{x}_1$ is the cup-product; and $\int_{\tilde{M}}$ is the pairing with the simplicial fundamental cycle. This satisfies
\begin{equation}
\iota(q \tilde{x}_0, \tilde{x}_1) = \iota(\tilde{x}_0,\tilde{x}_1) + \textstyle\int_M x_0 x_1,
\end{equation}
where the $x_j \in C^*(M)$ are pushdowns of $\tilde{x}_j$. Hence,
\begin{equation}
\begin{aligned}
\iota((q-1)^2 \tilde{x}_0,\tilde{x}_1) & = \iota(q^2\tilde{x}_0,\tilde{x}_1) - 2 \iota(q \tilde{x}_0,\tilde{x}_1) + \iota(\tilde{x}_0,\tilde{x}_1) \\ & \textstyle
= \iota(\tilde{x}_0,\tilde{x}_1) + 2\int_M x_0 x_1 - 2\iota(\tilde{x}_0,\tilde{x}_1) - 2\int_M x_0x_1 + \iota(\tilde{x}_0,\tilde{x}_1) = 0.
\end{aligned}
\end{equation}
The same holds on the other side, allowing one to descend to the quotient \eqref{eq:trivial-c-tilde}. Denote the induced pairing on cohomology by $I: \tilde{H}^* \otimes \tilde{H}^{n-*} \rightarrow \bK$. Given a submanifold $L \subset M$ as before, writing $\beta|L = d\gamma_L$ yields a lift of $L$ to $\tilde{M}$, hence defines a class $\lbr\tilde{L}\rbr \in \tilde{H}^*$. Standard topological arguments show that the pairing $I$ applied to those classes recovers \eqref{eq:bullet-trivial}, in parallel with Theorem \ref{th:cardy2}.

\subsection{Algebraic geometry\label{subsec:alg-geom}}
There is another and somewhat closer analogy, coming from mirror symmetry, in which the counterpart of Hamiltonian Floer cohomology is the Hochschild (co)homology of an algebraic variety (more precisely, the counterpart of $B$ lies in Hochschild cohomology, but the space $\tilde{H}^*$ is constructed using Hochschild homology).

Let $M$ be a smooth quasi-projective variety of dimension $n$ over $\bC$, which comes with a vector field $B \in \Gamma(M,TM)$. Define a sheaf of commutative dg algebras 
\begin{equation} \label{eq:tilde-c-dga}
\begin{aligned}
& \tilde{\scrC}^* = \Omega^{-*}[t]/t^2 = \Omega^{-*} \oplus t\, \Omega^{-*}, \\
& d_{\tilde{\scrC}} = -t\,i_B.
\end{aligned}
\end{equation}
In words, these are algebraic differential forms with the grading reversed, with an added copy indexed by a formal variable $t$ of degree $0$; and the differential is $-t$ times the contraction with the vector field $B$. The hypercohomology $\tilde{H}^*$ of \eqref{eq:tilde-c-dga} sits in a long exact sequence
\begin{equation}
\cdots \rightarrow \bigoplus_{p-q = *} H^p(M,\Omega^q) \longrightarrow \tilde{H}^* \longrightarrow \bigoplus_{p-q = *} H^p(M,\Omega^q) \xrightarrow{-i_B} \cdots
\end{equation}
Let $\scrE$ be a coherent sheaf on $M$. Suppose that this can be equipped with a partial connection, which allows one to differentiate sections in $B$-direction. Formally, such a partial connection is given by a sheaf homomorphism $\partial_\scrE: \scrE \rightarrow \scrE$ such that $\partial_\scrE(f\phi) = f \partial_\scrE(\phi) + (B.f) \phi$ for any function $f$ and section $\phi$. Let $\scrJ^1(\scrE)$ be the one-jet sheaf. One can think of the partial connection equivalently as an $\scrO$-module map $\scrJ^1(\scrE) \rightarrow \scrE$ which fits into a commutative diagram
\begin{equation} \label{eq:equivariant-jet}
\xymatrix{
0 \ar[rr] && \Omega^1 \otimes \scrE \ar[rr] \ar[d]^-{i_B \otimes \mathit{id}_\scrE} && \scrJ^1(\scrE) \ar[d]^-{\text{partial connection}} \ar[rrr]^-{\text{projection}} &&& \scrE \ar[d] \ar[r] & 0 \\
0 \ar[rr] && \scrE \ar[rr]^-{\mathit{id}_\scrE} && \scrE \ar[rrr] &&& 0 \ar[r] & 0.
} 
\end{equation}
We denote by $\tilde{\scrE}$ the datum consisting of $\scrE$ together with $\partial_\scrE$. The boundary map associated to the diagram \eqref{eq:equivariant-jet} is a class in
\begin{equation} \label{eq:equivariant-atiyah}
\mathit{At}(\tilde{\scrE}) \in 
\mathit{Ext}^1(\scrE, \{\Omega^1 \xrightarrow{i_B} \scrO\} \otimes \scrE) \iso
H^0(M,\{\Omega^1 \xrightarrow{i_B} \scrO\} \otimes \scrH\!\mathit{om}^*(\scrE,\scrE)),
\end{equation}
where the complex $\{\Omega^1 \rightarrow \scrO\}$ is placed in degrees $\{-1,0\}$, and $\scrH\!\mathit{om}^*(\scrE,\scrE)$ are local derived endomorphisms (an object of the bounded derived category; it agrees with the usual endomorphism sheaf if $\scrE$ is locally free). One can map \eqref{eq:equivariant-atiyah} to $H^0(M,\tilde{\scrC}^* \otimes \scrH\!\mathit{om}^*(\scrE,\scrE))$, take the exponential with respect to the dga structure and then the trace $\scrH\!\mathit{om}^*(\scrE,\scrE) \rightarrow \scrO$. The outcome is a class
\begin{equation} \label{eq:equivariant-chern}
\mathit{Ch}(\tilde{\scrE}) \in H^0(M,\tilde{\scrC}^*).
\end{equation}
By looking at the top line of \eqref{eq:equivariant-jet}, one sees that the image of \eqref{eq:equivariant-atiyah} in $H^1(M,\Omega^1 \otimes \scrH\!\mathit{om}^*(\scrE,\scrE)) \iso \mathit{Ext}^1(\scrE, \Omega^1 \otimes \scrE)$ is the Atiyah class \cite{atiyah57}. Hence, the $t = 0$ truncation of \eqref{eq:equivariant-chern} is the Atiyah-Chern character of $\scrE$. This motivates our choice of notation. 

Given two coherent sheaves $\tilde\scrE_0,\tilde\scrE_1$ with partial connections, we get an induced partial connection on $\scrH\!\mathit{om}^*(\scrE_0,\scrE_1)$. That (considered as an endomorphism of the sheaf) induces an endomorphism of the hypercohomology 
\begin{equation} \label{eq:hyper}
H^*(M,\scrH\!\mathit{om}^*(\scrE_0,\scrE_1)) \iso \mathit{Ext}^*(\scrE_0,\scrE_1),
\end{equation}
which we denote by $\Phi_{\tilde{\scrE}_0,\tilde{\scrE}_1}$. We remind the reader that $\scrH\!\mathit{om}^*(\scrE_0,\scrE_1)$ is the derived version of the local $\scrH\!om$, hence an object of the bounded derived category. This ensures that \eqref{eq:hyper} is true even though $\scrE_0$ is not usually locally free; on the other hand, it means that when talking about a partial connection on $\scrH\!\mathit{om}^*(\scrE_0,\scrE_1)$, we are implicitly using a generalization of the previously introduced definition. Suppose now that $\scrE_0,\scrE_1$ have compact support, so that \eqref{eq:hyper} is of finite total dimension, and define
\begin{equation} \label{eq:alg-phi}
\tilde{\scrE}_0 \bullet \tilde{\scrE}_1 = \mathrm{Str}(\Phi_{\tilde{\scrE}_0,\tilde{\scrE}_1}) \in \bC.
\end{equation}

\begin{example} \label{th:gm-action}
As an important special case, one can consider a vector field $B$ which generates an action of the multiplicative group $\bC^*$ on $M$. A $\bC^*$-equivariant coherent sheaf has a canonical partial connection (in fact, partial connections are the infinitesimal analogue of equivariance, which is the origin of the terminology we have carried over to the main body of the paper). Given two $\bC^*$-equivariant sheaves, $\Phi_{\tilde{\scrE}_0,\tilde{\scrE}_1}$ is the infinitesimal generator of the induced $\bC^*$-action on $\mathit{Ext}^*(\scrE_0,\scrE_1)$. The equivariant Mukai pairing is
\begin{equation}
\sum_\sigma q^\sigma \chi(\mathit{Ext}^*(\scrE_0,\scrE_1)_\sigma) \in \bC[q,q^{-1}].
\end{equation}
where $\mathit{Ext}^*(\scrE_0,\scrE_1)_\sigma$ is the part on which the $\bC^*$-action has weight $\sigma$, and as usual, $\chi$ is the Euler characteristic.  Then \eqref{eq:alg-phi} is the derivative at $q = 1$.
\end{example}

\begin{conjecture}
Suppose that $M$ is projective. Then there is a nondegenerate pairing on $\tilde{H}^*$, which when applied to the classes \eqref{eq:equivariant-chern} recovers \eqref{eq:alg-phi}.
\end{conjecture}

This is a plausible analogue of Theorem \ref{th:cardy2}. In the case $B = 0$, it can be reduced to the Cardy condition from \cite{ramadoss09} (see also \cite[Theorem 16]{caldararu-willerton10} for a proof in a more abstract context). It is possible that similar methods would lead to a proof in general, but that is somewhat outside the scope of this paper.

The analogue of the dilation condition is to suppose that $M$ is Calabi-Yau, which means that it comes with a complex volume form $\eta$, and that our vector field $B$ satisfies $L_B\eta = -\eta$ (this is only possible if $M$ is noncompact). Recall that for compactly supported sheaves, we have a canonical nondegenerate Serre duality pairing
\begin{equation} \label{eq:serre}
\mathit{Ext}^{n-*}(\scrE_0,\scrE_1 \otimes \Omega^n) \otimes \mathit{Ext}^*(\scrE_1,\scrE_0)  \longrightarrow \bK.
\end{equation}
If we denote by $\tilde{\Omega}^n$ the sheaf $\Omega^n$ equipped with the partial connection given by the Lie derivative $L_B$, then \eqref{eq:serre} implies that
\begin{equation}
\tilde{\scrE}_0 \bullet (\tilde{\scrE}_1 \otimes \tilde\Omega^n) = (-1)^{n+1} \tilde{\scrE}_1 \bullet \tilde{\scrE}_0.
\end{equation}
On the other hand, using the trivialization of $\Omega^n$ given by $\eta$, one sees that
\begin{equation}
\label{eq:shift-weight} \tilde{\scrE}_0 \bullet (\tilde{\scrE}_1 \otimes \tilde\Omega^n) = \tilde{\scrE}_0 \bullet \tilde{\scrE}_1 - \chi(\mathit{Ext}^*(\scrE_0,\scrE_1)).
\end{equation}
Comparing the two expressions yields the counterpart of Corollary \ref{th:new-main-properties}(i) for \eqref{eq:alg-phi}:
\begin{equation}
\tilde{\scrE}_1 \bullet \tilde{\scrE}_0 = (-1)^{n+1} \tilde{\scrE}_0 \bullet \tilde{\scrE}_1 + (-1)^n\chi(\mathit{Ext}^*(\scrE_0,\scrE_1)).
\end{equation}
%
%
%

%

\end{document}